\documentclass[12pt]{amsart}

\usepackage{amsmath,amssymb,amsxtra,amsthm,hyperref}

\DeclareMathOperator*\uplim{\overline{lim}}
\renewcommand{\limsup}{\uplim} 

\usepackage{multicol}
\usepackage{calc}
\usepackage{graphicx,tikz}
    \usetikzlibrary{calc,decorations.markings,arrows.meta}

\usepackage{mathtools}
    \DeclarePairedDelimiter\bra{\langle}{\rvert}
    \DeclarePairedDelimiter\ket{\lvert}{\rangle}
    \DeclarePairedDelimiterX\braket[2]{\langle}{\rangle}{#1\,\delimsize\vert\,\mathopen{}#2}

\usepackage[
    margin=2cm
    ]{geometry}

\usepackage{enumitem} 
    \makeatletter
    			\def\namedlabel#1#2{\begingroup#2%
    			\def\@currentlabel{#2}%
    			\phantomsection\label{#1}\endgroup
    		}
    		\makeatother
      
\usepackage{tensor} 
\usepackage{tikz-cd} 
\usepackage{upref}  
\usepackage{mathrsfs}
\usepackage{MnSymbol} 

\theoremstyle{plain}
    \newtheorem{theorem}{Theorem}[section]
    \newtheorem{corollary}[theorem]{Corollary}
    \newtheorem{proposition}[theorem]{Proposition} 
    \newtheorem{lemma}[theorem]{Lemma}

\theoremstyle{definition}
    \newtheorem{definition}[theorem]{Definition}
    \newtheorem{example}[theorem]{Example}
    \newtheorem{remark}[theorem]{Remark}
    
\numberwithin{equation}{section}

\newcommand{\etale}{\'{e}tale}
    \newcommand{\Etale}{\expandafter\MakeUppercase\expandafter{\etale}}    
\newcommand{\usc}{up\-per se\-mi-con\-ti\-nuous}
\newcommand{\uscBb}{\usc\ Ba\-nach bun\-dle}
\newcommand{\ib}{im\-prim\-i\-tiv\-ity bi\-mod\-u\-le}
\newcommand{\LCH}{lo\-cal\-ly com\-pact Haus\-dorff}
\newcommand{\igpd}{im\-prim\-i\-tiv\-ity groupoid}
\newcommand{\iFbdl}{im\-prim\-i\-tiv\-ity Fell bundle}
\newcommand{\SMEadj}{strongly Morita equivalent}
\newcommand{\SMEnoun}{strong Morita equivalence}
\newcommand{\psp}[2][]{principal #1#2-space} 
\newcommand{\demiequiv}[1][]{#1demi-equivalence} 

\newcommand{\cA}{\mathscr{A}}
\newcommand{\cB}{\mathscr{B}}
\newcommand{\cC}{\mathscr{C}}
\newcommand{\cE}{\mathscr{E}}
\newcommand{\cM}{\mathscr{M}}
\newcommand{\cN}{\mathscr{N}}
\newcommand{\cK}{\mathscr{K}}
\newcommand{\cV}{\mathscr{V}}

\newcommand{\cG}{\mathcal{G}}
\newcommand{\cH}{\mathcal{H}}
\newcommand{\cX}{\mathcal{X}}

\newcommand{\forwards}{\mathtt{fir}}
\newcommand{\backwards}{\mathtt{sec}}
\newcommand{\flip}{\mathtt{flip}}
\newcommand{\Flip}{\mathtt{Flip}}

\newcommand{\z}{^{\scalebox{.7}{$\scriptstyle (0)$}}} 
\newcommand{\inv}{^{\scalebox{.7}{$\scriptstyle -1$}}}
\newcommand{\comp}{^{\scalebox{.7}{$\scriptstyle (2)$}}}
\newcommand{\op}{^{\mathrm{op}}}
\newcommand{\cst}{\mathrm{C}^{*}}
\newcommand{\FBRel}{\mathbin{\mathtt{R}}} 
\DeclareMathOperator{\compacts}{\mathbb{K}}

\let\mbb\mathbb
\let\mf\mathfrak
\let\mbf\mathbf

\newcommand*\diff{\mathop{}\!\mathrm{d}}

\let\ipscriptstyle=\scriptscriptstyle
\def\lipsqueeze{{\mskip -2.0mu}}
\def\ripsqueeze{{\mskip -2.0mu}}
\def\ipcomma{\nobreak \mid\nobreak} 
\newbox\ipstrutbox
\setbox\ipstrutbox=\hbox{\vrule
    height5.5pt
    depth 2pt
    width 0pt
}

\newcommand{\norm}[1]{\left\| #1 \right\|}

\newcommand{\inner}[2]{\langle #1 \ipcomma #2 \rangle}

\newcommand{\linner}[4][]{
    {
    \tensor*[_{\ipscriptstyle
        #2
        }^{\hfill\ipscriptstyle
        #1}]{\langle #3 \ipcomma #4 \rangle}{}
    }
}
    \def\lip#1<#2,#3>{\linner{#1}{#2}{#3}}

\newcommand{\rinner}[4][]{
    {
    \tensor*{\langle #3
        \ipcomma
        #4 \rangle}{
       _{
        \ipscriptstyle
        #2
        }^{\ripsqueeze
        \ipscriptstyle
        #1}}
    }
}
    \def\rip#1<#2,#3>{\rinner{#1}{#2}{#3}}

\newcommand{\innercpct}[4][]{
\tensor*[_{\ipscriptstyle #2}^{\ipscriptstyle #1}]{\ket*{#3}}{}\bra*{#4}
}

\newcommand{\leoq}[4][]{\tensor*[_{\ipscriptstyle #2}^{\ipscriptstyle #1}]{\left\{ #3 \ipcomma #4\op \right\}}{}} 

\newcommand{\leoqempty}[2][]{\tensor*[_{\ipscriptstyle #2}^{\ipscriptstyle #1}]{\left\{ \mvisiblespace \ipcomma \mvisiblespace \right\}}{}}

\newcommand{\reoq}[4][]{\tensor*[]{\left\{ #2\op \ipcomma #3 \right\}}{_{\ipscriptstyle #4}^{\ipscriptstyle #1}}} 
\newcommand{\reoqempty}[2][]{\tensor*[]{\left\{ \mvisiblespace \ipcomma \mvisiblespace \right\}}{_{\ipscriptstyle #2}^{\ipscriptstyle #1}}}

    \newcommand\mvisiblespace[1][.7em]{%
        	\makebox[#1]{%
        		\kern.07em
        		\vrule height.3ex
        		\hrulefill
        		\vrule height.3ex
        		\kern.07em
        	}
        }

\newcommand{\bfp}[2]{\lipsqueeze\tensor*[_{\ipscriptstyle #1}]{\ast}{_{\ipscriptstyle #2}}\ripsqueeze} 

\newcommand\Item[1][]{%
  \ifx\relax#1\relax  \item \else \item[#1] \fi
  \abovedisplayskip=0pt\abovedisplayshortskip=0pt~\vspace*{-\baselineskip}}

\newcommand{\lact}{
        \mathchoice
              {\displaystyle\mathbin\smalltriangleright}
              {\textstyle\mathbin\smalltriangleright}
              {\scriptstyle\mathbin\smalltriangleright}
              {\scriptscriptstyle\mathbin\smalltriangleright}
        }

\newcommand{\ract}{
        \mathchoice
              {\displaystyle\mathbin\smalltriangleleft}
              {\textstyle\mathbin\smalltriangleleft}
              {\scriptstyle\mathbin\smalltriangleleft}
              {\scriptscriptstyle\mathbin\smalltriangleleft}
        }

        \newcommand{\lactB}{
                    \mathchoice
                          {\mathbin{\raisebox{-.9pt}{$\includegraphics[scale=1.3]{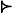}$}}}
                          {\mathbin{\raisebox{-.9pt}{$\includegraphics[scale=1.3]{semithickHleftB.pdf}$}}}
                          {\mathbin{\raisebox{-0.6pt}{$\includegraphics[scale=.8]{semithickHleftB.pdf}$}}}
                          {\mathbin{\raisebox{-.2pt}{$\includegraphics[scale=.6]{semithickHleftB.pdf}$}}}
                    }
                    
        \newcommand\ractB {
                    \mathchoice
                          {\mathbin{\raisebox{-.9pt}{$\includegraphics[scale=1.3,angle=180,origin=c]{semithickHleftB.pdf}$}}}
                          {\mathbin{\raisebox{-.9pt}{$\includegraphics[scale=1.3,angle=180,origin=c]{semithickHleftB.pdf}$}}}
                          {\mathbin{\raisebox{-0.6pt}{$\includegraphics[scale=.8,angle=180,origin=c]{semithickHleftB.pdf}$}}}
                          {\mathbin{\raisebox{-.2pt}{$\includegraphics[scale=.6,angle=180,origin=c]{semithickHleftB.pdf}$}}}
                    }

\usepackage[backend=biber, 
	giveninits=true,
	hyperref=true, 
	maxcitenames=3,
	maxbibnames=100,
	block=none,
	isbn=false]{biblatex}	
 
	\AtEveryBibitem{%
		\clearfield{mrclass}%
		\clearfield{series}%
		\clearfield{mrnumber}%
		\clearfield{mrreviewer}%
		\clearfield{doi}%
		\ifentrytype{online}{}{
			\clearfield{url}
			\clearfield{note}
		}
	}
\title{The \MakeUppercase{\iFbdl}} 

\author{Anna Duwenig}
\address{KU Leuven, Department of Mathematics, Leuven (Belgium)}

\email{anna.duwenig@kuleuven.be}
\date{\today}


\def \fix {\medskip \noindent $\blacktriangleright $\kern 12pt \textbf{For the remainder of the paper}, } 

\def \fixnow {\medskip \noindent $\blacktriangleright $\kern 12pt \textbf{For the moment}, }

\begin{document}

\allowdisplaybreaks

\maketitle

\begin{abstract}
	Given a full right-Hilbert $\cst$-module~$\mathbf{X}$ over a $\cst$-algebra $A$, the set $\compacts_{A}(\mathbf{X})$ of $A$-compact operators on~$\mathbf{X}$ is the (up to isomorphism) unique $\cst$-algebra that is \SMEadj\ to the coefficient algebra $A$ via $\mbf{X}$. As bimodule, $\compacts_{A}(\mathbf{X})$ can also be thought of as the balanced tensor product $\mathbf{X}\otimes_{A} \mathbf{X}\op$, and so the latter naturally becomes a $\cst$-algebra.
	We generalize both of these facts to the world of Fell bundles over groupoids:
	Suppose~$\cB$ is a  Fell bundle over a groupoid~$\cH$ and~$\cM$ an \uscBb\ over a \psp{$\cH$}~$X$. If~$\cM$ carries a right-action of~$\cB$ and a sufficiently nice~$\cB$-valued inner product, then its {\em \iFbdl} $\compacts_{\cB}(\cM)=\cM\otimes_{\cB} \cM\op$ is a Fell bundle over the \igpd\ of $X$, and it is the unique Fell bundle that is equivalent to~$\cB$ via~$\cM$. We show that $\compacts_{\cB}(\cM)$ generalizes the `higher order' compact operators of Abadie--Ferraro in the case of saturated bundles over groups, and that the theorem recovers results such as Kumjian's Stabilization trick.
\end{abstract}

\section{Introduction}

Suppose a groupoid~$\cH$ acts on a topological space $X$, meaning that we have a continuous open surjection  $\sigma\colon X\to \cH\z$ (the {\em anchor map} of the action) and a continuous map
\begin{equation}\label{eq:def bfp}
X\bfp{\sigma}{r}\cH
\coloneqq 
\{(x,h)\in X\times \cH : \sigma(x)=r_{\cH}(h)\}
\to X
,\quad
(x,h)\mapsto x \ract h.
\end{equation}
Let us further assume that~$X$ and~$\cH$ are \LCH, and that the action is free and proper (i.e.,~$X$ is a {\em \psp[(right) ]{$\cH$}}; \cite[p.\ 6]{MRW:Grpd}).
Out of  $X$, we can build a {\em left}~$\cH$-space $X\op$ as follows: as a topological space, it is just $X$, but we write its elements with a superscript-op to avoid confusion. Its left action uses the anchor map $\sigma\op\colon x\op \mapsto \sigma(x)$, and the action is given  by $h\lact x\op = (x\ract h\inv)\op$ for any $h$ with $s_{\cH}(h)=\sigma\op(x\op)$. Whenever $\sigma\op$ appears in a subscript, we will drop its superscript and simply write $\sigma$.
Now consider 
\[
    X\bfp{\sigma}{\sigma}X\op = \{(x,y\op)\in X\times X\op : \sigma(x)=\sigma\op(y\op)\}
    .
\]
With the subspace topology of the product topology, $X\bfp{\sigma}{\sigma} X\op$ is \LCH.
For any given $(x,y\op)\in X\bfp{\sigma}{\sigma}X\op$, consider the orbit 
\[
    [x,y\op]
    \coloneqq
    \{
        (x\ract h, h\inv \lact y\op)
        :
        h\in \sigma(x)\cH
    \}
\]
of $(x,y\op)$ under the `diagonal'~$\cH$-action on $X\bfp{\sigma}{\sigma}X\op$. Since 
$\cH$ is assumed to have open source map
and to
act properly on $X$,
it follows that the quotient $X\times_{\cH} X\op$ by this equivalence relation is also a \LCH\ space \cite[Proposition 2.18]{Wil2019}, and likewise is the quotient $X/\cH$.

We equip $X\times_{\cH} X\op$ with the structure of a groupoid with unit space $X/\cH$: its source and range maps are given by
\[
    s[x,y\op] =  [y, y\op] = r[y, z\op],
\]
and we can identify $[y,y\op]$ with $y\ract\cH\in X/\cH$. 
Given $\xi,\eta\in X\times_{\cH} X\op$ with $s(\xi)=r(\eta)$, we can find representatives $(x,y\op)$ and $(y, z\op)$ of $\xi$ respectively $\eta$ with matching component `in the middle', and then define their product by $\xi\eta = [x,z\op]$.\label{p:gpd structure on cG}

The groupoid $X\times_{\cH} X\op$ acts  on the left of $X$: we define the anchor map 
\begin{equation}\label{eq:anchor map of left G action on X}
\rho\colon X\to (X\times_{\cH} X\op)\z \cong X/\cH \quad \text{ by }
\quad\rho(x)=[x,x\op]
\triangleeq x\ract \cH.
\end{equation}
 Note that $\rho$ is 
 open \cite[Proposition 2.12.]{Wil2019} and continuous. If $(\xi,y)\in (X\times_{\cH} X\op)\bfp{s}{\rho}X$, then $s(\xi)=\rho(y)$ implies that there exists $x\in X$ such that $\xi=[x,y\op]$; this $x$ is unique by freeness of the~$\cH$-action on $X$. We can therefore define $\xi\lact y = x$, i.e.,
\begin{equation}\label{eq:left G action on X}
    X\times_{\cH} X\op\curvearrowright X\colon \qquad
    [x,y\op]\lact y = x.
\end{equation}

\begin{remark}\label{rmk:leoq}
    Suppose~$X$ is a $(\cG,\cH)$-groupoid equivalence (see \cite{MRW:Grpd}) with anchor maps $\sigma\colon X\to \cH\z$ and $\rho\colon X\to \cG\z$. 
    If $\sigma(x)=\sigma(y)$, then there exists a unique element $\leoq[X]{\cG}{x}{y}$ in~$\cG$ such that $x=\leoq[X]{\cG}{x}{y} \lact y$. Indeed, since $\sigma$ is assumed to induce a homeomorphism $\tilde{\sigma}\colon \cG\backslash X \to \cH\z$, the equality $\sigma(x)=\sigma(y)$ means exactly that $\cG\lact x=\cG\lact y$, and so freeness of the~$\cG$-action implies the existence of the unique $\leoq[X]{\cG}{x}{y}$ which transforms $y$ into $x$. Since the action is proper and $\rho$ is open, the surjective map $\leoqempty[X]{\cG}\colon X\bfp{\sigma}{\sigma}X\op\to \cG$ is continuous and open, and it factors through the quotient $X\times_{\cH} X$ of $X\bfp{\sigma}{\sigma}X\op$ by the diagonal~$\cH$-action, yielding a homeomorphism $X\times_{\cH} X\op\cong\cG$. 

    Likewise, if $\rho(x)=\rho(y)$, we write $\reoq[X]{x}{y}{\cH}$ for the unique element of~$\cH$ such that $x\ract \reoq[X]{x}{y}{\cH} =y$, and $\reoqempty[X]{\cH}\colon X\op\bfp{\rho}{\rho}X\to \cH$ induces a homeomorphism $X\op\times_{\cG} X\cong\cH$. A quick computation shows that the following equalities hold (wherever one side of the equation makes sense):
    \begin{align}\label{eq:leoq into reoq}
        \leoq[X]{\cG}{x}{y}\lact z
            &=
            x\ract \reoq[X]{y}{z}{\cH}
        \\
        \label{eq:leoq and reoq and action}
        \begin{split}
        \leoq[X]{\cG}{x}{(y\ract h\inv) }
        &=
        \leoq[X]{\cG}{x\ract h}{y}
        \\
        \reoq[X]{(g\inv \lact y)}{z}{\cH}
        &=
        \reoq[X]{y}{g \lact z}{\cH}
        \end{split}
    \end{align}
\end{remark}

\begin{example}\label{ex:X=H}
	Any groupoid~$\cH$ acts freely and properly on $X=\cH$; say, on the right. The anchor map is then given by $\sigma=s_{\cH}\colon X\to \cH\z$, and the associated \igpd\ $\cG=\cH\times_\cH \cH\op$ is isomorphic to~$\cH$ via $f\colon [h_{1},h_{2}\op]\mapsto h_{1}h_{2}\inv$. In particular,  the map $\leoqempty[X]{\cG}$ becomes the map $X\bfp{s}{s}X\op\to \cH, (h_{1},h_{2})\mapsto h_{1} h_{2}\inv$. This isomorphism further turns the anchor map $\rho\colon X\to \cG\z, h\mapsto [h,h\op],$ into the range map $r_{\cH}\colon X\to \cH\z$, and $\reoqempty[X]{\cH}\colon X\op\bfp{r}{r}X\to \cH$ is therefore given by $\reoq[X]{h_{1}}{h_{2}}{\cH} = h_{1}\inv h_{2}$. 
\end{example}

We arrive at a well-known result:
\begin{lemma}[motivation; {\cite[Lemma 2.45, Proposition 2.47]{Wil2019}}]\label{lem:thm for gpds}
     Suppose~$X$ is a \psp{$\cH$}. Then $X\times_{\cH} X\op$ is a \LCH\ groupoid with open source map that acts freely and properly on the left of~$X$ with anchor map $\rho\colon x\mapsto [x,x\op]$. With this structure,~$X$ is a $( X\times_{\cH} X\op,\cH)$-equivalence of groupoids. Moreover, if~$X$ is also a $(\cG,\cH)$-equivalence, then there exists a groupoid isomorphism $X\times_{\cH} X\op \to \cG$ that is uniquely determined by $[x,y\op]\mapsto \leoq[X]{\cG}{x}{y}$. 
\end{lemma}
Note that the above in particular states that, in the setting where $\cG=X\times_{\cH} X\op$, the element
\(
\leoq[X]{\cG}{x}{y}
\) of~$\cG$  for $x,y\in Xu$
is 
\(
[x,y\op]
\).

A reason why one might be interested in groupoid equivalences is the following result due to works by \citeauthor{MRW:1987:Equivalence}.
\begin{lemma}[{\cite{Wil:Haar}, \cite{MRW:1987:Equivalence}}]
    Suppose $\cG,\cH$ are \LCH\ groupoids, that $\{\lambda_u\}_{u\in \cH\z}$ is a Haar system on~$\cH$, and that there exists an equivalence~$X$ between~$\cG$ and~$\cH$ (with open anchor maps). Then~$\cG$ also allows a Haar system, and if say $\{\mu_v\}_{v\in \cG\z}$ is any such Haar system, then the groupoid $\cst$-algebras $\cst_{r}(\cG,\mu)$ and $\cst_{r}(\cH,\lambda)$ are \SMEadj\ via an \ib\ built as a completion of $C_c(X)$.
\end{lemma}
An analogous  theorem holds for the $\cst$-algebra of equivalent Fell bundles \cite{MW:2008:Disintegration}. Such results are powerful, for example because two \SMEadj\ $\cst$-algebras have the same representation theory, $\operatorname{K}$-theory, and lattices of ideals (via the so-called ``Rieffel Correspondence''). But not only do the above mentioned results state the existence of a \SMEnoun, they also {\em construct it explicitly}, which allows one to, for example, construct representatives of certain famous classes in $\operatorname{KK}$-theory \cite{DE:TransKron}.

\medskip

There is a statement analogous to Lemma~\ref{lem:thm for gpds} for a right Hilbert $\cst$-module~$\mathbf{X}$ over $\cst$-algebra $A$. In the following, we use the symbol $\innercpct[\mathbf{X}]{A}{\mathbf{x}}{\mathbf{y}}$ for $\mathbf{x,y}\in\mathbf{X}$  to denote the `$A$-rank-one' operator $\mathbf{X}\to\mathbf{X}$ that maps $\mathbf{z}$ to $\mathbf{x}\cdot \rinner[\mathbf{X}]{A}{\mathbf{y}}{\mathbf{z}}\in \mathbf{X}$, and we let $\compacts(\mathbf{X}_A)=\compacts_{A}(\mathbf{X})$ denote the $\cst$-algebra generated by these operators; it is an ideal of the $A$-adjointable operators (see \cite[Lemma 2.25]{RaWi:Morita}). We get that~$\mathbf{X}$ has a {\em left} $\compacts_{A}(\mathbf{X})$-inner product given by\footnote{Hopefully, it is not too confusing that the symbol on the right-hand side of Equation~\eqref{eq:ket-bra subscripts} carries $A$ in the subscript rather than the cumbersome $\compacts_{A}(\mathbf{X})$; the ket--bra notation in place of the bra--ket should be a clear indicator as to where the inner product takes values.}
\begin{equation}\label{eq:ket-bra subscripts}
    \linner[\mathbf{X}]{\compacts_{A}(\mathbf{X})}{\mathbf{x}}{\mathbf{y}}
    \coloneqq 
    \innercpct[\mathbf{X}]{A}{\mathbf{x}}{\mathbf{y}}.
\end{equation}
With inner products such as these, we will drop the sub- and/or superscripts whenever there is no ambiguity. Furthermore, we let $\mathbf{X}\op$ be the dual, left-Hilbert $\cst$-module as defined in \cite[p.\ 49]{RW:Morita}, meaning there exists an additive bijection $F \colon \mathbf{X}\to \mathbf{X}\op $ such that $F (\lambda \mathbf{x}) = \bar{\lambda}F (\mathbf{x})$ for all $\mathbf{x}\in \mathbf{X}$ and $\lambda\in \mbb{C}$. 
The analogue can now be stated as:
\begin{lemma}[motivation; {\cite[Proposition 2.21, Lemma 2.25, Proposition 3.8.]{RW:Morita}}]\label{RW:Morita:Prop3.8}
	Suppose~$\mathbf{X}$ is a full right-Hilbert $\cst$-module over a $\cst$-algebra~$A$.
	Then the set $\compacts_{A}(\mathbf{X})$ of $A$-compact operators  on~$\mathbf{X}$ is a $\cst$-algebra with respect to the operator norm, and~$\mathbf{X}$ is a $(\compacts_{A}(\mathbf{X}), A)$-\ib. Moreover, if~$\mathbf{X}$ is also a $(B,A)$-\ib, then there exists a *-isomorphism $\compacts_{A}(\mathbf{X}) \to B$ that is uniquely determined by $\innercpct[\mathbf{X}]{A}{\mathbf{x}}{\mathbf{y}} \mapsto \linner[\mathbf{X}]{B}{\mathbf{x}}{\mathbf{y}}$.
\end{lemma}

Let us elaborate on the analogy to Lemma~\ref{lem:thm for gpds}: 
If  $\mathbf{Y}$ is another full right-Hilbert $\cst$-module over $A$, then the map 
$\mathbf{X}\otimes_{A}\mathbf{Y}\op\to \compacts_{A}(\mathbf{Y},\mathbf{X})$ determined by $\mathbf{x}\otimes \mathbf{y}\op \mapsto \innercpct[\mathbf{X}]{A}{\mathbf{x}}{\mathbf{y}}$  is an isomorphism of bi-Hilbert $\compacts_{A}(\mathbf{X})-\compacts_{A}(\mathbf{Y})$-bimodules (see Lemma~\ref{lem:tensor and compacts}). In particular, $\mathbf{X}\otimes_{A}\mathbf{X}\op \cong \compacts_{A}(\mathbf{X})$ is, in fact, the (up to isomorphism) unique $\cst$-algebra that is equivalent to $A$ via $\mathbf{X}$. (Note that we similarly have $\mathbf{X}\op\otimes_{\compacts}\mathbf{X}\cong A$ as bi-Hilbert $A-A$-modules via $\mathbf{x}_{1}\op\otimes \mathbf{x}_{2}\mapsto \rinner[\mathbf{X}]{A}{\mathbf{x}_{1}}{\mathbf{x}_{2}}$.) 

\smallskip

The main result of this paper, Theorem~\ref{thm:demiequiv to equiv}, is  an analogue of Lemma~\ref{lem:thm for gpds} and Proposition~\ref{RW:Morita:Prop3.8} in the setting of Fell bundles over groupoids. To this end, we will first show that the correct analogue of a \psp{$\cH$} and a full right-Hilbert $\cst$-module over $A$ is a {\em \demiequiv[$\cB$-]}, which can be thought of as ``half'' of a Fell bundle equivalence in the sense of \cite[Def. 6.1]{MW:2008:Disintegration}. We will then prove:
\begin{theorem}\label{thm:demiequiv to equiv}
    Suppose~$\cH$ is a \LCH\  groupoid with open source map and~$X$ is a \psp{$\cH$} with open anchor map. If~$\cB$ is a saturated Fell bundle over~$\cH$ and~$\cM$  is a~\demiequiv[$\cB$-] over $X$, then there is a saturated Fell bundle $\compacts(\cM_{\cB})$ over the \igpd\ of~$X$ that is  equivalent to~$\cB$ via~$\cM$. Moreover, if~$\cM$ is also an $(\cA,\cB)$-Fell bundle equivalence, then there exists a Fell bundle isomorphism $\compacts(\cM_{\cB}) \to \cA$ that is uniquely determined by $\innercpct[\cM]{\cB}{m}{n} \mapsto \linner[\cM]{\cA}{m}{n}$.
\end{theorem}

    Many theorems about symmetric imprimitivity \cite{EKQR:ImprimThms,KMQW:ImprimThms2,aHKRW:2013:Imprim,DuLi:ImprimThms-pp,IW:2011:MoritaFell} and generalized fixed point algebras \cite{Rieffel:1990:Proper,Exel:2000:Morita,Ferraro:2021:FixedPt} have appeared over the years and are extending (in various directions) a \citeyear{Green:1977}-result of \citeauthor{Green:1977}'s
    \cite{Green:1977}.
    In Section~\ref{sec:Applications}, we will see in what way our theorem captures some of those result.

\bigskip

The structure of the paper is as follows.
After establishing notation and assumptions in Subsection~\ref{ssec:assumptions}, we give the definition of a \demiequiv[$\cB$-]~$\cM$ over a \psp{$\cH$}~$X$ (Definition~\ref{def:wordBdl}) and prove some of its basic properties in Section~\ref{sec:wordBdls}. Section~\ref{sec:Existence} is devoted to proving that~$\cM$ gives rise to a Fell bundle $\compacts(\cM_{\cB})$ over the \igpd\ of $X$, which we will also denote by $\cM\otimes_{\cB}\cM\op$ for reasons that will become apparent. We prove that this {\em \iFbdl} is equivalent to~$\cB$ via~$\cM$ in Section~\ref{sec:Equivalence}. In Section~\ref{sec:Uniqueness}, we prove that $\compacts(\cM_{\cB})$ is unique (up to isomorphism), and we see some applications in Section~\ref{sec:Applications}. There are two short appendices to establish some background results about Hilbert $\cst$-modules and about \uscBb s.

\subsection{Assumptions, conventions, and notation}\label{ssec:assumptions}

We will denote groupoids using \verb|\mathcal| ($\cG,\cH,\mathcal{K}\ldots$), bundles using \verb|\mathscr| ($\cA,\cB,\cC,\ldots$), and stand-alone Hilbert $\cst$-modules and their elements using \verb|\mathbf| ($\mathbf{x\in X, y\in Y, z\in Z}\ldots$). Fibred products are denoted by $\mvisiblespace\bfp{f}{g}\mvisiblespace$, as defined in \eqref{eq:def bfp}. The algebraic tensor product is $\odot$ and its completions is $\otimes$; we add a subscript for the internal (i.e., balanced) tensor product.  Right and left actions on spaces are usually denoted by $\mvisiblespace\ract\mvisiblespace$ and $\mvisiblespace\lact\mvisiblespace$, respectively, while actions on bundles are denoted by $\mvisiblespace\ractB\mvisiblespace$ and $\mvisiblespace\lactB\mvisiblespace$. (The triangle is always pointing at the object that is being acted on.)

We try to use the letter $p$ for the projection map of Fell bundles, while those of  general Banach~bundles are denoted by $q$. The fibre over a point $h$ of the base space of a bundle $\cB=(B\to \cH)$ with total space $B$ is denoted $B(h)$; we adopt analogous notation for other bundles. We write $s_{\cB}\colon B\to \cH\z$ for the source map of~$\cH$ composed with the projection map of $\cB$, and we adopt analogous notation for anchor maps $\sigma\colon X\to \cH\z$.

We use the notion of \uscBb s, Fell bundles, actions of Fell bundles on \uscBb s, and equivalences of Fell bundles as in \cite{DL:MJM2023}, Definitions 2.3, 2.9, 2.10, and 2.11, respectively. 
As done in \cite{MW:2008:Disintegration}, our Fell bundles   are assumed \textbf{saturated} in the sense that their fibres are full (and hence imprimitivity) bimodules; we will rarely mention this assumption again.  We frequently use the fact that our \uscBb s have enough continuous cross sections (see \cite[Cor.\ 2.10]{LAZAR2018448}, \cite[App.\ A]{MW:2008:Disintegration}).

\fix we assume that
\begin{itemize}
	\item~$\cH$ is a \LCH\ groupoid with open source map\footnote{Any \LCH\ groupoid with a Haar system (and thus any \etale\ groupoid) is an example of a groupoid with open source map \cite[Proposition 1.23]{Wil2019}.} $s_{\cH}\colon \cH\to \cH\z$,
	\item~$X$ is a \LCH\ space and $\sigma\colon X\to \cH\z$ is a continuous open surjection,
	\item~$X$ is a \psp[right ]{$\cH$} with anchor map $\sigma$, and
	\item the groupoid $X\times_{\cH} X\op$ will be denoted by~$\cG$, and we will refer to it as the {\em \igpd} of~$X$ as in \cite[Lemma 2.45]{Wil2019}.
\end{itemize}

In analogy to the notation $\cH u = s_{\cH}\inv (u)$ and $u\cH = r_{\cH}\inv (u)$ for $u\in \cH\z$, we will write $X u \coloneqq \sigma\inv(u)$. We chose the letter $\sigma$ since it serves the purpose of a `source' map; for left-actions, we will therefore generally use $\rho$ (as in `range') for the anchor map.

\section{Demi-Equivalences}\label{sec:wordBdls}

For the upcoming definition, we remind the reader that the anchor map $\rho\colon X\to \cG\z$ of the left action of $\cG= X\times_\cH X\op$ on the \psp{$\cH$} $X$ is exactly the quotient map if we identify $\cG\z$ with $X/\cH$; see \eqref{eq:anchor map of left G action on X}. 
\begin{definition}[{cf.\ \cite[Definition 2.1]{AF:EquivFb}}]\label{def:wordBdl}
    Suppose $\cB=(p_{\cB}\colon B\to \cH)$ is a Fell bundle over the groupoid~$\cH$, and  $\cM=(q_{\cM}\colon M\to X)$ is an \uscBb\ over the \psp{$\cH$} $X$. We call~$\cM$ a {\em (right)~\demiequiv[$\cB$-]}\footnote{While it would have been nice to call~$\cM$ a ``principal~$\cB$-bundle'' to show the analogy to the groupoid-world concept of a \psp{$\cH$}, the clash with the existing notions of principal group bundles in differential geometry made that a non-viable option.} if there exist maps
    \[\begin{tikzcd}[row sep=tiny, column sep = small]
      \mvisiblespace \ractB \mvisiblespace\colon&[-15pt]  M\bfp{\sigma}{r} B \ar[r]& M,&
            &(m,b)\ar[r,mapsto]& m\ractB b,
      \end{tikzcd}
      \]
    and
      \[
          \begin{tikzcd}[row sep=tiny, column sep = small]
          \rip\cB<\mvisiblespace,\mvisiblespace>\colon&[-15pt] M \bfp{\rho}{\rho} M \ar[r]& B, &
            &(m_{1},m_{2})\ar[r,mapsto]& \rip\cB<m_{1},m_{2}>,
          \end{tikzcd}
      \]
    such that the following\footnote{ In the ensuing list of properties, the DE in ``(DEn)'' stands for ``\demiequiv''.} hold for all appropriately chosen $m_{i}\in M$ and $b\in B$.
    \begin{enumerate}[leftmargin=1.5cm,label=\textup{(DE\arabic*)}]
  \item\label{item:rwordBdl:ractB:fibre} $\mvisiblespace \ractB \mvisiblespace $ covers the map $\mvisiblespace \ract \mvisiblespace $ in the sense that $q_{\cM}(m\ractB b) = q_{\cM}(m)\ract p_{\cB}(b)$;
  \item\label{item:rwordBdl:ip:fibre} $\rip\cB<\mvisiblespace,\mvisiblespace>$ covers the map $\reoqempty[X]{\cH}$ in the sense that 
        $q_{\cM}(m_{1})\ract p_{\cB}\bigl(\rip\cB<m_{1},m_{2}>\bigr)=q_{\cM}(m_{2})$;
  \item\label{item:rwordBdl:ip} $\rip\cB<\mvisiblespace,\mvisiblespace>$ is continuous, and fibrewise sesquilinear (meaning linear in the second and anti-linear in the first coordinate);
   \item\label{item:rwordBdl:ip:C*linear} 
      $\rip\cB<m_{1},m_{2}\ractB b>=\rip\cB<m_{1},m_{2}>b$;
   \item\label{item:rwordBdl:ip:adjoint}
      $\rinner[*]{\cB}{m_{1}}{m_{2}}=\rip\cB<m_{2},m_{1}>$;
   \item\label{item:rwordBdl:positive} $\rip\cB<m,m>\geq 0$ in the $\cst$-algebra $B(\sigma_{\cM}(m))$, and $\rip\cB<m,m>=0$ only if $m=0$;
   \item\label{item:rwordBdl:norm compatible} the norm $m\mapsto \norm{\rip\cB<m,m>}^{1/2}$ agrees with the norm that the \uscBb~$\cM$ carries; and
   \item\label{item:rwordBdl:fibrewise full} for each $x\in X$, the linear span of $\{\rip\cB<m_{1},m_{2}>:m_{i}\in M(x)\}$ is dense in $B(\sigma(x))$.
  \end{enumerate}
\end{definition}

An analogous definition of a left~\demiequiv[$\cB$-] can be made; there, each instance of a range map becomes a source map (and vice versa) and sesquilinearity means that the {\em first} coordinate is linear.
When there is no ambiguity, we will drop the subscript-$\cB$ on the inner product; conversely, when there is ample room for ambiguity, we might add a superscript-$\cM$.

The reader might have noticed that there are a few natural properties, both algebraic and analytic, that a \demiequiv\ should satisfy if it is to be `half' of an equivalence in the sense of \citeauthor{MW:2008:Disintegration}. Let us show that all those properties actually follow automatically:
\begin{lemma}[{cf.\ \cite[Lemma 2.7]{AF:EquivFb}}]\label{lem:rwordBdl are nice} Suppose $\cB=(p_{\cB}\colon B\to \cH)$ is a Fell bundle over the groupoid~$\cH$  and $\cM=(q_{\cM}\colon M\to X)$ is a \demiequiv[right $\cB$-]\   over the \psp[right ]{$\cH$}~$X$. Then we have the following, where $(m_{1},m_{2})\in M\bfp{\rho}{\rho}M$.
          \begin{enumerate}[leftmargin=1.5cm,label=\textup{(DE\arabic*)}, start=9]
            \item\label{item:rwordBdl:M(x) full} Each $M(x)$ is a full right-Hilbert $\cst$-module over the $\cst$-algebra $B(\sigma(x))$;
          \item\label{item:rwordBdl:ip:C*linear on left}
          $         	\rip\cB<m_{1} \ractB b^*,m_{2}> = b\rip\cB<m_{1},m_{2}>$ for all appropriate $b\in B$;
          \item\label{item:rwordBdl:ip:s and r}
          $\sigma_{\cM}(m_{1})=r_{\cB}(\rip\cB<m_{1},m_{2}>)$ and $\sigma_{\cM}(m_{2})=s_{\cB}(\rip\cB<m_{1},m_{2}>)$;
          \item\label{item:rwordBdl:ip:CS} $\rip\cB<m_{1},m_{2}>\rinner[*]{\cB}{m_{1}}{m_{2}}
                    \leq
                    \norm{m_{2}}^2 \rip\cB<m_{1},m_{1}>$
            as elements of the $\cst$-algebra $B(\sigma_{\cM}(m_{1}))$ \textup(Cauchy--Schwarz\textup).
          \end{enumerate}
Moreover,~$\cB$ acts on the right of~$\cM$ in the sense of \cite[Definition 2.11]{DL:MJM2023}, i.e.,
          \begin{enumerate}[leftmargin=1.5cm,label=\textup{(DE\arabic*)}, resume]      
            \item\label{item:rwordBdl:ractB:bilinear} $\mvisiblespace \ractB \mvisiblespace$ is fibrewise bilinear;
            \item\label{item:rwordBdl:ractB:associative} $(m\ractB b)\ractB b' = m\ractB (bb')$ for all appropriate $b,b'\in B$;
            \item\label{item:rwordBdl:ractB:norm}  $\norm{m\ractB b}\leq \norm{m}\norm{b}$;  and
            \item\label{item:rwordBdl:ractB:cts} $\mvisiblespace \ractB \mvisiblespace$ is continuous.
          \end{enumerate}
\end{lemma}

Again, the analogous result for a \demiequiv[left ] must be rephrased in a way that  every range becomes a source etc.
With the properties listed in Lemma~\ref{lem:rwordBdl are nice}, we see that a \demiequiv[$\cB$-] really satisfies all `one-sided' properties of a Fell bundle equivalence as defined in \cite[Def. 6.1]{MW:2008:Disintegration}. In Proposition~\ref{prop:from two-sided demi to equivalence}, we will see which properties are needed for a `two-sided' \demiequiv\ to be an equivalence.

\begin{proof}[Proof of Lemma~\ref{lem:rwordBdl are nice}]
    Condition~\ref{item:rwordBdl:M(x) full} is just a restatement of other properties. To be precise, $M(x)$ is a right inner product $B(\sigma(x))$-module in the sense of \cite[Definition 2.1]{RaWi:Morita} because of Conditions~\ref{item:rwordBdl:ip}--
    \ref{item:rwordBdl:positive}, and it is a Hilbert $B(\sigma(x))$-module because we assumed that the norm with respect to which $M(x)$ is complete coincides with the norm induced by the inner product (Assumption~\ref{item:rwordBdl:norm compatible}). Fullness is exactly Assumption~\ref{item:rwordBdl:fibrewise full}.

    Condition~\ref{item:rwordBdl:ip:C*linear on left} follows from \ref{item:rwordBdl:ip:adjoint} and \ref{item:rwordBdl:ip:C*linear}  combined, and \ref{item:rwordBdl:ip:s and r} follows directly from \ref{item:rwordBdl:ip:fibre}.

   \ref{item:rwordBdl:ractB:bilinear} Fix $x\in X$ and $m,m_{i}\in M(x)$. If $b\in B(h)$ for some $h\in \sigma(x)\cH$, then
    \begin{align*}
        \rinner{}{m}{(m_{1}+\lambda m_{2})\ractB b}
        &\overset{\ref{item:rwordBdl:ip:C*linear}}{=}
        \rinner{}{m}{m_{1}+\lambda m_{2}}b
        \\
        &=
        (\rinner{}{m}{m_{1}}+\lambda \rinner{}{m}{m_{2}})b
        &&\text{ by \ref{item:rwordBdl:ip} (sesquilinearity)}.
    \end{align*}
    Since multiplication in~$\cB$ is bilinear, we conclude that
    \begin{align*}
        \rinner{}{m}{(m_{1}+\lambda m_{2})\ractB b}
        =
        \rinner{}{m}{m_{1}\ractB b + \lambda (m_{2}\ractB b)}
    \end{align*}
    for all $m\in M(x)$. Choosing $m=(m_{1}+\lambda m_{2})\ractB b - m_{1}\ractB b \lambda (m_{2}\ractB b)$, it follows from \ref{item:rwordBdl:positive} that $m=0$, meaning that $\mvisiblespace \ractB \mvisiblespace$ is linear in the first component. A similar computation, using \ref{item:rwordBdl:ip:C*linear on left}, proves linearity in the second component and also \ref{item:rwordBdl:ractB:associative}, using associativity of the multiplication of $\cB$.

    \ref{item:rwordBdl:ractB:norm} We have
    \[
        \norm{m\ractB b}^2
        \overset{\ref{item:rwordBdl:norm compatible}}{=}
        \norm{\rip\cB<m\ractB b,m\ractB b>}
        \overset{(\dagger)}{=}
         \norm{b^* \rip\cB<m,m> b}
        \leq 
        \norm{b^*}\norm{\rip\cB<m,m>}\norm{b}
        \overset{\ref{item:rwordBdl:norm compatible}}{=}
        \norm{m}^2\norm{b}^2,
    \]
    where $(\dagger)$ follows from \ref{item:rwordBdl:ip:C*linear} and \ref{item:rwordBdl:ip:C*linear on left}.
    
    \ref{item:rwordBdl:ractB:cts} Suppose $(m_{i },b_{i })$ is a net in $M\bfp{\sigma}{r}B$ that converges to $(m,b)$; in particular, $x_{i }\coloneqq q_{\cM}(m_{i })$ converges to $x\coloneqq q_{\cM}(m)$ in~$X$ and $h_{i }\coloneqq p_{\cB}(b_{i })$ converges to $h\coloneqq p_{\cB}(b)$ in~$\cH$.  Now choose a continuous section $\mu\in\Gamma_0 (X;\cM)$ of~$\cM$ with $\mu(x\ract h)=m\ractB b$. Since $\mu$ and the right $\cH$-action $ \mvisiblespace \ract \mvisiblespace$ on~$X$ are continuous, we have  $\mu(x_{i} \ract h_{i} ) \to m \ractB b$. Using \ref{item:rwordBdl:ip:C*linear}, \ref{item:rwordBdl:ip:C*linear on left}, and sesquilinearity of the inner product, we have
    \begin{align*}
        &\rinner{}{m_{i } \ractB b_{i }
        -
        \mu(x_{i } \ract h_{i } )}{m_{i } \ractB b_{i }
        -
        \mu(x_{i } \ract h_{i } )}
        \\
        &=
        b_{i }^*\rinner{}{m_{i } }{m_{i } }b_{i }
        -
        b_{i }^*
        \rinner{}{m_{i }}{\mu(x_{i } \ract h_{i } )}
        -
        \rinner{}{\mu(x_{i } \ract h_{i } )}{m_{i }}b_{i }
        +
        \rinner{}{
        \mu(x_{i } \ract h_{i } )}{
        \mu(x_{i } \ract h_{i } )}
        .
    \end{align*}
    By continuity of the involution and multiplication of~$\cB$ and by continuity of the inner product on~$\cM$, the above converges to
    \begin{align*}
        b^*\rinner{}{m  }{m  }b 
        -
        b^*
        \rinner{}{m }{m\ractB b}
        -
        \rinner{}{m\ractB b}{m }b 
        +
        \rinner{}{
        m\ractB b}{m\ractB b}
        ,
    \end{align*}
    which, by choice of $\mu$ and again by \ref{item:rwordBdl:ip:C*linear} and \ref{item:rwordBdl:ip:C*linear on left}, is $0$. Since 
    \begin{align*}
        \norm{
        m_{i } \ractB b_{i }
        -
        \mu(x_{i } \ract h_{i } )
        }^2
        &=
        \norm{
        \linner{}{m_{i } \ractB b_{i }
        -
        \mu(x_{i } \ract h_{i } )}{m_{i } \ractB b_{i }
        -
        \mu(x_{i } \ract h_{i } )}
        }
    \end{align*}
    by \ref{item:rwordBdl:norm compatible},
    it thus follows from \cite[Lemma A.5]{DL:MJM2023} that $m_{i } \ractB b_{i }
        -
        \mu(x_{i } \ract h_{i } )$ converges in the total space
        of~$\cM$ to  
        the zero-element of the Banach space $M(x\ract h)$. Since $\mu(x_{i} \ract h_{i} )$ converges to $m \ractB b$ and since limits are unique, this implies $m_{i } \ractB b_{i }\to \mu(x\ract h)=m \ractB b$, as claimed.

    The proof of \ref{item:rwordBdl:ip:CS}
    is verbatim that for \cite[Lemma 4.7]{DL:MJM2023} (after translating from the left to the right).
\end{proof}

\begin{remark}    
    In \cite[Definition 6.1]{MW:2008:Disintegration}, neither continuity of the inner product on a Fell bundle equivalence nor Condition~\ref{item:rwordBdl:norm compatible} were explicitly assumed but often invoked. I am unsure whether there is a way to deduce continuity from a combination of the algebraic properties of~$\cM$ and the topological properties of~$\cB$, like it was the case for continuity of the action, \ref{item:rwordBdl:ractB:cts}. Consequently, I chose to include these conditions as assumptions in the definition.
\end{remark}

Note that a two-sided version of a \demiequiv\ is an equivalence in the sense of \cite[Definition 6.1]{MW:2008:Disintegration}, provided there is some algebraic compatibility between the two structures:

\begin{proposition}\label{prop:from two-sided demi to equivalence}
    Suppose that $X$ is a groupoid equivalence between~$\cG$ and~$\cH$ with anchor maps $\rho$ respectively $\sigma$, that $\cA=(A\to \cG)$ and $\cB=(B\to \cH)$ are Fell bundles, and that $\cM=(M\to X)$ is a left $\cA$- and a \demiequiv[right $\cB$-]. Assume further that 
    \begin{enumerate}[label=\textup{(\arabic*)}]
        \item\label{item:demi to full:ractB and lactB commute} the actions of $\cA$ and~$\cB$ on~$\cM$ commute, i.e., for all $(a,m,b)\in A\bfp{s}{\rho}M\bfp{\sigma}{r}B$, we have $(a\lactB m)\ractB b= a\lactB (m\ractB b)$;
        \item\label{item:demi to full:ractB and lactB by adjointables}
        for each $x\in X$, $A(\rho(x))$ acts by $B(\sigma(x))$-adjointable operators on $M(x)$, meaning that
        \begin{align*}
            \rip\cB<a \lactB m_{1},m_{2}>
            =
            \rip\cB<m_{1},a^* \lactB m_{2}>
        \end{align*}
        for all $m_{i}\in M(x)$ and all $a\in A(\rho(x))$;
        \item\label{item:demi to full:ips compatible} for all $(m_{1},m_{2},m_{3})\in M\bfp{\sigma}{\sigma}M \bfp{\rho}{\rho} M$, the inner products on~$\cM$ satisfy
        \begin{equation*}\label{eq:demi to full:ips compatible}
            \lip\cA<m_{1},m_{2}>\lactB m_{3}
            =
            m_{1}\ractB \rip\cB<m_{2},m_{3}>.
        \end{equation*}
    \end{enumerate}
    Then~$\cM$ is an equivalence between $\cA$ and $\cB$.
\end{proposition}

Note that Assumption~\ref{item:demi to full:ractB and lactB by adjointables} is  distinctly asymmetric; it could have  been replaced by its counterpart:
\begin{enumerate}
    \item[\ref*{item:demi to full:ractB and lactB by adjointables}'] 
        For each $x\in X$, $B(\sigma(x))$ acts by $A(\rho(x))$-adjointable operators on $M(x)$.
\end{enumerate}

\begin{proof}
    First, we will do some sanity checks: Condition~\ref{item:demi to full:ractB and lactB commute} makes sense because the actions on $X$ commute, meaning that $(a\lactB m)\ractB b$ and $a\lactB (m\ractB b)$ live over the same fibre of~$\cM$ by \ref{item:rwordBdl:ractB:fibre}.   
     Next, let us check that Condition~\ref{item:demi to full:ips compatible} makes sense.
    Since $\lip\cA<\mvisiblespace,\mvisiblespace>$ is defined on $M\bfp{\sigma}{\sigma}M$ and $\rip\cB<\mvisiblespace,\mvisiblespace>$ on $M\bfp{\rho}{\rho}M$, we can evaluate the shown inner products. By \ref{item:rwordBdl:ip:s and r}, we have
          $\sigma_{\cM}(m_{2})=r_{\cB}(\rip\cB<m_{2},m_{3}>)$, 
          and since $(m_1,m_2)\in M\bfp{\sigma}{\sigma}M$, we therefore have  that $m_{1}$ can be acted on by $\rip\cB<m_{2},m_{3}>$ on the right; a similar argument shows that $m_{3}$ can be acted on by $\lip\cA<m_{1},m_{2}>$ on the left. Lastly, note that Conditions~\ref{item:rwordBdl:ractB:fibre} and \ref{item:rwordBdl:ip:fibre} combined with Equation~\eqref{eq:leoq into reoq} show that the elements in question are indeed living in the same fibre of $\cM$.

    To see that~$\cM$ is an equivalence, we have to check Properties~(FE1), (FE2), and~(FE3) in \cite[Definition 2.11]{DL:MJM2023}. By assumption,~$\cM$ is an \uscBb\ over a groupoid equivalence, and as explained in Lemma~\ref{lem:rwordBdl are nice}, 
    $\mvisiblespace\lactB\mvisiblespace$ and $\mvisiblespace\ractB\mvisiblespace$ are actions in the sense of \cite[Definition 2.10]{DL:MJM2023}. Since they are assumed to commute, we therefore get Condition~(FE1). For (FE2), note that the inner products are assumed to be sesquilinear in the correct sense. Furthermore,
    \begin{enumerate}[leftmargin = 2cm]
        \item[(FE2.a)] is satisfied  by \ref{item:rwordBdl:ip:fibre};
        \item[(FE2.b)] is satisfied  by \ref{item:rwordBdl:ip:adjoint}; 
        \item[(FE2.c)] is satisfied  by \ref{item:rwordBdl:ip:C*linear}; and
        \item[(FE2.d)] is exactly Assumption~\ref{item:demi to full:ips compatible}.
    \end{enumerate}
    Lastly, for (FE3), we must check that each $M(x)$ is an \ib\ between $A(\rho(x))$ and $B(\sigma(x))$. We know by \ref{item:rwordBdl:M(x) full} that it is a  full left- and right-Hilbert $\cst$-module. By Assumption~\ref{item:demi to full:ips compatible}, the left- and right-inner products are compatible and the left-action is by $B(\sigma(x))$-adjointable operators; it remains to show that the right-action is by $A(\rho(x))$-adjointable operators. For  $n_{i}\in M(x)$, we have
    \begin{align*}
        \lip\cA<n_{1},n_{2}\ractB b> \lactB n_{3}
        &\overset{\ref{item:demi to full:ips compatible} }{=}
        n_{1} \ractB \rip\cB<n_{2}\ractB b, \lactB n_{3}>
        \overset{\ref{item:rwordBdl:ip:C*linear on left}}{=}
        n_{1} \ractB \left( b^* \rip\cB<n_{2}, \lactB n_{3}>\right)
        \\
        &=        
         (n_{1} \ractB b^*)\ractB \rip\cB<n_{2}, \lactB n_{3}>
        \overset{\ref{item:demi to full:ips compatible} }{=}
        \lip\cA<n_{1}\ractB b^*,n_{2}>\lactB n_{3}.
    \end{align*}
    This proves that the element $a\coloneqq \lip\cA<n_{1},n_{2}\ractB b>  - \lip\cA<n_{1}\ractB b^*,n_{2}>$ of $A(\rho(x))$ `kills' all of $M(x)$. By Assumption~\ref{item:demi to full:ractB and lactB by adjointables}, $a$ is an adjointable operator with respect to the $B(\sigma(x))$-valued right-inner product. It therefore follows from \cite[Remark 2.29.]{RaWi:Morita} that $a^*a$ and hence $a$ is $0$, i.e., $\lip\cA<n_{1},n_{2}\ractB b> = \lip\cA<n_{1}\ractB b^*,n_{2}>$, so $b$ is an  $A(\rho(x))$-adjointable operator.
\end{proof}

\begin{remark}\label{rmk:first comparison to AF}
    In the setting of Proposition~\ref{prop:from two-sided demi to equivalence}, it follows from \cite[Corollary 4.6.]{DL:MJM2023} that Condition~\ref{item:demi to full:ips compatible} actually also holds more generally: if $(a,m_{1},m_{2})\in A\bfp{s}{\rho}M\bfp{\sigma}{\sigma}M$, then $\rip\cB<a \lactB m_{1},m_{2}>
            =
            \rip\cB<m_{1},a^* \lactB m_{2}>$ as elements of $B(h)$ where $h=\reoq[X]{q_{\cM}(m_{1})}{p_{\cA}(a)\inv\lact q_{\cM}(m_{2})}{\cH}$. In other words, $\cA$ is  acting on~$\cM$ by `$\cB$-adjointable' operators.  This observation foreshadows a connection to \cite{AF:EquivFb}, where results similar to the main theorem of the paper at hand have appeared. We will do a more in-depth comparison later in Corollary~\ref{cor:cf AF's K(M)}, but let us already point out some points of distinction:
    On the one hand,  their result is more general in that they consider Fell bundles  that are not necessarily separable or saturated; 
    instead of \ref{item:rwordBdl:fibrewise full}, they only assume that $\cM=(M\to X)$ satisfies
    \begin{align}\label{eq:7R}
    \tag{7R}
        \overline{\mathrm{span}}\{\rip{\cB}<M(x),M(x)>: x\in Xu\}
    &=
    B
    (u)
    \end{align}
    On the other hand, their result is more restrictive in that our topological groupoid~$\cH$ is replaced by a topological {\em group}. Moreover, they only consider $X=\cH$, which is restrictive even in the case of groups (see Example~\ref{ex:Paul's Ex 14}).
\end{remark}

We have  introduced all ingredients for the main theorem which says that any~\demiequiv[$\cB$-] can be rigged (in a unique way) to give a Fell bundle equivalence in the sense of  \cite[Definition 6.1]{MW:2008:Disintegration}.

\section{The \iFbdl: Existence}\label{sec:Existence}

 \fix we fix a Fell bundle $\cB=(p_{\cB}\colon B\to \cH)$ over the groupoid~$\cH$  and a \demiequiv[right $\cB$-]\  $\cM=(q_{\cM}\colon M\to X)$ over the \psp[right ]{$\cH$} $X$.

\medskip

Suppose that~$\cA$ and~$\cC$ are, like~$\cB$, Fell bundles over \LCH\ groupoids. 
In \cite[Sections~5 and 6]{DL:MJM2023}, it was shown that a so-called {\em hypo-equivalence} $\cN_{1}$ from~$\cA$ to~$\cB$ and a hypo-equivalence $\cN_{2}$ from~$\cB$ to~$\cC$ can be `multiplied' to yield a hypo-equivalence $\cN_{1}\otimes_{\cB}\cN_{2}$ from~$\cA$ to~$\cC$.
A careful examination of the proofs shows that not all the structure of hypo-equivalences is needed to construct the \uscBb\ $\cN_{1}\otimes_{\cB}\cN_{2}$.
We will start this section by making this claim more precise.

	Given $x,y\in X$ with $u\coloneqq\sigma(x)=\sigma(y)$, the \demiequiv[]{}~$\cM$ gives us two full Hilbert $\cst$-modules over the $\cst$-algebra $B(u)$, namely the right-module $M(x)$ and the {\em left}-module $M(y)\op$. In particular, combining the  well-known results mentioned earlier, we get an \ib\ between $\compacts_{B(u)}(M(x))$ and $\compacts_{B(u)}(M(y)\op )$ by taking the balanced tensor product:
	\begin{align}\label{eq:defn of K's fibres}
	K(x,y\op )\coloneqq M(x) \otimes_{u} M(y)\op 
	\cong 
	\compacts_{B(u)} \left(M(y),M(x)\right),
	\end{align}
 where we write $\mvisiblespace\otimes_{u}\mvisiblespace$ as short-hand for $\mvisiblespace\otimes_{B(u)}\mvisiblespace$.
The norm of this Banach space is, on sums of elementary tensors, given by
	\begin{align}\label{eq:norm on compacts}
	\norm{ \sum_{i
		} m_{i}\otimes n_{i}\op}
	&=
	\norm{
		\sum_{i,j
		} 
		\ket*{n_{i}\ractB \rinner[M(x)]{B(u)}{m_{i}}{m_{j}} }
		\bra*{ n_{j}}
	}^{1/2}
    && \text{operator norm on $\compacts_{B(u)}(M(y))$}
    \\
    \notag
    &=
    \norm{
        \sum_{i,j} \innercpct[]{}{m_{i}}{m_{j}\ractB \rinner[M(y)]{B(u)}{n_{j}}{n_{i}}}
    }^{1/2}
    && \text{operator norm on $\compacts_{B(u)}(M(x))$}.
	\end{align}
	We construct an  \uscBb\ over $X\bfp{\sigma}{\sigma}X\op$ as follows.
	\begin{lemma}[cf.\ {\cite[Lemma 5.2]{DL:MJM2023}}]\label{lem:DL:MJM2023:Lemma 5.2}
		On the set 
		$$K=\bigsqcup_{(x,y\op )\in X \bfp{\sigma}{\sigma}X\op} 
		K(x,y\op )
		,$$
		consider all cross-sections of the form
		\[
		\mu \otimes\nu\op \colon \quad   X\bfp{\sigma}{\sigma}X\op \to K,\qquad (x,y\op ) \mapsto \mu (x)\otimes \nu (y )\op,
		\]
		for $\mu,\nu \in \Gamma_{0} (X;\cM)$. Then there is a unique topology on 
		$K$ making it an \uscBb\ over $ X\bfp{\sigma}{\sigma}X\op $
		such that all cross-sections $\mu \otimes\nu\op $ are continuous.
	\end{lemma}

    In the literature \cite{DL:MJM2023,Muhly:Bdles}, the above bundle is denoted $\cM\otimes_{\cB\z}\cM\op$. But since we fixed $\cM$, we will ease notation and instead denote the bundle by  $\cK=(q_{\cK}\colon K\to X\bfp{\sigma}{\sigma}X\op)$.
	\begin{proof}
  Consider 
		\[
		\Gamma \coloneqq \mathrm{span}_{\mbb{C}}
		\{
		\mu \otimes\nu\op : \mu,\nu \in \Gamma_{0} (X;\cM)
		\}.
		\]
    Recall that~$\cM$ has enough continuous sections, so that for each $m\in M(x)$, there exists $\mu\in\Gamma_0(X;\cM)$ such that $\mu(x)=m$. In particular, it follows that 
    \( \{\gamma(x,y\op) : \gamma\in\Gamma\}
    \) is dense in $K(x,y\op)$. 
   	If we can show for an arbitrary element $\gamma=\sum_{i=1}^{k}\mu_{i}\otimes\nu\op_{i}$ of $\Gamma$ that $(x,y\op)\mapsto \norm{\gamma(x,y\op)}$ is \usc, then the claim follows from    
   	\cite[Corollary 3.7]{Hofmann1977} (see also \cite[13.18]{FellDoranVol1}, \cite[Theorem C.25]{Wil2007}, \cite[Remark 2.7]{DL:MJM2023}).
    We have by Equation~\eqref{eq:norm on compacts}
		\begin{align*}
			\norm{\gamma(x,y\op)}^2
			&=
			\norm{
				\sum_{i,j
				} 
				\ket*{\nu_{i}(y)\ractB \inner{\mu_{i}(x)}{\mu_{j}(x)}}
				\bra*{ \nu_{j}(y) }
			}.
		\end{align*}
   	It was shown in \cite[Lemma 2.65]{RaWi:Morita} that  the $k\times k$-matrix with $i,j$-entry $\inner{\mu_{i}(x)}{\mu_{j}(x)}$ is positive for each $x\in X$, so there exists a matrix $(b_{i,j}(x))_{i,j}$ over $B(\sigma(x))$ such that $\inner{\mu_{i}(x)}{\mu_{j}(x)} = \sum\nolimits_{l=1}^{k} b_{il}(x)b_{jl}(x)^*$. Note that we may choose each $b_{i,j}\colon x\mapsto b_{i,j}(x)$ to be a {\em continuous} section of the pullback bundle $\sigma^*(\cB)$.
For $1\leq l \leq k$, we let $\nu_{l}(x,y\op)\coloneqq \sum_{i=1}^{k} \mu_{i}(y)\ractB b_{il}(x)$. By continuity of $\ractB$, of each $\mu_{i}$ and $b_{il}$, and of summation in~$\cM$, each $\nu_{l}$ is continuous. Moreover, 
		\begin{align*}
			\norm{\gamma(x,y\op)}^2
			&=
			\norm{
				\sum\nolimits_{l=1}^{k
				} 
				\ket*{\nu_{l}(x,y\op)}
				\bra*{ \nu_{l}(x,y\op)}
			}
	\\		&
   =
			\norm{
				\sum\nolimits_{l=1}^{k
				} 
				\inner{\nu_{l}(x,y\op)}{ \nu_{l}(x,y\op)}
			}&&\text{by Lemma~\ref{lem:norm of compacts}}.
		\end{align*}
 Since the~$\cB$-valued inner product on~$\cM$, summation in~$\cB$, and each $\nu_{l}$ are continuous, and the norm on~$\cB$ is \usc, we conclude that $(x,y\op)\mapsto \norm{\gamma(x,y\op)}$ is \usc.
	\end{proof}

\begin{remark}\label{rmk:DL:MJM2023:Lemma5.4}
    The proof of \cite[Lemma 5.4]{DL:MJM2023} shows that, if we have two convergent nets $m_{\lambda}\to m$ and $n_{\lambda}\to n$ in $M$ with $\sigma_{\cM}(m_{\lambda})=\sigma_{\cM}(n_{\lambda})$, then $m_{\lambda}\otimes n_{\lambda}\op\to m\otimes n\op$ in $K$.
\end{remark}

Implicitly, Lemma~\ref{lem:DL:MJM2023:Lemma 5.2} is making use of the `dual' bundle to the \demiequiv{$\cB$}~$\cM$: given a fibre $M(x)$ of~$\cM$, its conjugate vector space $M(x)\op$ is, by definition, the fibre of 
$\cM\op=(M\op\to X\op)$ over $x\op$. Clearly, $\cM\op$ is also an \uscBb\ if $M\op$ and $X\op$ are given the same topologies as $M$ and $X$, respectively. Moreover, the right~$\cB$-action on~$\cM$ induces a left~$\cB$-action on $\cM\op$ if we define $b\lactB m\op \coloneqq (m\ractB b^{*})\op$. The reader should compare all of this with \cite[Example 6.7]{MW:2008:Disintegration} in the setting where~$\cM$ is an equivalence.

Note that~$\cK$ is the bundle-analogue of the space $X\bfp{\sigma}{\sigma}X\op$ in the world of groupoids and of the unbalanced tensor product $\mbf{X}\otimes \mbf{X\op}$ of Hilbert $\cst$-modules in the world of $\cst$-algebras. What we are actually after, though, is a bundle-analogue of their {\em quotients}, that is, $X\times_{\cH}X\op$  respectively $\mbf{X}\otimes_{A} \mbf{X\op}$:
	To get $\mbf{X}\otimes_{A} \mbf{X\op}$ from $\mbf{X}\otimes  \mbf{X\op}$, we quotient out by elements  of the form 
	\[
	(\mbf{x}\cdot a)\otimes \mbf{y} - \mbf{x}\otimes (a\cdot \mbf{y})
	\]
	where $\mbf{x}\in \mbf{X}, \mbf{y}\in\mbf{X\op}$, and $a\in A$.
	For bundles, such a difference does not make sense, since the two elementary tensors may live in different fibres: if $m\in M(x)$, $n\in M(y) $, and $b\in B(h)$ are such that $\sigma (x)=r_{\cH}(h)$ and $s_\cH (h)=\sigma(y)$, then
	\begin{align*}
	(m\ractB b)\otimes n\op  &\in M(x\ract h)\otimes_{s(h)} M(y)\op 
	, 
	\intertext{while}
	 m\otimes (b \lactB n\op) 
 &\in  M(x)\otimes_{r(h)} M (y\ract h\inv )\op.
	\end{align*}
	In order to identify $(m\ractB b)\otimes n\op $ with $m\otimes (b \lactB n\op)=m\otimes ( n \ractB b^*)\op$, we therefore need a way to identify the above two fibres of $\cK$:
	\begin{lemma}[cf.~{\cite[Theorem 5.20]{DL:MJM2023}}]\label{lem:Psi_h}
		For any $u,v\in\cH\z$ and $h\in u\cH v$, there exists a map 
		\[
		\Psi_{h}\colon\quad
		\bigsqcup_{\substack{x\in Xu\\ y \in Xv}}
		M(x\ract h)\otimes_{v} M(y)\op 
		\to
		\bigsqcup_{\substack{x\in Xu\\y \in Xv}}
		M(x)\otimes_{u} M (y\ract h\inv )\op
		\]
		determined on elementary  tensors  by
		\begin{equation}\label{eq:def:Psi_h}
			\Psi_{h}
			((m\ractB b)\otimes n\op) = m\otimes (n\ractB b^*)\op,     
		\end{equation}
		where $\sigma_{\cM}(m)=u$, $\sigma_{\cM}(n)=v$, and
    $b\in B(h)$.   These maps have the following properties: 
			\begin{enumerate}[label=\textup{($\Psi$\arabic*)}]
				\item When restricted to a single fibre $K(x\ract h,y\op)\to K(x,(y\ract h\inv)\op) $, each $\Psi_h$ is linear.
            \item\label{item:Psi:isometric} Each $\Psi_{h}$ is isometric, i.e., $\norm{\Psi_{h}(\xi)}=\norm{\xi}$.
                \item
				$\Psi_{h'}\circ \Psi_{h}=\Psi_{h'h}$ for $(h',h)\in\cH\comp $.
				\item
				$\Psi_{h\inv}$ is inverse to $\Psi_{h}$.
				 \item
				$\Psi_{u}$ for $u\in \cH\z$ is the identity map.
			\end{enumerate}
    
		\end{lemma}	 
  \begin{proof}
    We construct $\Psi_{h}$ on each of the fibres $M(x\ract h)\otimes_{v} M(y)\op$.
    Consider  the map
    \[
        M(x) \times B(h) \times M(y)\op 
		\to
		M(x)\otimes_{u} M (y\ract h\inv )\op,
      \quad
      (m,b,n\op) \mapsto m\otimes (n\ractB b^*)\op
      .
    \]
    Since it is multilinear, it descends to a linear map with domain $[M(x) \odot B(h)] \odot M(y)\op $, where $\odot$ denotes the algebraic tensor product. 
    Because of the $B(u)$-balancing in the codomain, the map descends to a map with domain $[M(x) \odot_{B(u)} B(h)] \odot M(y)\op$.
    Since 
    \(
    [n\ractB b_{v}^*]\ractB b^*
    =
    n\ractB [bb_{v}]^*     
    \)
    for any $b_v\in B(v)$, it further descends to a map with domain $[M(x) \odot_{B(u)} B(h)] \odot_{B(v)} M(y)\op$. It is easy to check that that map is isometric (when $M(x) \odot_{B(u)} B(h)$ is given the norm of the balanced tensor product of Hilbert modules), and therefore extends to $[M(x) \otimes_{u} B(h)] \otimes_{v} M(y)\op$. Now, Assumption~\ref{item:rwordBdl:fibrewise full} implies that every element of $M(x\ract h)$ is the limit of sums of elements of the form $m\ractB b$; in other words, the linear map $M(x) \odot_{B(u)} B(h)\to M(x\ract h)$ determined by $m\odot b \mapsto m\ractB b$ has dense range. 
    We conclude the existence of    $\Psi_{h}$. The remaining properties are easy to check.
  \end{proof}

  The collection of maps $\Psi_h$ stitch together to give an isomorphism of \uscBb s as follows. 	
Let $t\colon X \bfp{\sigma}{\sigma}X\op\to \cH\z$ be given by $t(x,y\op )\coloneqq \sigma(x)=\sigma\op (y\op )$, and consider the continuous projection map
	\[
	f\colon \quad\cH\bfp{s}{t}(X \bfp{\sigma}{\sigma}X\op)
	\to X \bfp{\sigma}{\sigma}X\op,\quad
	(h,x,y\op )\mapsto (x,y\op ).
	\]
	The pull-back bundle of the \uscBb~$\cK$ via $f$
	is the bundle over the domain of $f$ 
 defined by
	\begin{align*}
		&\bigl\{
		(h,x,y\op ,\xi)\in [\cH\bfp{s}{t}(X \bfp{\sigma}{\sigma}X\op)]\times K
		:
		f(h,x,y\op )=q_{\cK}(\xi)
		\bigr\}
		\\
		&\cong
		\bigl\{
		(h,\xi)\in \cH\times K
		:
		s_{\cH}(h)=t(q_{\cK}(\xi))
		\bigr\}.
	\end{align*}
	While this bundle is often denoted by $f^{*} (\cK)$, we will denote it by $\cH \bfp{s}{t}  \cK$ instead, so that the letter $f$ does not have to be introduced. We make an analogous definition for $\cK \bfp{t}{r}  \cH$.
 
    It was shown in \cite[Lemma 2.8.]{DL:MJM2023} that basic open sets of this \uscBb\ are of the form $U\bfp{s}{t}V$ for $U\subseteq \cH$, $V\subseteq K$ open, and the sections  
	\[
	\cH \bfp{s}{t}  (X\bfp{\sigma}{\sigma}X\op)
	\to
	\cH \bfp{s}{t}  \cK,
	\quad
	(h,x,y\op )\mapsto (h,\tau(x,y\op )),
	\]
	for $\tau$ a continuous section of $\cK$, 
    are the continuous sections that uniquely determine the topology on $\cH \bfp{s}{t}  \cK$.

  \begin{lemma}[cf.\ {\cite[Theorem 5.20.]{DL:MJM2023}}]\label{lem:DL:MJM2023:Thm5.20}
      The map
		\begin{align*}
			\Psi\colon\quad 
			\cH \bfp{s}{t}  \cK
			&\to
			\cK\bfp{t}{r}  \cH,
			&
			(h, \xi)&\mapsto (\Psi_{h}
			(\xi), h)
            ,
			\intertext{is an isomorphism of \uscBb s covering the homeomorphism}
				\psi\colon
				\quad
				\cH \bfp{s}{t}  (X \bfp{\sigma}{\sigma}X\op)
				&\to
				(X \bfp{\sigma}{\sigma}X\op) \bfp{t}{r} \cH,
				&
				(h,x\ract h,y\op )& \mapsto (x, (y \ract h\inv)\op , h),
			\end{align*}
   In particular, $\Psi$ is jointly continuous.
  \end{lemma}
  
		We remind the reader that `{\em $\Psi$ covers $\psi$}' means that the following diagram commutes:
		\[
		\begin{tikzcd}
			\cH \bfp{s}{t}  \cK
			\ar[r, "\Psi"]\ar[d, "q"]
			&
			\cK\bfp{t}{r}  \cH\ar[d, "q"]
			\\
			\cH \bfp{s}{t}  (X \bfp{\sigma}{\sigma}X\op)
			\ar[r, "\psi"]
			&
			(X \bfp{\sigma}{\sigma}X\op) \bfp{t}{r} \cH
		\end{tikzcd}
		\]
		So given $\xi\in K$ and $h\in\cH$ with $s_{\cH}(h)=t(q_{\cK}(\xi))$, this means that
		\begin{equation}\label{eq:q of Psi_h}
			q_{\cK} (\xi)
			= (x\ract h, y\op )
			\iff 
			q_\cK (\Psi_{h} (\xi)) = (x,(y \ract h\inv)\op ).
		\end{equation}

    \begin{proof}    
        To see that $\Psi$ is an isomorphism, we will invoke Lemma~\ref{lem:uscBb:cts} and \cite[Proposition A.8]{DL:MJM2023}. Because of Lemma~\ref{lem:DL:MJM2023:Lemma 5.2} and our comment preceding Lemma~\ref{lem:DL:MJM2023:Thm5.20}, the set
        \[
            \Gamma 
            =
            \{
	           (h,x,y\op )\mapsto (h,\mu(x)\otimes \nu(y)\op )
                :
                \mu,\nu\in \Gamma_0(X;\cM)
            \}
        \]
        is a collection of continuous sections of $\cH \bfp{s}{t}  \cK$. As the fibre of $\cH \bfp{s}{t}  \cK$ over a given $(h,x,y\op )$ is just $K(x,y\op)$ and since 
        the linear span of $\{\mu(x)\otimes \nu(y)\op:\mu,\nu\in \Gamma_0(X;\cM)\}$ is dense in $K(x,y\op)$, we conclude that the linear span of $\{\gamma(h,x,y\op):\gamma\in\Gamma\}$ is likewise dense in the fibre of $\cH \bfp{s}{t}  \cK$. 
        We claim that, given any element $\gamma\colon (h,x,y\op )\mapsto (h,\mu(x)\otimes \nu(y)\op )$ of $\Gamma$,
        the section
        \(
            \Psi\circ \gamma \circ \psi\inv 
        \)
         of $\cK\bfp{t}{r}  \cH$
        is  continuous. Indeed, it suffices to check that, if 
        $$(x_{\lambda},(y_{\lambda} \ract h_{\lambda}\inv)\op,h_{\lambda})\to (x,(y \ract h)\op,h)\text{ in }(X\bfp{\sigma}{\sigma}X\op)\bfp{t}{r}\cH,$$
         then 
         $$
        \Psi_{h_{\lambda}}\bigl(\mu(x_{\lambda} \ract h_{\lambda})\otimes \nu(y_{\lambda})\op\bigr) \to  
        \Psi_{h}\bigl(\mu(x \ract h)\otimes \nu(y)\op\bigr)\text{ in }K\bfp{t}{r}  \cH.$$

    Fix $\epsilon>0$.  As mentioned earlier, Assumption~\ref{item:rwordBdl:fibrewise full} implies that $M(x\ract h)=M(x)\ractB B(h)$, so we may take finitely many elements $m_{i}\in M(x)$ and $b_{i}\in B(h)$ such that
    \begin{align}\label{eq:choice of mi bi}       \norm{\mu(x\ract h)
   -
   \sum_{i=1}^{k} m_{i}\ractB b_{i}
   }
   <\epsilon.
   \end{align}
    For each $1\leq i\leq k$, fix a section $\mu_{i}$ of~$\cM$ and $\tau_{i}$ of~$\cB$ with $\mu_{i}(x)=m_i$ respectively $\tau_{i}(h)=b_{i}$. Since addition in~$\cM$ and the right $\cB$-action are continuous,the net $\zeta_{\lambda}\coloneqq \sum_{i} \mu_{i}(x_{\lambda})\ractB \tau_{i}(h_{\lambda})$ converges to $\zeta\coloneqq \sum_{i} \mu_{i}(x)\ractB \tau_{i}(h)$ in $M$.
    By \cite[Lemma A.3, $(i)\implies (iii)$]{DL:MJM2023}, we therefore have
    \begin{align}\label{eq:limsup}
        \limsup_{\lambda}
        \norm{
        \mu(x_{\lambda} \ract h_{\lambda})
        -
        \zeta_{\lambda}
        }
        \leq 
        \norm{
        \mu(x \ract h)
        -
        \zeta
        }
        <\epsilon.
    \end{align}
    Note that
        \begin{align*}
            \Psi_{h_{\lambda}}\bigl(\zeta_{\lambda}\otimes \nu(y_{\lambda})\op\bigr)
            &=             
            \sum_{i=1}^{k}
            \mu_{i}(x_{\lambda})
            \otimes \bigl(\nu(y_{\lambda})\ractB \tau_{i}(h_{\lambda})^*\bigr)\op,
        \end{align*}
    and likewise without subscript-$\lambda$'s. Thus, by continuity of $\mu_{i},\nu,\tau_{i}$, of the involution in $\cB$, of the right $\cB$-action, and of addition in $\cK$, combined with Remark~\ref{rmk:DL:MJM2023:Lemma5.4}, implies that
    \begin{equation}\label{eq:cty with zetas}
    \Psi_{h_{\lambda}}(\zeta_{\lambda}\otimes \nu(y_{\lambda})\op) \to   \Psi_{h}(\zeta \otimes \nu(y)\op)
    .
    \end{equation}
    Since $\Psi_{h}$ is isometric \ref{item:Psi:isometric}, we then have for large enough $\lambda$
    \begin{align*}
        &\norm{\Psi_{h_{\lambda}}(\zeta_{\lambda}\otimes \nu(y_{\lambda})\op)
        -
        \Psi_{h_{\lambda}}(\mu(x_{\lambda} \ract h_{\lambda})\otimes \nu(y_{\lambda})\op)}
        =
        \norm{
        (\zeta_{\lambda} 
        -
        \mu(x_{\lambda} \ract h_{\lambda})
        )
        \otimes \nu(y_{\lambda})\op}
        \\
        &\quad
        \leq
        \norm{
        \zeta_{\lambda}
        -
        \mu(x_{\lambda} \ract h_{\lambda})
        }
        \norm{\nu(y_{\lambda})}
        \overset{(*)}{\leq}  \epsilon (\epsilon + \norm{\nu(y)}),
    \end{align*}
    where $(*)$ follows from \eqref{eq:limsup} and continuity of $\nu$. The same computation without subscripts yields
    \begin{align*}
        &\norm{\Psi_{h}(\zeta\otimes \nu(y)\op)
        -
        \Psi_{h}(\mu(x \ract h)\otimes \nu(y)\op)}
        \leq
        \epsilon (\epsilon + \norm{\nu(y)})
        .
    \end{align*}
    As $\epsilon$ was arbitrary, Lemma~\ref{lem:convergence uscBb:close nets} states that this combined with \eqref{eq:cty with zetas} implies  \[
        \Psi_{h_{\lambda}}\bigl(\mu(x_{\lambda} \ract h_{\lambda})\otimes \nu(y_{\lambda})\op\bigr) \to  
        \Psi_{h}\bigl(\mu(x \ract h)\otimes \nu(y)\op\bigr).
        \]          
        Because each $\Psi_{h}$ is linear, isometric, and surjective,
        the claim now follows from Lemma~\ref{lem:uscBb:cts} and \cite[Proposition A.8]{DL:MJM2023}.  
    \end{proof}
   We are now able to proceed with the balancing by defining a relation $\FBRel$ on the set~$K$ as follows; the results are adapted from \cite[Section~6]{DL:MJM2023}.
		\begin{lemma}\label{lem:FBRel}
			If $\xi_{1}, \xi_{2}\in K$ with $q_{\cK} (\xi_{i})=(x_{i},y_{i}\op )$, let
			\[
			\xi_{1} \FBRel  \xi_{2}
			\qquad
			:\iff 
			\qquad
			\exists h\in \sigma(x_{1})\cH \sigma(x_{2}) \text{ such that } 
			\quad
			x_{2}=x_{1}\ract h,
			\quad 
			y_{2} = y_{1}\ract h,
			\quad
			\Psi_{h} (\xi_{2}) = \xi_{1}.
			\]
		This defines a closed equivalence relation on the total space~$K$ of $\cK$ whose quotient map $Q\colon K\to K/\FBRel$ is open.
    \end{lemma}
  We often write $[\xi]$ for $Q(\xi)$.

\begin{proof}
The proof that $\FBRel$ is closed is verbatim that given for \cite[Lemma 6.2]{DL:MJM2023}, where it is also explained why $\FBRel$ is an equivalence relation.

To show that the quotient map $Q$ is open, the proof of \cite[Proposition 6.8]{DL:MJM2023} goes through for $\cN\coloneqq\cM\op$, even though~$\cM$ was assumed to be an equivalence and~$\cH$ was assumed to be \etale\ in \cite{DL:MJM2023}.
In fact, \etale ness was only  needed to be allowed to invoke \cite[Lemma 6.10]{DL:MJM2023}, but that lemma's proof holds as long as the range map of~$\cH$ is open (which we assumed here). 
\end{proof}

\begin{lemma}\label{lem:p_cA}
    On the quotient  of~$K$ by $\FBRel$,        the map $p\colon K/\FBRel \to \cG, [\xi]\mapsto [q_{\cK}(\xi)]$, is well defined, surjective, continuous, and open.
\end{lemma}

\begin{proof}
    Verbatim the proof of \cite[Lemma 6.5]{DL:MJM2023} for $\cN\coloneqq \cM\op$, even though~$\cM$ was assumed to be an equivalence in that lemma.
\end{proof}

\fixnow we will denote the quotient bundle  of $\cK$  by $\cA=(p_{\cA}=p\colon A\to \cG)$. Once we have shown that $\cA$ is a Fell bundle that is equivalent to~$\cB$ via $\cM$, we will change notation.

\begin{lemma}[cf.~{\cite[Proposition 6.6. and Lemma 6.7.]{DL:MJM2023}}]\label{lem:Q| is homeo}
        Fix $g=\leoq[X]{\cG}{x}{y}\in \cG$. With respect to the subspace topology of $A=K/\FBRel$ on the fibre $A(g)$ of $\cA$, the restriction of the quotient map $Q\colon K\to A$ to the fibre over $(x,y\op)$ is a homeomorphism:
       \begin{align}\label{eq:cA's fibres}
        Q_{(x,y\op)}\colon \quad K(x,y\op)=M(x)\otimes_{\sigma(x)}M(y)\op
        \overset{\approx}{\longrightarrow}A(g).
       \end{align}
\end{lemma}

      The above lemma means that we can give  $A(g)$ the Banach space structure that makes all maps $Q_{(x,y\op)}$ {\em isomorphisms}: for $[\xi_{1}],[\xi_{2}] \in A(g)$, 
   		there exists a unique $h\in\cH$ such that $q_{\cK}(\xi_{1})= q_{\cK}(\Psi_{h} (\xi_{2}))$, and we may let 
			\[ 
			\lambda_{1}[\xi_{1}]+   \lambda_{2}[\xi_{2}]
			\coloneqq
			[\lambda_{1}\xi_{1} +   \lambda_{2}\Psi_{h}(\xi_{2})]
			\quad\text{and}\quad
			\norm{[\xi]}
			\coloneqq \norm{\xi}_{K(q(\xi) )}.
			\]
    Note that, in light of the Hilbert module isomorphism in Equation~\eqref{eq:defn of K's fibres}, we can also think of the fibres of $\cA$ as generalized compact operators:
    \begin{align}\label{eq:cA's fibres as compacts}
        \compacts_{B(\sigma(x))}\bigl(M(y), M(x)\bigr)
        \overset{\approx}{\longrightarrow}A([x,y\op]).
       \end{align}
\begin{proof}
    If $\xi,\eta\in K(x,y\op)$, then since the $\cH$-action on~$X$ is free, $h=\sigma(x)\in \cH\z$ is the only possible element that allows $\xi\FBRel\eta$, in which case $\Psi_{h}(\eta)=\eta$ (since $h$ is a unit) and hence $\xi=\eta$. Thus, $Q_{(x,y\op)}$ is injective. For surjectivity, note that $Q$ is surjective, so any element $a$ of $A(g)$ can be written as $a=Q(\xi)$ for $\xi\in K$ with $[q_{\cK}(\xi)]=g=\leoq[X]{\cG}{x}{y}$. In particular, there exists $h\in \cH$ with $q_{\cK}(\xi)=(x\ract h, (y\ract h)\op)$, so that $\Psi_{h}(\xi) \in K(x,y\op)$ and $a=Q(\xi)=Q_{(x,y\op)}(\Psi_{h}(\xi))$.

    Since $p=p_{\cA}$ is continuous, $A(g)=p_{\cA}\inv (\{g\})$ is closed in $A$, and so continuity of $Q$ implies continuity of $Q_{(x,y\op)}$. To see that $Q_{(x,y\op)}$ is closed (and hence a homeomorphism), we follow the idea in the proof of \cite[Lemma 6.7.]{DL:MJM2023}: Let $F\subseteq K(x,y\op)$ be a closed set. To show that $Q(F)\subseteq A(g)$ is closed, it suffices to show that it is closed in $A$, i.e., that $Q\inv (Q(F))$ is closed in $K$. So assume $\xi_{\lambda}\to \xi$ in $K$ with $Q(\xi_{\lambda})\in Q(F)$; we must prove that $Q(\xi)\in Q(F)$. Since $Q(\xi_{\lambda})\in F$, there exist $\eta_{\lambda}\in F$ with $\eta_{\lambda}\FBRel\xi_{\lambda}$. Since $F\subseteq K(x,y\op)$, this means there exist $h_{\lambda}\in \sigma(x)\cH $ with $q_{\cK}(\xi_{\lambda})=(x\ract h_{\lambda}, (y\ract h_{\lambda})\op)$ and $\Psi_{h_{\lambda}}(\xi_{\lambda})=\eta_{\lambda}$. Since $\xi_{\lambda}\to \xi$, we have
    \[
    (x\ract h_{\lambda}, (y\ract h_{\lambda})\op)
    =
    q_{\cK}(\xi_{\lambda})
    \to
    q_{\cK}(\xi)
    =:
    (x_0,y_0\op)
    .
    \]
    Since the $\cH$-action on~$X$ is proper, it follows from \cite[Proposition 2.17]{Wil2019} that (a subnet of) $\{h_{\lambda}\}_{\lambda}$ converges to, say, $h$; since~$X$ is Hausdorff, it follows that $(x_0,y_0\op)=(x\ract h, (y\ract h)\op)$. Since $\Psi$ is jointly continuous, we conclude that (a subnet of) $\eta_{\lambda}
    =
    \Psi_{h_{\lambda}}(\xi_{\lambda})$ converges to $\Psi_{h}(\xi)$
    in $K$. Since $\eta_{\lambda}\in F$ and $F$ is closed in $K(x,y\op)$ and hence also in $K$, we have $\Psi_{h}(\xi)\in F$. Since $\xi\FBRel\Psi_{h}(\xi)$, this means that $\xi\in Q\inv (Q(F))$, as needed.
\end{proof}

  \begin{lemma}\label{lem:cA is uscBb}
    The bundle $\cA$ - equipped with the quotient topology on $A=K/\FBRel$ and with the fibrewise linear structure inherited from $\cK$ via the maps given in \eqref{eq:cA's fibres} - is an \uscBb. With respect to this topology, all sections of the form 
	\[
		[x,y\op]\mapsto [\mu(x)\otimes \nu(y)\op]
	\]
	are continuous, where  $\mu,\nu$ are continuous sections of~$\cM$.
  \end{lemma}

\begin{proof}
    That $\cA$ is an \uscBb\ follows from an application of \cite[Proposition 6.13]{DL:MJM2023}; here, we need that both $\leoqempty[X]{\cG}\colon X\bfp{\sigma}{\sigma}X\op\to \cG$ and $Q\colon K\to A$ are open quotient maps (Lemma~\ref{lem:FBRel}) and that $Q_{(x,y\op)}$ is surjective (Lemma~\ref{lem:Q| is homeo}) and a linear isometry (by definition of the linear structure on each fibre of $\cA$).

    That the given section is continuous follows  since $\leoqempty[X]{\cG}$ is open, $\mu \otimes\nu\op$ is a continuous section of $\cK$ (Lemma~\ref{lem:DL:MJM2023:Lemma 5.2}), and $Q$ is continuous. 
\end{proof}

At this point, we will diverge from what was done in \cite{DL:MJM2023} and prove that 
$\cA=(p_{\cA}\colon A\to \cG)$
is not just an \uscBb\ but actually a Fell bundle; see Theorem~\ref{thm:A is FB}. In particular, we need to construct two maps, namely a multiplication $\mvisiblespace\cdot\mvisiblespace\colon \cA\comp\to A$ and an involution $\mvisiblespace^*\colon A\to A$. Conceptually, the involution is easier, so this is where we will start.

\begin{lemma}\label{lem:cK:involution}
   There exists a unique bijective, fibrewise isometric and conjugate-linear map $\Flip\colon\cK\to \cK$ that covers the homeomorphism
   \[
        \flip\colon X\bfp{\sigma}{\sigma}X\op\to X\bfp{\sigma}{\sigma}X\op,
        \quad 
        (x,y\op)\mapsto (y,x\op),
   \]
   and that is fibrewise given on dense spanning elements by
   \begin{align}
    \label{eq:cK:*}
    \begin{split}
        K(x,y\op)
        &\to  
        K(y,x\op)
        \\
        m\otimes n\op
        &\mapsto
        n\otimes m\op.
    \end{split}
    \end{align}
\end{lemma}

Note that the following diagram commutes:
\begin{equation}\label{diag:Flip and Psi}
	\begin{tikzcd}[row sep=large, column sep = large, every label/.append style ={font = \small}]
		K(x\ract h, y\op)
		\ar[r, "\Flip"]\ar[d, "\Psi_{h}"']&
		K(y,(x\ract h)\op)\ar[d, "\Psi_{h}"]
		\\
		K(x, (y\ract h\inv)\op) 
		\ar[r, "\Flip"]&
		K(y\ract h\inv, x\op) 
	\end{tikzcd}
\end{equation}

\begin{proof}
    To see that the map $\Flip$ exists on each fibre, fix $(x,y\op)\in X\bfp{\sigma}{\sigma}X\op$ and let $u=\sigma(x)=\sigma(y)$. Recall that $K(x, y\op)$ is isomorphic to $\compacts_{B(u)} (M(y),M(x))$ as bi-Hilbert $\compacts_{B(u)} (M(x))-\compacts_{B(u)} (M(y))$-bimodules. With this identification, the restriction of  $\Flip$    to this fibre is given by
    \(
        \innercpct[]{}{m}{n}   \mapsto \innercpct[]{}{n}{m},
    \)
    or in other words, $\Flip$ is the adjoint map $T\mapsto T^*$. In particular,  $\Flip$ is fibrewise antilinear and isometric. Its continuity is built into the topology on $\cK$:
    If $\mu,\nu$ are continuous sections of $\cM$, then $\mu\otimes\nu\op$ and $\nu\otimes\mu\op$ are continuous sections of $\cK$. Since $\Flip$ transforms the former into the latter, it is continuous  and open by an application of \cite[Propositions A.7 and~A.8]{DL:MJM2023}.\footnote{To be pedantic, \cite{DL:MJM2023} only deals with fibrewise linear rather than antilinear maps, so the cited results give an isomorphism $\cK\to\cK\op$ of Banach bundles determined by $m\otimes n\op\mapsto (n\otimes m\op)\op$ covering the map $(x,y\op)\mapsto (y,x\op)\op$. But we can then compose that map with the homeomorphism $\cK\op\to\cK$, $\xi\op\mapsto\xi$.}
\end{proof}

Continuity of $\Flip$ and continuity and openness of the quotient map $Q$ now immediately imply the following:
\begin{corollary}\label{cor:cA:involution}
   On the quotient bundle $\cA$, there exists a unique continuous, fibrewise conjugate-linear map $\mvisiblespace^{\ast}\colon A \to A$ given by $[\xi]^*=[\Flip(\xi)]$.
\end{corollary}

Now that we have (a candidate for) the involution on $\cA$, we proceed to construct the multiplication. Recall that our end goal is not only to show that $\cA$ is a Fell bundle but also that there is a left action $\lactB$ of $\cA$ on $\cM$. We want this action to behave nicely with respect to the multiplication in the sense that $a_{1}\lactB (a_{2}\lactB m)=(a_{1}\cdot a_{2})\lactB m$; it is therefore easier to {\em first} construct a candidate for the left-action and then use it to construct the multiplication. Like with the involution, we first do everything on the level of $\cK$ before we move to its quotient $\cA$.  For clarity, we remind the reader  that the fibre of $\cK$ over $(x,y\op)\in X\bfp{\sigma}{\sigma}X\op$ is given by
\(
K(x,y\op)=M(x)\otimes_{\sigma(x)}M(y)\op
\).

\begin{lemma}\label{lem:left action on fibres}
    For $(x,y\op)\in X\bfp{\sigma}{\sigma}X\op$, there exists a continuous bilinear map $$\Phi_{x,y}\colon\quad K(x,y\op)\times M(y) \to M(x)$$ determined on elementary tensors by 
    \begin{equation}\label{eq:def:Phi_xy}
        \Phi_{x,y}(m\otimes n\op, k) = m\ractB \rip\cB<n,k>.
    \end{equation}
    For all $\xi\in K(x,y\op) $ and $k\in M(y)$, it satisfies 
    \begin{equation}\label{eq:Phi:contractive}
    \norm{\Phi_{x,y}(\xi,k)}\leq \norm{\xi}\norm{k}
    .
    \end{equation}
\end{lemma}

\begin{proof}
	The existence of $\Phi_{x,y}$ follows from the following, well-studied  isomorphisms  of bi-Hilbert $A-A$-modules for any full right-Hilbert $A$-module $\mbf{Y}$: Firstly, $\mbf{Y}\op\otimes_{\compacts}\mbf{Y} \to A$ determined by  
 $\mbf{y}_{1}\op\otimes \mbf{y}_{2} \mapsto \rinner[\mbf{Y}]{A}{\mbf{y}_{1}}{\mbf{y}_{1}}$, and secondly, $\mbf{Y} \otimes_{A} A \to \mbf{Y}$ determined by $\mbf{y}\otimes a\mapsto \mbf{y}\cdot a$. 
 In our situation, $A=B(u)$ where $u=\sigma(x)=\sigma(y)$, and we use first $M(y)$ and then $M(x)$ to take the r\^ole of $\mbf{Y}$. To be precise, if we write $\compacts$ for $\compacts_{B(u)} (M(y))$, then
 \[
 K(x,y\op)\times M(y)
  \twoheadrightarrow
  \bigl( M(x)\otimes_{u}M(y)\op \bigr) \otimes_{\compacts} M(y)
  \cong
  M(x)\otimes_{u} B(u)
  \cong
  M(x).
 \]

  The claim about the norm is also well known, but we add it here for completion. Since $\Phi_{x,y}$ is continuous, it suffices to prove the claim for $\xi=\sum_{i=1}^{\ell} m_{i}\otimes n_{i}\op$ a sum of elementary tensors.
  We have
    \begin{align*}
        \norm{\Phi_{x,y}(\xi,k)}^2
       & =\norm{\sum_{i} m_{i}\ractB \rip\cB<n_{i},k>}^2=
        \norm{
            \sum_{i,j} 
            \innercpct[]{}{m_{i}\ractB \rip\cB<n_{i},k>}{m_{j}\ractB \rip\cB<n_{j},k>}
        },
    \end{align*}
    where the norm on the right-hand side is the operator norm on $\compacts_{B(u)}(M(x))$. Note that $\rip\cB<n_{i},k>\in B(u)$ acts by $\compacts_{B(u)}(M(x))$-adjointable operators, so that
    \begin{align*}
    \norm{\Phi_{x,y}(\xi,k)}^2
        &=
        \norm{
            \sum_{i,j} \innercpct[]{}{m_{i}}{\left({m_{j}\ractB \rip\cB<n_{j},k>}\right)\ractB {\rip\cB<k,n_{i}>}}
        }
        .
    \intertext{Now, $\rip\cB<n_{j},k>\rip\cB<k,n_{i}> = \rip\cB<n_{j},k\ractB{\rip\cB<k,n_{i}>}>$ by \ref{item:rwordBdl:ip:C*linear}. Since $k,n_{i}\in M(y)$, we further know that $k\ractB{\rip\cB<k,n_{i}>}=\innercpct[]{}{k}{k}(n_{i})$. Recall that $\innercpct[]{}{k}{k}$ is a positive element of $\compacts_{B(u)}(M(y))$, so we can write it as $T^*T$. In particular,} 
         \norm{\Phi_{x,y}(\xi,k)}^2  &=
        \norm{\sum_{i,j} \innercpct[]{}{m_{i}}{m_{j}\ractB {\rip\cB<Tn_{j},Tn_{i}>}}}
        \overset{\eqref{eq:norm on compacts}}{=}
        \norm{\sum_{i} m_{i}\otimes (Tn_{i})\op}^2
    \end{align*}
 Since $\norm{\sum_{i} m_{i}\otimes (Tn_{i})\op}\leq \norm{1\otimes T}\norm{\sum_{i} m_{i}\otimes n_{i}\op}$ and $\norm{1\otimes T}=\norm{T}=\norm{k}$, we conclude  that
    \[
        \norm{\Phi_{x,y}(\xi, k)}
        \leq
        \norm{\xi}\norm{k},\] as claimed.
\end{proof}

Bilinearity of $\Phi_{x,y}$  helps us prove the following result.

\begin{lemma}\label{lem:U maps}
	For any $u\in\cH\z$ and any $y\in Xu$, there exists a map 
	\[
	U_{y}\colon\quad
	\bigsqcup_{x,z\in Xu}
	K(x,y\op)\times K(y,z\op)
	\to
	\bigsqcup_{x,z\in Xu}
	K(x,z\op)
	\]
	determined by
	\begin{equation}\label{eq:def:U}
		U_{y}
		(m\otimes n_{1}\op,n_{2}\otimes k\op)  
        =
        (m \ractB \rip\cB<n_{1},n_{2}>)\otimes k\op.     
	\end{equation}
	These maps have the following  properties:
	\begin{enumerate}[label=\textup{(U\arabic*)}]
			\item\label{item:U:linear} When restricted to a  fibre $K(x,y\op)\times K(y,z\op) \to K(x,z\op)$, $U_y$ is bilinear.
			\item\label{item:U:norm} Each $U_{y}$ satisfies $\norm{U_{y}(\xi,\eta)}\leq \norm{\xi}\norm{\eta}$
   and $\norm{U_{y} \bigl(\xi, \Flip(\xi)\bigr)}= \norm{\xi}^2$.
		\item\label{item:U:cH equivariant} If $h\in u\cH$, then
		\(
		\Psi_{h}\circ U_{y \ract h } =
		U_{y}\circ (\Psi_{h}\times \Psi_{h})
		\).
        \item\label{item:U:Flip} We have
        \(
        \Flip 
        \bigl(
          U_{y}  (\xi,\eta)
        \bigr)
        =
        U_{y} 
        \Bigl(
          \Flip(\eta),\Flip(\xi)
        \Bigr)
        \).
	\end{enumerate}
\end{lemma}

Note that `equivariance' of $U$ (Condition~\ref{item:U:cH equivariant})
 is equivalent to commutativity of the following diagram for all $x,z$:
\begin{equation}\label{diag:U}
	\begin{tikzcd}[row sep=large, column sep = large, every label/.append style ={font = \small}] 
		K(x \ract h , h\inv \lact y\op) \times K(y \ract h, h\inv \lact z\op)
		\ar[r, "U_{y\ract h}"]\ar[d, "\Psi_{h}\times \Psi_{h}"']&
		K(x\ract h, h\inv \lact z\op)\ar[d, "\Psi_{h}"]
		\\
		K(x, y\op) \times K(y, z\op)
		\ar[r, "U_{y}"]&
		K(x,z\op)
	\end{tikzcd}
\end{equation}

\begin{proof}[Proof of Lemma~\ref{lem:U maps}]
	The existence of $U_{y}$ follows from observations of general Hilbert $\cst$-modules similar to those following Proposition~\ref{RW:Morita:Prop3.8}:
	If $\mbf{X,Y,Z}$ are full right-Hilbert $\cst$-modules over a $\cst$-algebra $A$, then the maps
 \pagebreak[3]
	\begin{align}
		\label{eq:Yop otimes Y is trivial}
		(\mbf{X}\otimes_{A}\mbf{Y}\op)
		&\otimes_{\compacts_{A}(\mbf{Y})}
		(\mbf{Y}\otimes_{A}\mbf{Z}\op)
		&\longrightarrow& &
		\mbf{X}\otimes_{A}\mbf{Z}\op
		&&\longrightarrow& &
		\compacts_{A}(\mbf{Z},\mbf{X}) 
		\intertext{determined by}
		\notag
		(\mbf{x}\otimes \mbf{y}_{1}\op)&\otimes(\mbf{y}_{2}\otimes \mbf{z}\op)
		&\longmapsto&&
		(\mbf{x}\cdot \rinner[\mbf{Y}]{A}{\mbf{y}_{1}}{\mbf{y}_{2}}) \otimes \mbf{z}\op
		&&\longmapsto&&
		\innercpct[]{}{
			\mbf{x}\cdot \rinner[\mbf{Y}]{A}{\mbf{y}_{1}}{\mbf{y}_{2}}
		}{\mbf{z}}
	\end{align}
	are isomorphisms of bi-Hilbert $\compacts_{A}(\mathbf{X})-\compacts_{A}(\mathbf{Z})$-modules. The above isomorphism can be precomposed with the universal bilinear map $\mbf{X}\otimes_{A}\mbf{Y}\op
	\times
	\mbf{Y}\otimes_{A}\mbf{Z}\op
	\to 
	(\mbf{X}\otimes_{A}\mbf{Y}\op)
	\otimes_{\compacts_{A}(\mbf{Y})}
	(\mbf{Y}\otimes_{A}\mbf{Z}\op)$. 	
	We apply this to the case where $\mbf{Y}=M(y)$ and $A=B(u)$. Condition~\ref{item:U:linear} is therefore by construction.
	
	\ref{item:U:norm} 
	Since the maps in \eqref{eq:Yop otimes Y is trivial} are isomorphisms of modules (in particular, they are isometric), we have 
	\(
		\norm{U_{y}(\xi,\eta)}=\norm{\xi\otimes\eta}\) for $\xi\in K(x,y\op)$ and $\eta\in K(y,z\op)$, which is known (or easily shown) to be 
        bounded by $\norm{\xi}\norm{\eta}$. In particular, $U_{y}$ is not just linear but also continuous on each fibre.
		For the second claim, it therefore suffices to consider one of the dense elements $\xi= \sum_{i}m_{i}\otimes n_{i}\op$ in which case $\Flip(\xi) = \sum_{j}n_{j}\otimes m_{j}\op$ and
      \begin{align}\label{eq:U for xi Flip-xi}
      U(\xi,\Flip(\xi))
      &=
      \sum_{i,j}
      (m_{i} \ractB \rip\cB<n_{i},n_{j}>)\otimes m_{j}\op.
      \end{align}
      Again, since the maps in \eqref{eq:Yop otimes Y is trivial} are isometric, we see that
      \begin{align*}
      \norm{U(\xi,\Flip(\xi))}
      &=
      \norm{
      \sum_{i,j}
      \innercpct[]{}{ m_{i} \ractB \rip\cB<n_{i},n_{j}>}{ m_{j}}
      },
      \end{align*}
      which equals $\norm{\xi}^2$ by Equation~\eqref{eq:norm on compacts}.
 
\ref{item:U:cH equivariant} The properties of $\ractB$ and of $\rip\cB<\mvisiblespace,\mvisiblespace>$ imply that
	\[
	(m \ractB b)\ractB \rip\cB<n_{1},n_{2}>
	=
	m \ractB(b \rip\cB<n_{1},n_{2}>)
	=
	m  \ractB \rip\cB<n_{1} \ractB b^*,n_{2}>
	\]
	whenever $ m ,n_{1},n_{2}\in M$ and $b\in B$ are chosen such that the left (and hence each)  side of the above equations makes sense. Likewise, we have
	\(
	\rip\cB<n_{1},n_{2} \ractB b> =\rip\cB<n_{1},n_{2}> b.\) For
	$h\in u\cH$,
	suppose we are given elements
	$b_{i}\in B(h)$, $ m \in M(x),n_{1} \in M(y \ract h ), n_{2} \in M(y), k\in M( z \ract h )$. Then
	\begin{align*}
		&(\Psi_{h}\circ U_{y\ract h})\bigl(( m \ractB b_{1})\otimes n_{1}, (n_{2}\ractB b_{2})\otimes  k \op\bigr) 
		\\
		&=
		\Psi_{h}\bigl(( m \ractB b_{1})\ractB \rip\cB<n_{1},n_{2}\ractB b_{2}>\otimes k\op\bigr)
		=
		\Psi_{h}\bigl([m \ractB (b_{1}\rip\cB<n_{1},n_{2}>)]\ractB b_{2}\otimes k\op\bigr)
				\\
				&
		=
		\bigl(m\ractB \rip\cB<n_{1}\ractB b_{1}^*,n_{2}>\bigr)\otimes ( b_{2} \lactB k\op)
		\\
		&=
		U_{y}\bigl( m \otimes (b_{1}\lactB n_{1}\op),n_{2}\otimes (b_{2}\lactB  k \op)\bigr) 
		\\
		&
		=
		(U_{y}\circ \Psi_{h}\times \Psi_{h})\bigl(( m \ractB b_{1})\otimes n_{1},  (n_{2}\ractB b_{2})\otimes  k \op\bigr) 
		.
	\end{align*}
    Because of continuity and (bi)linearity of $\Psi_{h}$ and  of $U_{y},U_{y\ract h}$, we conclude that $\Psi_{h}\circ U_{y\ract h}=U_{y}\circ \Psi_{h}\times \Psi_{h}$, as claimed. One likewise checks that \ref{item:U:Flip} holds for for elementary tensors, and again uses (bi)linearity and continuity of $U_{y}$ and of $\Flip$ to deduce the claim.
\end{proof}

Similarly to how we constructed $\Psi$ in Lemma~\ref{lem:DL:MJM2023:Thm5.20} out of the maps $\Psi_{h}$ from Lemma~\ref{lem:Psi_h}, we now want to `stich together' the maps $\Phi_{x,y}$ on the one hand and the maps $U_{y}$ on the other hand, to bundle maps. To do so, we first need a definition.

    \begin{definition}
    		For two \uscBb s $\cM=(M\to X)$ and $\cN=(N\to Y)$, consider the product bundle $\cM\times\cN = (M\times N\to X\times Y)$; the norm on its fibre over $(x,y)$ can be chosen as the maximum of the norms of $M(x)$ and $N(y)$, and it is given the component-wise vector space structure. The global topology of the total space $M\times N$
    		of $\cM\times\cN$ is induced by the $\mathbb{C}$-linear span of sections of the form
    		\(
    		(x,y) \mapsto (\mu (x), \nu(y)),
    		\)
    		where $\mu$ and $\nu$ are continuous sections of~$\cM$ respectively $\cN$. 
    		
    		If 
    		$f\colon X\to Z$ and $g\colon Y\to Z$ are continuous functions into some other topological space,  then we write $\cM\bfp{f}{g}\cN$ for the restriction of $\cM\times\cN$ to the closed subset $X\bfp{f}{g}Y$ of the base $X\times Y$.
    \end{definition}

We now let 
    $\forwards,\backwards\colon X\bfp{\sigma}{\sigma}X\op\to X$ be given by 
     $\forwards(x, y\op)=x$ respectively 
    $\backwards(x, y\op)=y$ (``$\forwards$'' for ``first'', ``$\backwards$'' for ``second'').
\begin{lemma}
	   Writing $\Phi_{(x,y\op)}$ for the map $ \Phi_{x,y}$ of Lemma~\ref{lem:left action on fibres}, the map
	\begin{align*}
            \Phi\colon\quad 
		      \cK \bfp{\backwards}{q}  \cM
			&\to
			\cM,
			&
			(\xi, m)&\mapsto \Phi_{q(\xi)}(\xi,m)
            ,
			\intertext{is {\em bilinear}, jointly continuous, and covers the continuous surjection}
				(X \bfp{\sigma}{\sigma} X\op)\bfp{\backwards}{\operatorname{id}} X
        		&\to
        		X,
				&
				(x,y\op,y)& \mapsto x.
	\end{align*}
\end{lemma}

Note that $\Phi$ is not a homomorphism of Banach bundles, since $\Phi$ is not fibrewise linear.

\begin{proof} Since  $\Phi_{x,y}$ lands in $M(x)$ by construction,  $\Phi$ covers $\forwards\colon (x,y\op,y) \mapsto x$.
	To see that $\Phi$ is jointly continuous, assume that we are given a convergent net $(\xi_{\lambda},k_{\lambda})\to (\xi, k)$ in $K \bfp{\backwards}{q}  M$, and let $q (\xi_{\lambda},k_{\lambda}) = (x_{\lambda},y_{\lambda}\op, y_{\lambda})$ and $q (\xi,k) = (x,y\op,y)$; note that  $x_{\lambda}\to x$ and $y_{\lambda}\to y$ in $X$. To prove that $\Phi(\xi_{\lambda},k_{\lambda})\to \Phi(\xi, k)$, fix $\delta>0$; we must find $\kappa\in \Gamma_0(X;\cM)$ with $\norm{ \Phi(\xi, k) - \kappa(x)}<\delta$ and $\norm{\Phi(\xi_{\lambda},k_{\lambda}) - \kappa(x_{\lambda})}<\delta$  for large $\lambda$ \cite[Lemma A.3]{DL:MJM2023}.

    Because of how the topology on~$\cK$ is defined and because $\xi_{\lambda}\to\xi$ in~$\cK$, we can find 
    finitely many nets 
    $m_{j,\lambda}\to m_{j}$ and $n_{j,\lambda} \to n_{j}$ in $M$
    such that
    \begin{equation}\label{eq:close via tau}
        \norm{\xi - \sum_{j=1}^{\ell} m_{j}\otimes n_{j}\op} < \frac{\delta}{2(\norm{k}+1)} ,
        \quad
        \norm{\xi_{\lambda} - \sum_{j=1}^{\ell} m_{j,\lambda}\otimes n_{j,\lambda}\op} < \frac{\delta}{2(\norm{k}+1)}
        .
    \end{equation}
    Since $\rip\cB<\mvisiblespace,\mvisiblespace>$ is jointly continuous, this implies that for each $j$, $\rip\cB<n_{j,\lambda},k_{\lambda}> \to \rip\cB<n_{j},k>$ in $B$. Since $\mvisiblespace\ractB\mvisiblespace$ is jointly continuous, this in turn implies that $\sum_{j=1}^{\ell} 
        m_{j,\lambda} \ractB \rip\cB<n_{j,\lambda},k_{\lambda}> \to \sum_{j=1}^{\ell} 
        m_{j} \ractB \rip\cB<n_{j},k>$ in $M$. That means that we can find a section $\kappa$ of~$\cM$ with 
    \begin{equation}\label{eq:close via kappa}
        \norm{\sum_{j=1}^{\ell} 
        m_{j} \ractB \rip\cB<n_{j},k> - \kappa(x)}<\frac{\delta}{2}
        ,\quad
        \norm{\sum_{j=1}^{\ell} 
        m_{j,\lambda} \ractB \rip\cB<n_{j,\lambda},k_{\lambda}> - \kappa (x_{\lambda})} < \frac{\delta}{2}
    \end{equation}
    for all $\lambda \geq \lambda_0$. 
    Let $\lambda$ be large enough such that we also have $|\norm{k_{\lambda}}-\norm{k}| < 1.$ Linearity of $\Phi$ in the first component yields:
       \begin{align*}
        \norm{
            \Phi(\xi,k)
            -\kappa(x)
        }
        &\leq
        \norm{
            \Phi(\xi,k)
            -
            \sum_{j=1}^{\ell} 
            m_{j} \ractB \rip\cB<n_{j},k>
        }
        +
        \norm{
            \sum_{j=1}^{\ell} 
            m_{j} \ractB \rip\cB<n_{j},k>
            -
            \kappa(x)
        }
        \\
        &
        \overset{\eqref{eq:close via kappa}}{\leq}
        \norm{
            \Phi\bigl(
            \xi-
            \sum_{j=1}^{\ell} 
            m_{j} \otimes n_{j}\op
            ,k\bigr)
        }
        +
        \frac{\delta}{2}
        \\
        &
        \overset{\eqref{eq:Phi:contractive}}{\leq} 
        \norm{  
            \xi
            -
            \sum_{j=1}^{\ell} 
            m_{j} \otimes n_{j}\op
        }
        \norm{k}
        +
        \frac{\delta}{2}
        \overset{\eqref{eq:close via tau}}{<}
       \delta,
    \end{align*}
    and the exact same computation with subscript-$\lambda$'s yields $\norm{
            \Phi(\xi_{\lambda},k_{\lambda})
            -\kappa(x_{\lambda})
        } < \delta$ for large $\lambda$, which proves the claim.
\end{proof}

We proceed with stitching together the $U_{y}$'s:

\begin{lemma}\label{lem:U}
	The map
	\begin{align*}
		U\colon\quad 
		\cK\bfp{\backwards}{\forwards}\cK
		&\to
		\cK,
		&
		(\xi,\eta)&\mapsto U_{\backwards(\xi)}
		(\xi,\eta)
		,
		\intertext{is {\em bilinear}, jointly continuous, and covers the continuous surjection}
				(X\times X\op)\bfp{\backwards}{\forwards}(X\times X\op)
			&\to
			X\times X\op,
			&
			(x,y\op,y,z\op)& \mapsto (x,z\op).
		\end{align*}
		Moreover, given $\mu,\xi,\eta\in K$ and $m\in M$ with $(\mu,\xi),(\xi,\eta)\in 	\cK\bfp{\backwards}{\forwards}\cK$ and $(\xi,m)\in \cK\bfp{\backwards}{q}\cM$, we have
		\begin{align}\label{eq:U:associative, transitive}
			\Phi(U(\mu, \xi), m)
			&=
			\Phi(\mu, \Phi(\xi, m) )
			&&\text{and thus}&&&
			U\bigl(U (\mu , \xi), \eta\bigr)
			&=
			U\bigl(\mu , U (\xi, \eta)\bigr)
			.
		\end{align}
\end{lemma}

\begin{proof}
To see that $U$ is jointly continuous, assume that  $\{(\xi_{\lambda},\eta_{\lambda})\}_{\lambda}$ is a net in $\cK\bfp{\backwards}{\forwards}\cK$ that converges to $(\xi,\eta)$. Let $(x_{\lambda},y_{\lambda}\op)\coloneqq q_{\cK}(\xi_{\lambda})$, which converges to $(x,y\op)\coloneqq q_{\cK}(\xi)$, and let $(y_{\lambda},z_{\lambda}\op)\coloneqq q_{\cK}(\eta_{\lambda})$, which converges to $(y,z\op)\coloneqq q_{\cK}(\eta)$.
Fix an arbitrary $\delta\in (0,1)$; by \cite[Lemma A.3]{DL:MJM2023},  we should find a section $\kappa\in \Gamma_0(X\bfp{\sigma}{\sigma}X\op;\cK)$ with 
\[
\norm{ U\bigl( \xi , \eta \bigr) - \kappa(x,z\op)}<4\delta
\text{ and }
\norm{    U\bigl( \xi_{\lambda} , \eta_{\lambda} \bigr)
	- \kappa(x_{\lambda},z_{\lambda}\op)}<4\delta
\]
for large $\lambda$.
Because of how the topology on~$\cK$ is defined and because $\xi_{\lambda}\to\xi$, $\eta_{\lambda}\to\eta$ in~$\cK$,
there exist elements
\[
\zeta_{\lambda} \in
M(x_{\lambda})\odot M(y_{\lambda})\op
\text{ converging to }
\zeta \in
M(x)\odot M(y)\op
\text{ in } K
\]
and elements
\[
\chi_{\lambda} \in
M(y_{\lambda})\odot M(z_{\lambda})\op
\text{ converging to }
\chi \in
M(y)\odot M(z)\op
\text{ in } K
\]
such that
\begin{align}\label{eq:all four close}
	\norm{\xi - \zeta} ,
	\norm{\xi_{\lambda} - \zeta_{\lambda}}  < \frac{\delta}{\norm{\eta}+1},
	&&\text{ and }&
	\norm{\eta - \chi},
	\norm{\eta_{\lambda} - \chi_{\lambda}} < \frac{\delta}{\norm{\xi}+1}
\end{align}
for large $\lambda$. Here, it is important to point out that $\odot$ refers to the {\em algebraic} tensor product, so each of the elements considered above is a finite sum of elementary tensors. 
Note that
\begin{equation}\label{eq:U in terms of Phi}
	U(m\otimes n\op, k\otimes \ell\op)
	=
	\Phi
	(m\otimes n\op, k) \otimes \ell\op,
\end{equation}
so using bilinearity of $U$ and continuity of $\Phi$, the fact that $\zeta_{\lambda}\to \zeta$ and $\chi_{\lambda}\to \chi$ in $K$ and that addition and scalar multiplication is continuous on $\cK$, imply that 
\[
U( \zeta_{\lambda}, \chi_{\lambda})
\to
U(\zeta,\chi)
\]
for these particular sums of elementary tensors $ \zeta_{\lambda}, \chi_{\lambda}$.
This convergence means that we can find a section $\kappa$ of~$\cK$ with 
\begin{equation}\label{eq:U and kappa}
	\norm{
		U( \zeta_{\lambda}, \chi_{\lambda})
		- \kappa(x_{\lambda},z_{\lambda}\op)}<\delta
	\text{ and }
	\norm{
		U( \zeta,\chi)
		- \kappa(x,z\op)}<\delta.
\end{equation}
Combining this with Condition~\ref{item:U:norm}, we conclude for large $\lambda$:    
\begin{align*}
	\norm{
		U\bigl( \xi , \eta \bigr) - \kappa(x,z\op)
	}
	&\leq
	\norm{
		U\bigl( \xi , [\eta-\chi]\bigr)
	}
	+
	\norm{
		U\bigl( [\xi-\zeta] , \chi\bigr)
	}
	+
	\norm{
		U\bigl(\zeta,\chi\bigr)
		-
		\kappa(x,z\op)
	}\\
	&
	<
	\norm{
		\xi
	}\,
	\norm{\eta-\chi}
	+
	\norm{\xi-\zeta}
	\,\norm{\chi}
	+\delta
	&&\text{by \eqref{eq:U and kappa}, \ref{item:U:norm}}\\
	&<
	\delta
	+
	\frac{\delta}{\norm{\eta}+1}
	\,\frac{\delta}{\norm{\xi}+1} 
	+\delta
	\leq 3\delta
	&&\text{by \eqref{eq:all four close}}
	.
\end{align*}
We likewise get for $\lambda$ large enough such that $|\norm{\xi}-\norm{\xi_{\lambda}}|$ and $ |\norm{\eta}-\norm{\eta_{\lambda}}|$ are bounded by $1$,
\begin{align*}
	\norm{
		U\bigl( \xi_{\lambda} , \eta_{\lambda} \bigr) - \kappa(x_{\lambda},z_{\lambda}\op)
	}
	&<\norm{
		\xi_{\lambda}
	}\,
	\norm{\eta_{\lambda}-\chi_{\lambda}}
	+
	\norm{\xi_{\lambda}-\zeta_{\lambda}}
	\,\norm{\chi_{\lambda}}
	+\delta
	\\
	&\overset{\eqref{eq:all four close}}{<}
	(\norm{
		\xi
	}+1)\,
	\frac{\delta}{\norm{\xi}+1}
	+
	\frac{\delta}{\norm{\eta}+1}
	\,\norm{\chi_{\lambda}}
	+\delta.
\end{align*}
Note that
\begin{align*}
	\norm{\chi_{\lambda}}
	\leq
	\norm{\chi_{\lambda} - \eta_{\lambda}} 
	+ \norm{\eta_{\lambda}}
	\leq
	\norm{\chi_{\lambda} - \eta_{\lambda}} 
	+ \norm{\eta}+1
	\leq
	\frac{\delta}{\norm{\xi}+1}
	+ \norm{\eta}+1        
\end{align*}
so that
\begin{align*}
	\frac{\delta}{\norm{\eta}+1}
	\,\norm{\chi_{\lambda}}
	\leq
	\frac{\delta}{\norm{\eta}+1}
	\,
	\left[\frac{\delta}{\norm{\xi}+1}
	+ \norm{\eta}+1\right]
	\leq 2\delta.
\end{align*}
We conclude that
\begin{align*}
	\norm{
		U( \xi_{\lambda} , \eta_{\lambda} ) - \kappa(x_{\lambda},z_{\lambda}\op)
	}
	<
	(\norm{
		\xi
	}+1)\,
	\frac{\delta}{\norm{\xi}+1}
	+
	2\delta
	+\delta
	<
	4\delta,
\end{align*}
as needed.

For \eqref{eq:U:associative, transitive}, it likewise suffices to prove the claims for elementary tensors $\mu=\ell\otimes m_{1}\op$, $\xi=m_{2}\otimes n_{1}\op$, and $\eta=n_{2}\otimes k\op$. The computations make use of $\cB$-linearity of $\rip\cB<\mvisiblespace,\mvisiblespace>$ and are left to the reader.\qedhere
\end{proof}

For the next result, we remind the reader that we denote the groupoid $X\times_{\cH} X\op = (X\bfp{\sigma}{\sigma}X\op)/\cH$ by~$\cG$ and that we write $r_{\cA}\coloneqq r_{\cG}\circ p_{\cA}$ and $s_{\cA}\coloneqq s_{\cG}\circ p_{\cA}$.
\begin{lemma}\label{lem:cA:mult}
   On the quotient bundle $\cA=(p_{\cA}\colon A\to \cG)$, the following map is well defined, continuous, and fibrewise bilinear:
   \begin{align*}
   \mvisiblespace\cdot\mvisiblespace\colon\qquad
     \cA\comp \coloneqq 
     \left\{
        ([\xi],[\eta])\in A\times A  : s_{\cA}([\xi])=r_{\cA}([\eta])
     \right\}
    \to
    A,
\end{align*}given by\begin{align*}
    [\xi]\cdot[\eta]
    \coloneqq
    \Bigl[
        U\bigl( \xi , \Psi_{h}(\eta)\bigr)
    \Bigr],
   \end{align*}
   where $U$ is as defined in Lemma~\ref{lem:U} and $h\in\cH$ is the unique element such that  $\backwards_{\cK}(\xi) = \forwards_{\cK}\bigl(\Psi_{h}(\eta)\bigr)$.
\end{lemma}

\begin{remark}\label{rmk:structure on fibres of A}
    If $(g_{1},g_{2})\in \cG\comp$, we may choose representatives $(x,y\op)$ and $(y\op,z)$ of $g_{1}$ respectively $g_{2}$. If we then identify the fibres $A(g_{1}), A(g_{2})$  with $K(x,y\op)=M(x)\otimes_{\sigma(x)}M(y)\op$ and $K(y,z\op)$ via \eqref{eq:cA's fibres}, then the map $\mvisiblespace\cdot\mvisiblespace$ is fibrewise  on dense spanning sets given by 
    \begin{align}\label{eq:cA:cdot}
     \begin{split}
        \mvisiblespace\cdot\mvisiblespace\colon\qquad 
        A(g_{1})\times A(g_{2})
        &\to 
        K(x,z\op)
        \cong
        A(g_1g_2)
        \\
        \bigl(
            m_1\otimes n_1\op,
            m_2\otimes n_2\op
        \bigr)
        &\mapsto
            (m_1 \ractB \rinner[\cM]{\cB}{n_1}{m_2})\otimes n_2\op
     \end{split}
    \end{align}
\end{remark}

\begin{proof}[Proof of Lemma~\ref{lem:cA:mult}]
 Since $
         U_{h}\circ (\Psi_{h}\times \Psi_{h})=\Psi_{h}\circ U_{y\ract h}$, the map $\mf{m}\coloneqq \mvisiblespace\cdot\mvisiblespace$ is well defined. Since $U_{y}$ and $\Psi_{h}$ are linear, linearity of $\mf{m}$ follows directly from the definition of the fibrewise Banach space structure on $\cA$.
         
To see that it is continuous, we will  invoke openness of the quotient map $Q$: Suppose $\{\chi_{\lambda}\}_{\lambda}$ is a net in $\cA\comp$ which converges to $\chi$. Since it suffices to show that a {\em subnet} of $\{\mf{m}(\chi_{\lambda})\}_{\lambda}$ converges to $\mf{m}(\chi)$, we can without loss of generality assume that the {\em entire} net $\{\chi_{\lambda}\}_{\lambda}$ lifts to a net $\{(\xi_{\lambda}, \eta_{\lambda})\}_{\lambda}$ in $K\times K$ that converges to a lift $(\xi,\eta)$ of $\chi$ (here, we have made use of \cite[Proposition 1.1 (Fell's criterion)]{Wil2019} for the open map $Q$).
    Denote $q_{\cK}(\xi_{\lambda})=(x_{\lambda}, y_{\lambda}\op)$.
    Since $\chi_{\lambda}\in \cA\comp$, we may let $h_{\lambda},h\in\cH$ be the unique elements such that $q_{\cK}\bigl(\Psi_{h_{\lambda}}(\eta_{\lambda})\bigr)=(y_{\lambda}, z_{\lambda}\op)$ and $q_{\cK}\bigl(\Psi_{h}(\eta)\bigr)=(y, z\op)$ for some $z_{\lambda},z\in X$. Since $\xi_{\lambda}\to \xi$, it follows that $y_{\lambda}\to y$. Since $\eta_{\lambda}\to\eta$ and $q_{\cK}\bigl(\Psi_{h_{\lambda}}(\eta_{\lambda})\bigr)=(y_{\lambda}, z_{\lambda}\op)$, we further have that $(y_{\lambda}\ract h_{\lambda}, h_{\lambda}\inv \lact z_{\lambda}\op)\to q_{\cK}(\eta) = (y\ract h, h\inv \lact z\op) $. Since both $y_{\lambda}$ and $y_{\lambda}\ract h_{\lambda}$ converge, it follows from properness and freeness of the~$\cH$-action on~$X$ that a subnet of $\{h_{\lambda}\}_{\lambda}$ converges to $h$; again, without loss of generality we can assume that the entire net converges. By continuity of $\Psi$, we conclude that $ \Psi_{h_{\lambda}}(\eta_{\lambda})\to \Psi_{h} (\eta) $.
    Since $U$ is jointly continuous by Lemma~\ref{lem:U}, we therefore have
    \[
     U \bigl( \xi_{\lambda} ,  \Psi_{h_{\lambda}}(\eta_{\lambda}) \bigr)
     \longrightarrow
     U\bigl( \xi , \Psi_{h} (\eta)  \bigr)
    \]
    in~$K$, which suffices since $Q$ is continuous. 
\end{proof}

We arrive at our first main result:

\begin{theorem}\label{thm:A is FB}
   With respect to the multiplication in Lemma~\ref{lem:cA:mult} and the involution in Corollary~\ref{cor:cA:involution}, the \uscBb~$\cA$ described in Lemma~\ref{lem:cA is uscBb} is a saturated Fell bundle over the groupoid $\cG=X\times_{\cH} X\op$, called the {\em \iFbdl} of~$\cM$. 
\end{theorem}

\begin{proof}
    We will check the conditions as stated in \cite[Definition 2.9]{DL:MJM2023}.
    It is clear that  multiplication $\mvisiblespace\cdot\mvisiblespace$ is bilinear and that $\mvisiblespace^{*}$ is  conjugate linear  and self-inverse; this takes care of Conditions (F2), (F6), and (F8). For the following, fix $([\xi],[\eta])\in \cA\comp$, and let $h\in\cH$ be the unique element such that $q_{\cK}(\xi)=(x,y\op)$ and $q_{\cK}(\Psi_{h}(\eta))=(y,z\op)$, meaning that $q_{\cK}(\eta)=(y\ract h,h\inv \lact z\op)$.
    
    (F1) Since $U_{y}\bigl( \xi , \Psi_{h}(\eta)\bigr)\in K(x,z\op)$ and since $p_{\cA}([\xi])=[q_{\cK}(\xi)]$, we get the first and last equality in the following computation:
    \[
    p_{\cA} ([\xi]\cdot[\eta])
    =
    [x,z\op]
    =
    [x,y\op]\cdot [y,z\op]
    =
     p_{\cA} ([\xi])\cdot p_{\cA}([\eta]).
    \]

    (F3) Associativity of the multiplication
    can be shown using uniqueness of $h$, the second identity in \eqref{eq:U:associative, transitive}, commutativity of Diagram~\ref{diag:U}, and the fact that $\Psi_{hh'}=\Psi_{h}\circ\Psi_{h'}$; the details are left as an exercise.
    
    (F4) The definition of the norm on the fibres of $\cA$ implies that
    \[
        \| [\xi]\cdot[\eta]
        \|
        =
        \left\| U_{y}\bigl( \xi , \Psi_{h}(\eta)\bigr) \right\|
        \overset{\ref{item:U:norm}}{\leq}
        \left\| \xi \right\|\, \left\| \Psi_{h}(\eta) \right\|
        \overset{\ref{item:Psi:isometric}}{=}
        \left\| \xi \right\| \, \left\|\eta \right\|
        =
        \left\| [\xi] \right\| \, \left\|[\eta] \right\|,
    \]
    proving  
    that    multiplication is norm-decreasing. 

    (F5) 
    Since $p_{\cA}([\xi]^*)= [q_{\cK}(\Flip(\xi))]$ and since $[x,y\op]\inv=[y,x\op]$ in $\cG$, we have that $\mvisiblespace^{*}$ maps $A(g)$ to $A(g\inv)$. 

    (F7) We must show that $([\xi]\cdot[\eta])^*=[\eta]^*\cdot[\xi]^*$. We have
    \begin{align*}
        ([\xi]\cdot[\eta])^*
        &=
        \left[
        \Flip\left(
        U_{y}\bigl( \xi , \Psi_{h}(\eta)\bigr)
        \right)
        \right]
        \\
        &=
        \left[
        U_{y} 
        \Bigl(
          \Flip(\Psi_{h}(\eta)),\Flip(\xi)
        \Bigr)
        \right]
        &&\text{by \ref{item:U:Flip}}
        \\
        &=
        \left[
        U_{y} 
        \Bigl(
          \Psi_{h}(\Flip(\eta)),\Flip(\xi)
        \Bigr)
        \right]
        &&\text{by commutativity of Diagram~\eqref{diag:Flip and Psi}}
         \\
        &=
        \left[
        U_{y} 
        \Bigl(
          \Flip(\eta),\Psi_{h\inv}(\Flip(\xi))
        \Bigr)
        \right]
        &&\text{by
        \ref{item:U:cH equivariant} and definition of $\FBRel$},
    \end{align*}
    which is exactly $[\eta]^*\cdot[\xi]^*$.

    (F9) It is clear that $\|[\xi]\| = \|[\xi]^*\|$. To show $\|[\xi]\cdot[\xi]^*\| =\|[\xi]\|^2$, note that $\eta\coloneqq \Flip(\xi)$ has
    $q_{\cK}(\eta)=(y,x\op)$; in particular, $u\coloneqq \sigma(x)$ is the unique element of~$\cH$ with $\backwards_{\cK}(\xi) = \forwards_{\cK}\bigl(\Psi_{u}(\eta)\bigr)$. Since $\Psi_{u}$ is the identity, we conclude that
    \[
    	[\xi]\cdot[\xi]^*
    	=
    	\bigl[ U_{y} \bigl(\xi, \Flip(\xi))\bigr],
    \]
    which has norm equal to $\norm{\xi}^2=\norm{[\xi]}^2$ by \ref{item:U:norm}, as needed.

    (F10) To see that $[\xi]\cdot [\xi]^* \geq 0$, we 
    again invoke \cite[Lemma 2.65]{RaWi:Morita}: The matrix with $i,j$-entry $\rip\cB<n_{i},n_{j}>$ is positive, so there exists a matrix $(b_{i,j})_{i,j}$ over $B(u)$ such that $\rip\cB<n_{i},n_{j}> = \sum\nolimits_{l=1} b_{il}b_{jl}^*$. In particular, the $B(u)$-balancing in $M(x)\otimes_{u}M(y)\op$ allows us to write
    \begin{align*}
        \sum_{i,j}
      (m_{i} \ractB \rip\cB<n_{i},n_{j}>)\otimes m_{j}\op
      =
        \sum_{i,j,l}
      (m_{i} \ractB  (b_{il}b_{jl}^*))\otimes m_{j}\op
      =
        \sum_{l}
      (\sum_{i} m_{i} \ractB  b_{il})\otimes (\sum_{j} m_{j}\ractB b_{jl})\op.
    \end{align*}
    If we let $\mathbf{x}_{l}\coloneqq \sum_{i} m_{i} \ractB  b_{il}$, then this combined with \eqref{eq:U for xi Flip-xi} shows that
    \begin{align*}
        [\xi]\cdot [\xi]^*
        =
        \bigl[ U_{y} \bigl(\xi, \Flip(\xi))\bigr]
        =
        \sum_{l}
        \left[
          \mathbf{x}_{l}\otimes \mathbf{x}_{l}\op
        \right].
    \end{align*}
    Since the element $\mathbf{x}_{l}\otimes \mathbf{x}_{l}\op$ of $K(x,x\op)$ is positive (it corresponds to the positive operator $\innercpct[]{}{\mathbf{x}_{l}}{\mathbf{x}_{l}}$ in $\compacts_{B(u)}(M(x))$), and since the algebra $A([x,x\op])$ is *-isomorphic to $K(x,x\op)$ (see Remark~\ref{rmk:structure on fibres of A}), the claim follows.

    Lastly, to see that $\cA$ is saturated, we must show that the linear span of elements of the form $[\xi]\cdot [\eta]$ for $[\xi]\in A([x,y\op])$ and $[\eta]\in A([y,z\op])$ is dense in $A([x,z\op])$. When we consider that $A([x,y\op])\cong M(x)\otimes_{u}M(y)\op$ and that these isomorphisms respect the multiplicative and linear structure we defined on $\cA$, then the claim follows directly from the  first isomorphism in Equation~\eqref{eq:Yop otimes Y is trivial}. 
\end{proof}

\section{Equivalence from the \iFbdl\ to the `coefficient bundle'}\label{sec:Equivalence}

We now proceed to equip the  \demiequiv[right $\cB$-]~$\cM$ with the structure of a  \demiequiv[left $\cA$-], where $\cA=(p_{\cA}\colon A\to \cG)$ is the quotient of $\cK=(K\to X\bfp{\sigma}{\sigma}X\op)$ by the equivalence relation $\FBRel$ defined in Lemma~\ref{lem:FBRel};  we have shown in Theorem~\ref{thm:A is FB} that $\cA$ is a Fell bundle.

\begin{proposition}\label{prop:left action}
    There is a map $\mvisiblespace\lactB\mvisiblespace\colon A\bfp{s}{\rho}M \to M$ given by
    \begin{align}\label{eq:def lactB}
        [\xi]\lactB m \coloneqq \Phi_{x,y}\bigl(\Psi_{h}(\xi), m\bigr) 
    \end{align}
    where $y=q_{\cM}(m)$ and where $(x,h)\in X\bfp{\sigma}{r}\cH$ is the unique element such that $q_{\cK}(\Psi_{h}(\xi))=(x,y\op)$. This map furthermore has the following properties.\footnote{The LA in ``(LAn)'' stands for ``left action.''}
    \begin{enumerate}[label=\textup{(LA\arabic*)}]
        \item\label{item:lactB covers} It covers the map $\mvisiblespace\lact\mvisiblespace\colon \cG\bfp{s}{\rho}X \to X$; and
        \item \label{item:lactB and ractB commute}
        $[\xi]\lactB (m\ractB b)=([\xi]\lactB m)\ractB b$
        for all appropriate $b\in B$.
    \end{enumerate}
\end{proposition}

Of course, $\mvisiblespace\lactB\mvisiblespace$ will turn out to be an action in the sense of \cite[Definition 2.10]{DL:MJM2023}; in particular, it is continuous, which one can indeed show by hand. But thanks to Proposition~\ref{prop:from two-sided demi to equivalence}, we can make do with proving fewer properties.

\begin{proof}[Proof of Proposition~\ref{prop:left action}]
    Write $p_{\cA}([\xi])=[x_{1},x_{2}\op]$. Since  $([\xi], m)\in A\bfp{s}{\rho}M$, we have $s_{\cG}([x_{1},x_{2}\op])=\rho_{\cM} (m)=[y,y\op]$, which means that there exists a unique $h\in \cH$ such that
    $[x_{1},x_{2}\op]=[x_{1}\ract h, y\op]$. Now that the second component is fixed as $y\op$, the first component $x\coloneqq x_{1}\ract h$ is also uniquely determined, and satisfies $\sigma(x)=s_{\cH}(h)$. In other words, we have argued that the equality $s_{\cA}([\xi])=\rho_{\cM} (m)$ implies that there exists a {\em unique} representative $\Psi_{h}(\xi)\in K$ of the $\FBRel$-equivalence class $[\xi]$ for which $q_{\cK}(\Psi_{h}(\xi))$ has $q_{\cM}(m)\op$ as its second component: $\Psi_{h}(\xi)\in M(x)\otimes_{\sigma(x)} M(y\op)$. 
     We can now invoke Lemma~\ref{lem:left action on fibres} to conclude that $[\xi]\lactB m$ is a well-defined element of $M(x)$.

    \ref{item:lactB covers} Follows from $p_{\cA}([\xi])\lact q_{\cM}(m) = [x_{1},x_{2}\op] \lact y = [x,y\op] \lact y = x = q_{\cM}([\xi]\lactB m)$.

    \ref{item:lactB and ractB commute} We compute
    \begin{align*}
        \Phi_{x,y}\bigl(m \otimes n \op, k\bigr) \ractB b
        =
        \left( m \ractB \rip\cB< n ,k> \right) \ractB b
        =
        m \ractB \left( \rip\cB< n ,k> b\right)
        =
        m \ractB \rip\cB< n ,k\ractB b>
        =
        \Phi_{x,y}\bigl(m \otimes n \op, k\ractB b\bigr),
    \end{align*}
    so that bilinearity and continuity of $\Phi$ implies
    \[
         \Phi_{x,y}\bigl(\eta, k\bigr) \ractB b
        =
        \Phi_{x,y}\bigl(\eta, k\ractB b\bigr)
    \]
    for any appropriate $\eta\in K$. The claim follows.
\end{proof}

\begin{proposition}\label{prop:left ip}
    The map $\lip\cA<\mvisiblespace,\mvisiblespace>\colon M\bfp{\sigma}{\sigma}M\to A$  given by
    \[
        \lip\cA<m ,n > \coloneqq 
        [m \otimes n\op]
    \]
   has the following properties.\footnote{The LIP in ``(LIPn)'' stands for ``left inner product.''}
    \begin{enumerate}[label=\textup{(LIP\arabic*)}]
        \item\label{item:cA:FE2a} It covers the map $\leoqempty[X]{\cG}\colon X\bfp{\sigma}{\sigma}X \to \cG$, meaning that \( p_{\cA}(\lip\cA<m ,n > )
        \lact q_{\cM}(n )=q_{\cM}(m )\);
        \item\label{item:cA:cts and sesquilinear}  it is continuous and fibrewise linear in the first and antilinear in the second component;
        \item \label{item:cA:FE2b} $\lip\cA<m ,n >^* =\lip\cA<n ,m > $;
        \item\label{item:cA:FE2c}  $
        [\xi] \cdot
        \lip\cA<m ,n > 
        =\lip\cA<[\xi]\lactB m ,n >  $ for all $[\xi]\in A$ with $s_{\cA}([\xi])=\rho_{\cM}(m )$; and
        \item\label{item:cA:FE2d}  $\lip\cA<m_{1},m_{2}> \lactB m_{3} = m_{1}\ractB \rip\cB<m_{2},m_{3}>$.
    \end{enumerate}
\end{proposition}

\begin{proof}
    The map is clearly well defined. Property~\ref{item:cA:FE2a}  follows immediately from Equation~\eqref{eq:left G action on X}.

    \ref{item:cA:cts and sesquilinear} Sesquilinearity is obvious. For continuity, it suffices to check that the map $M\bfp{\sigma}{\sigma}M\to K$, $(m,n)\mapsto m\otimes n\op$, is continuous since the quotient map $\cK\to \cA$ is continuous. But said continuity is built into the definition of $\cK$'s topology; see Lemma~\ref{lem:DL:MJM2023:Lemma 5.2}.
    
    \ref{item:cA:FE2b} Follows from the definition of the involution ${}^*$ on~$\cA$. 
    
    \ref{item:cA:FE2c} Recall from Proposition~\ref   {prop:left action} that
    \begin{align*}
            [\xi]\lactB m_{1}= \Phi_{x,y}\bigl(\Psi_{h}(\xi), m_{1}\bigr) 
    \end{align*}
    where $y=q_{\cM}(m_{1})$ and where $(x,h)\in X\bfp{\sigma}{r}\cH$ is the unique element such that $q_{\cK}(\Psi_{h}(\xi))=(x,y\op)$, so that
    \begin{align*}
        \lip\cA<[\xi]\lactB m_{1},m_{2}>
        &=
        \left[
            \Phi_{x,y}\bigl(\Psi_{h}(\xi), m_{1}\bigr)
            \otimes
            m_{2}\op
        \right]
        \\
        &=
        \left[
            U\bigl(\Psi_{h}(\xi), m_{1}\otimes 
            m_{2}\op\bigr)
        \right]
        &&\text{by Equation~\eqref{eq:U in terms of Phi}}
        \\
        &=
        [\xi] \cdot [m_{1}\otimes 
            m_{2}\op]
        =
        [\xi] \cdot \lip\cA<m_{1},m_{2}>,
    \end{align*}
    so \ref{item:cA:FE2c} holds.

    \ref{item:cA:FE2d} Let $y=q_{\cM}(m_{3})$ and let $h\in\cH$ be the unique element such that $q_{\cK}(\Psi_{h}(m_{1}\otimes m_{2}\op))=(x,y\op)$ for some (unique) $x\in X$. Because of \eqref{eq:q of Psi_h}, we know that $q_{\cM}(m_{1})=x\ract h$ and $q_{\cM}(m_{2})=y\ract h$. In particular can approximate $m_{1}$ by $\sum_{i} n_{i}\ractB b_{i}$ for finitely many $n_{i}\in M(x)$ and $b_{i}\in B(h)$.
    In the following final computation, the first instance of $\approx$ is valid because $\Phi_{x,y}$ and $\Psi_{h}$ are continuous and linear (see \ref{item:Psi:isometric}), and the second   because the right $\cB$-action on~$\cM$ is continuous by \ref{item:rwordBdl:ractB:cts}.
    \begin{align*}
        \lip\cA<m_{1},m_{2}> \lactB m_{3}
        &=
        [m_{1}\otimes m_{2}\op] \lactB m_{3}
        \overset{\eqref{eq:def lactB}}{=}
        \Phi_{x,y}\bigl(\Psi_{h}(m_{1}\otimes m_{2}\op), m_{3}\bigr)
        \\&
        \approx
        \sum\nolimits_{i}\Phi_{x,y}\bigl(\Psi_{h}([n_{i}\ractB b_{i}]\otimes m_{2}\op), m_{3}\bigr)
        \\&
        \overset{\eqref{eq:def:Psi_h}}{=}
        \sum\nolimits_{i}
        \Phi_{x,y}\bigl(n_{i}\otimes (m_{2}\ractB b_{i}^*)\op, m_{3}\bigr)
        \\
        &\overset{\eqref{eq:def:Phi_xy}}{=}
        \sum\nolimits_{i}
        n_{i}\ractB \rip\cB<m_{2}\ractB b_{i}^*, m_{3}>
        = \sum\nolimits_{i}
        n_{i}\ractB (b_{i}\, \rip\cB<m_{2}, m_{3}>)
        \approx 
        m_{1}\ractB \rip\cB<m_{2}, m_{3}>.
        \qedhere
    \end{align*}
\end{proof}

\begin{corollary}\label{cor:iFbl is equivalent to cB}
    The saturated Fell bundle $\cA$ is equivalent to~$\cB$ via~$\cM$ with the $\cA$-action and $\cA$-inner product specified in Propositions~\ref{prop:left action} and~\ref{prop:left ip}, respectively. In particular, the left action is continuous.
\end{corollary}

\fix     we will frequently denote the \iFbdl\ $\cA$ of~$\cM$ by  $\cA=\cM\otimes_{\cB}\cM\op$ in analogy with $X\times_{\cH} X\op$ and with $\mbf{X}\otimes_{A}\mbf{X}\op$. However, the notation $\cA=\compacts(\cM_\cB)$ also lends itself well and has the advantage that it makes explicit reference to the underlying bundle~$\cB$ (and the fact that it is acting on the right-hand side).

\begin{proof}[Proof of Corollary~\ref{cor:iFbl is equivalent to cB}]
    We first check that~$\cM$ is a \demiequiv[left $\cA$-]. In the following lists, all references are either to Proposition~\ref{prop:left action} or to Proposition~\ref{prop:left ip}.
    \begin{itemize}[leftmargin=2cm]
      \item[\ref{item:rwordBdl:ractB:fibre}] was shown in \ref{item:lactB covers};
      \item[\ref{item:rwordBdl:ip:fibre}] was shown in \ref{item:cA:FE2a};
      \item[\ref{item:rwordBdl:ip}] was shown in \ref{item:cA:cts and sesquilinear};
       \item[\ref{item:rwordBdl:ip:C*linear}] 
          was shown in \ref{item:cA:FE2c}; and
       \item[\ref{item:rwordBdl:ip:adjoint}]
          was shown in \ref{item:cA:FE2b}.
  \end{itemize}
  For the remaining conditions necessary for a \demiequiv\ that we have not checked in any previous lemma, note that each $M(x)$ is a full right-Hilbert $\cst$-module by \ref{item:rwordBdl:M(x) full}, and so it is an \ib\ between $M(x)\otimes_{u} M(x)\op$ and $B(u)$ for $u\coloneqq \sigma(x)$. All structure with which we equipped the bundle $\cA$ is compatible with the homeomorphism $A(\rho(x))\cong M(x)\otimes_{u} M(x)\op$ in Lemma~\ref{lem:Q| is homeo}; in other words, with the restriction to $M(x)$ of the given $\cA$-action and $\cA$-inner product on $\cM$, we recover exactly the structure that $M(x)$ naturally carries. In particular:
\begin{itemize}[leftmargin=2cm]
      \item[\ref{item:rwordBdl:positive}:]  We have $\lip\cA<m,m>\geq 0$ in the $\cst$-algebra $A(\rho_{\cM}(m))$, and $\lip\cA<m,m>=0$ only if $m=0$.
       \item[\ref{item:rwordBdl:norm compatible}:] By \cite[Lemma 2.30.]{RaWi:Morita}, the norm of $\rip\cB<m,m>$ in $B(\sigma(x))$ coincides with the norm of $m\otimes m\op$ in $M(x)\otimes_{u}M(x)\op$ which, by definition, is exactly the norm of $[m\otimes m\op]=\lip\cA<m,m>$ in $A(\rho(x))$. Since~$\cM$ is a \demiequiv[right $\cB$-], the norm $m\mapsto \norm{\rip\cB<m,m>}^{1/2}$ agrees with the norm that the \uscBb~$\cM$ already carries, which is therefore in turn identical to the norm $m\mapsto \norm{\lip\cA<m,m>}^{1/2}$.
       \item[\ref{item:rwordBdl:fibrewise full}:]  For each $x\in X$, the linear span of $\{\lip\cA<m_{1},m_{2}>:m_{i}\in M(x)\}$ is dense in $A(\rho(x))$, since $M(x)$ is an \ib.
\end{itemize}
That $A(\rho(x))\cong M(x)\otimes_{u}M(x)\op$ further implies that 
Assumption~\ref{item:demi to full:ractB and lactB by adjointables} of Proposition~\ref{prop:from two-sided demi to equivalence} is satisfied. Since Assumptions~\ref{item:demi to full:ractB and lactB commute} and~\ref{item:demi to full:ips compatible} were shown in~\ref{item:lactB and ractB commute} and~\ref{item:cA:FE2d}, respectively, we can conclude that~$\cM$ is an equivalence by Proposition~\ref{prop:from two-sided demi to equivalence}.
\end{proof}

\section{The \iFbdl: Uniqueness}\label{sec:Uniqueness}

For the next result, we remind the reader of Remark~\ref{rmk:leoq}: given a $(\cG,\cH)$-groupoid equivalence~$X$ and two elements $x,y\in X$ with the same~$\cH$-anchor, we write $\leoq[X]{\cG}{x}{y}$ for the unique element of~$\cG$ that satisfies $\leoq[X]{\cG}{x}{y}\lact y=x $.

\begin{proposition}[cf.\ {\cite[Corollary 3.11]{AF:EquivFb}}]\label{prop:iFbdl is unique}
    Suppose $\cA =(A\to \cG)$, $\cA' =(A'\to \cG')$, and $\cB=(B\to \cH)$ are Fell bundles over groupoids $\cG,\cG',\cH$ respectively, and that $\cM=(q_{\cM}\colon M\to X)$ is both an $(\cA,\cB)$- and an $(\cA', \cB)$-Fell bundle equivalence. Then there exists a unique Fell bundle isomorphism $\Omega\colon \cA\to\cA'$ determined by the following commutative diagram.
    \[
      \begin{tikzcd}[row sep=small]
        \cA \ar[ddd, "p_{\cA}"']\ar[rrr, "\Omega"]&&& \cA'\ar[ddd, "p_{\cA'}"]
        \\
          & \linner[\cM]{\cA}{m}{n} \ar[r, mapsto]\ar[d, mapsto] & \linner[\cM]{\cA'}{m}{n}\ar[d, mapsto]&
          \\
          &\leoq[X]{\cG}{q_{\cM}(m)}{q_{\cM}(n)}\ar[r, mapsto] & \leoq[X]{\cG'}{q_{\cM}(m)}{q_{\cM}(n)}&
        \\
        \cG \ar[rrr, "\omega"']&&& \cG' 
      \end{tikzcd}  
    \]
\end{proposition}

\begin{proof}[Proof of Proposition~\ref{prop:iFbdl is unique}]
    By Corollary~\ref{cor:iFbl is equivalent to cB}, we may without loss of generality assume that $\cA=\cM\otimes_{\cB}\cM\op$ and $\cG=X\times_{\cH}X\op$. In that setting, $\leoq[X]{\cG}{q_{\cM}(m)}{q_{\cM}(n)}$ is exactly $[q_{\cM}(m),q_{\cM}(n)] \in \cG$; see Lemma~\ref{lem:thm for gpds}. We also remind the reader that $\omega\colon X\times_{\cH} X\to \cG', [x,y\op]\mapsto \leoq[X]{\cG'}{x}{y}$, is known to be an  isomorphism of topological groupoids. It is therefore clear that $\cA'\cong \omega^*(\cA')$, and so we may without loss of generality assume that $\cG'=\cG$ (and hence $\omega=\operatorname{id}$ and $\cA' = \omega^*(\cA')$). As before, we will denote the anchor maps of~$X$ by  $\sigma\colon X\to \cH\z$ and $\rho\colon X\to \cG\z, x\mapsto [x,x\op]$. Recall that $\cG\z\cong X/\cH$ via $[x,x\op]\mapsto x\ract \cH$.

    To construct $\Omega$, we will first  construct a surjective bundle map 
    \(
    \tilde{\Omega} \colon \cK \to \cA'
    \) covering the quotient map $X\bfp{\sigma}{\sigma}X\op\to \cG, (x,y\op)\mapsto [x,y\op]$. We will then show that 
    $\tilde{\Omega}$ is constant on $\FBRel$-equivalence classes
    (where $\FBRel$ is as in Lemma~\ref{lem:FBRel}),  so that $\tilde{\Omega}$ induces a bundle map $\Omega$ from the quotient $\cM\otimes_{\cB}\cM\op = Q(\cK)$ to $\cA'$ covering the identity map $\cG\to\cG$. After we have shown that $\Omega$ is injective, we will then prove the conditions in \cite[Propositions A.7 and A.8]{DL:MJM2023} to  deduce that $\Omega$ is open and continuous and hence the claimed isomorphism of Fell bundles.

    We define $\tilde{\Omega}$ fibrewise. Fix $u\in \cH\z$ and $x,y\in Xu$, let $g\coloneqq [x,y\op]$, and consider
    \begin{equation}\label{eq:fibrewise on Cartesian}
    M(x)\times M(y)\op \to \cA'_{g},
    \quad
    (m, n\op) \mapsto \linner[\cM]{\cA'}{m}{n}.
    \end{equation}
    By the assumption on the left $\cA'$-valued inner product on~$\cM$, this map is bilinear. Moreover,  by \cite[Corollary 4.6]{DL:MJM2023}\footnote{Note that we cannot simply invoke the assumption that each $M(x)$ is an \ib\ (\cite[Definition 6.1(c)]{MW:2008:Disintegration}), because we are in the setting where $x$ might not equal $y$.}, we have $\linner[\cM]{\cA'}{m \ractB b}{n} = \linner[\cM]{\cA'}{m }{n\ractB b^*} $ for all $b\in B(u)$, meaning that the map in \eqref{eq:fibrewise on Cartesian} descends to a linear map 
    \[
    \tilde{\Omega}_{(x,y\op)}\colon M(x)\odot_{B(u)}M(y)\op \to A'(g),
    \quad\text{determined by}\quad
    \tilde{\Omega}_{(x,y\op)}(m\odot n\op) = \linner[\cM]{\cA'}{m}{n},
    \]
    where $M(x)\odot_{B(u)}M(y)\op$ denotes the balanced algebraic tensor product. 

       We claim that $\tilde{\Omega}_{(x,y\op)}$ has dense range. Since $\cA'$ is a Fell bundle, $A'(g)$ is an $A'(x\ract \cH)-A'(y\ract \cH)$-\ib, so by \cite[Proposition 2.33 (Hewitt-Cohen)]{RaWi:Morita},  any element of $ A'(g)$ can be written as  $a_{g} a_{y\ract \cH}$ for some $a_{g}\in  A'(g)$ and some  $a_{y\ract \cH}\in A'(y\ract \cH)$, where the product is  the multiplication in the Fell bundle~$\cA'$. By assumption, $M(y)$ is a $A'(y\ract \cH)-B(u)$-\ib. In particular, $\linner[\cM]{\cA'}{M(y)}{M(y)}$ is dense in $A'(y\ract \cH)$, meaning we can approximate $a_{y\ract \cH}$ by linear combinations of elements of the form $\linner[\cM]{\cA'}{n'}{n}$ for $n',n\in M(y)$. Since $a_{g}\linner[\cM]{\cA'}{n'}{n}=\linner[\cM]{\cA'}{a_{g}\lactB n'}{n}$ and $a_{g}\lactB n' \in M(g\lact y)=M(x)$, we conclude that the arbitrary element $a_{g} a_{y\ract \cH}$  of $ A'(g)$ can be approximated by linear combinations of elements of the form $\linner[\cM]{\cA'}{m}{n}$ for $m\in M(x),n\in M(y)$, as claimed.

    Since~$\cM$ is an equivalence, $M(x)$ is an \ib\ between $A'(\rho(x))=A'(x\ract \cH)$ and $B(\sigma(x))=B(u)$, so that there exists a canonical isomorphism $A'(x\ract \cH)\cong \compacts_{B(u)}(M(x))$ of $\cst$-algebras. In particular, the norm with which  we equip $M(x)\odot_{B(u)}M(y)\op$, is unambiguously given by Equation~\eqref{eq:norm on compacts} and, using the $*$-isomorphism,  can be rewritten to
    \begin{equation*}
	\norm{ \sum_{i=1}^{k
		} m_{i}\otimes n_{i}\op}
	=
	\norm{
		\sum_{i,j=1}^{k
		} 
		\linner[\cM]{\cA'}{n_{i}\ractB \rinner[\cM]{\cB}{m_{i}}{m_{j}} }{ n_{j}}
	}^{1/2}.
	\end{equation*}
    Using the properties \cite[Def.\ 2.9, (F9)]{DL:MJM2023} and
    \cite[Def.\ 2.11, (FE2.b), (FE2.c)(FE2.d)]{DL:MJM2023}, we therefore see that 
    each $ \tilde{\Omega}_{(x,y\op)}$ is isometric and hence extends to a map on the completion $K(x,y\op)$ of $M(x)\odot_{B(u)}M(y)\op$. All in all, we get a fibrewise linear map
    \[
    \tilde{\Omega}\colon \quad
    K=\bigsqcup_{(x,y\op)\in X\bfp{\sigma}{\sigma}X\op} K(x,y\op) \to \bigsqcup_{g\in \cG}A'(g),
    \quad\text{determined by}\quad
    \tilde{\Omega}(m\otimes n\op) = \linner[\cM]{\cA'}{m}{n},
    \]
    covering the quotient map $X\bfp{\sigma}{\sigma}X\op\mapsto \cG$. 
    
    Next, fix $u,v\in \cH\z$, $h\in u\cH v$, and $x\in Xu$, $y\in Xv$. For  an arbitrary but fixed element $\xi=\sum_{i}(m_{i}\ractB b_{i})\odot n_{i}\op$  of $M(x\ract h)\odot_{B(u)}M(y)\op$, 
    we have
    \begin{align*}
    \tilde{\Omega}_{(x,y\op)}\left(  \Psi_{h}(\xi)\right) 
    &=
    \sum_{i}
    \linner[\cM]{\cA'}{m_{i}}{n_{i}\ractB b_{i}^*}
    &&\text{by definition of $\Psi$, see \eqref{eq:def:Psi_h}}
    \\
    &=
    \sum_{i}
    \linner[\cM]{\cA'}{m_{i}\ractB b_{i}}{n_{i}}
    =
     \tilde{\Omega}_{(x,y\op)}(\xi)
    &&\text{by \cite[Corollary 4.6]{DL:MJM2023}}.
    \end{align*}
    Since $\tilde{\Omega}_{(x,y\op)}(\Psi_{h}(\xi))=\tilde{\Omega}_{(x,y\op)}(\xi)$ for all $\xi$ in the algebraic tensor product, we conclude that  $\tilde{\Omega}$ descends to a bundle map
    \[
        \Omega\colon \quad \cM\otimes_{\cB}\cM\op \to \cA',
    \quad\text{determined by}\quad
    \Omega([m\otimes n\op]) = \linner[\cM]{\cA'}{m}{n},
    \]
    covering the identity map $\cG\to \cG$.
    We must now show that this map is a continuous open bijection.

  Since $\Omega$ covers the identity map, surjectivity of $\Omega$ follows from the fact that each $\tilde{\Omega}_{(x,y\op)}$ has dense range, and injectivity of $\Omega$ follows since each $\tilde{\Omega}_{(x,y\op)}$ is isometric. To see that $\Omega$ is continuous, suppose that we have a convergent net in $\cA=\cM\otimes_{\cB}\cM\op$, say $[\xi_{\lambda}]\to [\xi]$; by $(2)\implies (1)$ in \cite[Theorem 18.1.]{Munkres}, it suffices to show that a subnet of $\{\Omega([\xi_{\lambda}])\}_{\lambda}$ converges to $\Omega([\xi])$. Since the quotient map $Q\colon\cK\to \cA$ is open (Lemma~\ref{lem:FBRel}), we can without loss of generality assume that $\xi_{\lambda}\to\xi$ in~$\cK$. Let $q_{\cK}(\xi_{\lambda})=(x_{\lambda},y_{\lambda}\op)$ with limit $q_{\cK}(\xi)=(x,y)$ in $X\bfp{\sigma}{\sigma}X\op$. We now invoke the definition of the topology on~$\cK$ (Lemma~\ref{lem:DL:MJM2023:Lemma 5.2}; see also \cite[Lemma A.3]{DL:MJM2023}): for any given $\epsilon >0$, we can find finitely many continuous sections $\mu_{i},\nu_{i}$ of~$\cM$ and a $\lambda_0$ such that for all $\lambda\geq \lambda_0$, we have
  \begin{equation}\label{eq:approx in cK}
    \norm{\xi_{\lambda} - \sum_{i=1}^{k}\mu_{i}(x_{\lambda})\otimes \nu_{i}(y_{\lambda})\op} < \epsilon
    \text{ and }
    \norm{\xi - \sum_{i=1}^{k}\mu_{i}(x)\otimes \nu_{i}(y)\op}<\epsilon.
  \end{equation}
  The norm on each fibre $A([x,y\op])$ of~$\cA$ is defined as the norm on one of its lifts under $Q$, meaning that $Q$ restricted to  $M(x)\otimes_{u}M(y)\op$ is isometric. Since $\Omega$ is likewise isometric, \eqref{eq:approx in cK} implies
  \begin{align}\label{eq:Omega circ Q}
    \norm{\Omega([\xi_{\lambda}]) - \Omega\left(\left[\sum\nolimits_{i}\mu_{i}(x_{\lambda})\otimes \nu_{i}(y_{\lambda})\op\right]\right)} < \epsilon
    \text{ and }
    \norm{\Omega([\xi]) - \Omega\left(\left[\sum\nolimits_{i}\mu_{i}(x)\otimes \nu_{i}(y)\op\right]\right)}<\epsilon.
  \end{align}
  By definition of $\Omega$,
  \[
  \Omega\left(\left[\sum\nolimits_{i}\mu_{i}(x)\otimes \nu_{i}(y)\op\right]\right)
  =
  \sum\nolimits_{i}\linner[\cM]{\cA'}{\mu_{i}(x)}{\nu_{i}(y)}.
  \]
  By assumption on~$\cM$, $\linner[\cM]{\cA'}{\mvisiblespace}{\mvisiblespace}$ is continuous, so that continuity of $\mu_{i},\nu_{i}$ and of addition in~$\cA$ implies 
  \[
  \sum\nolimits_{i}\linner[\cM]{\cA'}{\mu_{i}(x_{\lambda})}{\nu_{i}(y_{\lambda})}
  \to 
  \sum\nolimits_{i}\linner[\cM]{\cA'}{\mu_{i}(x)}{\nu_{i}(y)}.
  \]
  Combining \eqref{eq:Omega circ Q} with Lemma~\ref{lem:convergence uscBb:close nets}, we conclude that $\Omega([\xi_{\lambda}])\to \Omega([\xi])$, so $\Omega$ is continuous. Lastly, since $\Omega$ is isometric, it follows from \cite[Proposition A.8]{DL:MJM2023} that $\Omega$ is open as well.
\end{proof}

This also finishes the proof of the main theorem, Theorem~\ref{thm:demiequiv to equiv}, since it is a combination of Corollary~\ref{cor:iFbl is equivalent to cB} and Proposition~\ref{prop:iFbdl is unique}.

\begin{remark} 
    A brief comment on another property related to equivalences that has a version not only in the realm of groupoids and $\cst$-algebras but also of Fell bundles: linking objects.    
    
    By \cite[Theorem 3.19.]{RaWi:Morita}, two $\cst$-algebras are \SMEadj\ if and only if they are complementary full corners of another algebra, called the {\em linking algebra}. In \cite[Theorem 4.1]{SW:2012:RsEquivalence}, this result was moved into groupoid-land: an equivalence of groupoids gives rise to a {\em linking groupoid}  whose $\cst$-algebra is the linking algebra that witnesses the \SMEnoun. Likewise, a Fell bundle equivalence 
    has a {\em linking (Fell) bundle} 
    (see \cite[Section 3]{SW:Equivalence}, \cite[Theorem 3.2]{AF:EquivFb}) whose $\cst$-algebra is the linking algebra of the Fell bundle $\cst$-algebras \cite[Theorem 14]{SW:Equivalence}.
    
    It would be interesting to study whether  `one-sided' versions of these linking objects exist. In the paper at hand, we considered \psp{groupoid}s  $X$, full right-Hilbert $\cst$-modules $\mathbf{X}$, and Fell bundle-\demiequiv s $\cM$, and turned them into equivalences between the coefficient object and the `generalized compacts', that is, the \igpd\ $X\times_\cH X\op$, the $\cst$-algebra of compact operators $\compacts(\mathbf{X}_A)$, and the \iFbdl\ $\cM\otimes_{\cB}\cM\op = \compacts(\cM_{\cB})$, respectively. Do these results have analogues in the language of linking objects?
\end{remark}

\section{Applications}\label{sec:Applications}
The original impetus for the paper at hand was Proposition A.2.3 in the appendix of \cite{MS:2023:stablerank}; it can easily seen to be a corollary of our main theorem in the special case that~$\cH$ is a group that acts freely and properly on a space $X$. Before we study other interesting applications, let us first do some sanity checks.
In analogy to the isomorphism of \igpd s in Example~\ref{ex:X=H}, we have:
\begin{corollary}\label{cor:ib of cB is cB}
    Let~$\cB$ be a Fell bundle, considered as a self-Fell bundle equivalence \cite[Example 6.6]{MW:2008:Disintegration}. Then there is an isomorphism of Fell bundles $\compacts(\cB_{\cB})=\cB\otimes_{\cB}\cB\op\cong \cB$ determined by $[b_{1}\otimes b_{2}\op]\mapsto b_{1}\cdot b_{2}^*$.
\end{corollary}

In analogy to the isomorphism of Hilbert $\cst$-bimodules in \eqref{eq:Yop otimes Y is trivial}, we have:
\begin{corollary}
    Suppose  that~$\cM$, $\cN$,~$\cK$ are~\demiequiv[right $\cB$-]s, and let $\cA=\compacts(\cN_{\cB})$ be the \iFbdl\ of $\cN$. Then $(\cM \otimes_{\cB}\cN\op)\otimes_{\cA} (\cN\otimes_{\cB} \cK\op)\cong \cM \otimes_{\cB}\cK\op$ as \uscBb s.
\end{corollary}
\begin{proof}[Proof sketch.]
    First, let us make sure that the objects we have written down are not nonsense: By Theorem~\ref{thm:demiequiv to equiv}, all three \demiequiv s are equivalences between their respective \iFbdl s and~$\cB$. As explained in \cite[Example 6.7]{MW:2008:Disintegration}, their opposites are then likewise equivalences, just in the other direction. It was shown in \cite{DL:MJM2023} that equivalences can be concatenated: $\cM \otimes_{\cB}\cN\op$ is an equivalence from the \iFbdl\ of~$\cM$ to~$\cA$, and so forth. 
    In particular, both the left- and the right-hand side of the alleged isomorphism is an equivalence between the \iFbdl\ of~$\cM$ and that of~$\cK$. 

    Since the  balanced tensor product of \ib\ is associative, it is easy to see that the same is true for $\otimes_{\cB}$ of Fell bundle equivalences. In particular, 
    \begin{align*}
        (\cM \otimes_{\cB}\cN\op)\otimes_{\cA} (\cN\otimes_{\cB} \cK\op)
        &
        \cong
        \bigl(\cM \otimes_{\cB}(\cN\op\otimes_{\cA} \cN)\bigr)\otimes_{\cB} \cK\op.
    \end{align*}
    By Proposition~\ref{prop:iFbdl is unique} (applied to $\cN\op$), the Fell bundle $\compacts(\cN\op_{\cA})=\cN\op\otimes_{\cA} \cN$ is  isomorphic to~$\cB$. Since the balanced tensor product of \ib\ absorbs the coefficient algebra (meaning that $\mbf{X}\otimes_{A} A\cong \mbf{X}$), it is easy to see that $\otimes_{\cB}$ absorbs the coefficient Fell bundle. Thus,
    \begin{align*}
        (\cM \otimes_{\cB}\cN\op)\otimes_{\cA} (\cN\otimes_{\cB} \cK\op)
        &
        \cong
        (\cM \otimes_{\cB} \cB )\otimes_{\cB} \cK\op
        \cong 
        \cM\otimes_{\cB}\cK\op,
    \end{align*}
    as claimed.
\end{proof}

We can also conclude some other, well-known results.
\begin{corollary}
    Suppose~$\cH$ is a \LCH\ groupoid with Haar system $\{\lambda_u\}_{u\in\cH\z}$ and~$\cB$ is a Fell bundle over~$\cH$. Then $\cst(\cB)$ is \SMEadj\ to $\mathrm{M}_n(\cst(\cB))$. 
\end{corollary}
Note that the right choice of~$\cB$ allows $\cst(\cB)$ to model the full $\cst$-algebras of groupoids (with or without a twist \'a la Kumjian) and full crossed product $\cst$-algebras. 

\begin{proof}
    Take $X=\cH$, with $\sigma\colon X\to \cH\z$ the source map of~$\cH$. The existence of a Haar system implies that the source map of~$\cH$ (and hence the anchor map of $X$) is open \cite[Proposition 1.23]{Wil2019}, so we are in good shape to use our main theorem, provided we can find the \demiequiv[right $\cB$-]. Let~$\cM$ to be the bundle over~$X$ with fibres $\mbb{C}^n\times \cB(h)$ for each $h \in\cH$. We define
    \begin{align*}
        \mvisiblespace\ractB \mvisiblespace\colon&
        \qquad
        \cM\bfp{\sigma}{r}\cB \to \cM,
        \quad
        (\vec{x},b)\ractB b' \coloneqq (\vec{x}, b\cdot b'),
    \intertext{and}
    \rinner[\cM]{\cB}{\mvisiblespace}{\mvisiblespace}\colon&
        \qquad
        \cM\bfp{\sigma}{\sigma}\cM \to \cB,
        \quad
        \rinner[\cM]{\cB}{(\vec{x},b)}{(\vec{y},c)}
        \coloneqq
        \rinner[\mbb{C}^n]{\mbb{C}}{\vec{x}}{\vec{y}} \, b^*\cdot c,
    \end{align*}
    where we choose the inner product on $\mbb{C}^n$ to be conjugate linear in the first coordinate to match up with our definitions of Hilbert $\cst$-modules.
    Then, using Corollary~\ref{cor:ib of cB is cB}, it is easy to see that
    \[
    \cM\otimes_{\cB}\cM\op \cong \mathrm{M}_n (\mbb{C})\times \cB
    \text{ via }
    [(\vec{e}_{i}, b_{1})\otimes (\vec{e}_{j}, b_{2})\op]
    \mapsto
    (E_{i,j},b_{1}\cdot b_{2}^*),
    \]
    where $\vec{e}_{i}$ is the standard basis vector of $\mbb{C}^n$ and $E_{i,j}$ is the matrix unit with a $1$ in the $i^\text{th}$ row and $j^\text{th}$ column and $0$s everywhere else. By Theorem~\ref{thm:demiequiv to equiv}, the bundles~$\cB$ and $\mathrm{M}_n (\mbb{C})\times \cB$ are equivalent. By 
    \cite[Theorem 6.4]{MW:2008:Disintegration}, it follows that $\cst(\cB)$ is \SMEadj\ to  $\cst(\mathrm{M}_n (\mbb{C})\times \cB)\cong\mathrm{M}_n(\cst(\cB))$.
\end{proof}

\subsection{Kumjian's Stabilization trick}

In this subsection, we will see that our main theorem implies \cite[Corollary 4.5]{Kum:1998:Fell}, which in its original form was only shown for principal $r$-discrete groupoids. The approach here is inspired by that preceding \cite[Theorem 15]{Muhly:Bdles}. We will go through it in great detail to help the reader understand our technical results.

We start with a (saturated) Fell bundle $\cB=(B\to \cH)$ over a \LCH\ groupoid~$\cH$ with Haar system  $\{\lambda_v\}_{v\in \cH\z}$. Fix $h\in \cH $, and let $v\coloneqq r(h)$ and $u\coloneqq s(h)$. We define $M_0(h)\coloneqq\Gamma_c (\cH u;\cB)$, sections of the restriction of~$\cB$ to $\cH u = s\inv (s(h))$. If $g\in \cH$ is another element with $r(g)=v$, then we define a sesquilinear form
\[
    \rip{B(g\inv h)}<\mvisiblespace,\mvisiblespace>\colon M_0(g)\times M_0(h) \to B(g\inv h)
\]
by
\begin{equation}\label{eq:rip for my version of cE}
    \rip{B(g\inv h)}<\mu,\xi> = \int_{\cH s(g)} \mu(\ell )^*  \xi(\ell g\inv  h) \diff \lambda_{s(g)}(\ell).
\end{equation}
Note that this indeed makes sense: if $\ell\in \cH s(g)$, then  $\ell$ is in the domain of $\mu$ and $\ell  g\inv $ is defined. Moreover, $s(\ell   g\inv h)=s(h)$, so $\ell   g\inv h$ is in the domain of $\xi$. Since  $\mu(\ell)^* \in B(\ell\inv )$ and $\xi(\ell   g\inv h)\in B(\ell g\inv h)$, their product is indeed an element of $B(\ell\inv   \ell  g\inv    h)=B(g\inv h)$. And since $\mu$ and $\xi$ are compactly supported, the integral exists.

For $h=g$, the form is valued in the $\cst$-algebra $B(u)$, so we may take the completion $M(h)$ of $M_0(h)$ with respect to the induced norm $$\norm{\xi}_{M(h)}\coloneqq \norm{ \rip{B(u)}<\xi,\xi>}^{1/2}.$$ 
If $k\in u\cH w$, so that $(h,k)\in \cH\comp$, then the fibre $B(k)$  can act on an element $\xi$ of $M_0(h)$ and deliver an element $\xi \ractB b$ of $M_0(hk)$: if $\ell$ is in $\cH w$ (the domain of any element of $M_0(hk)$), then $\ell k\inv\in \cH u$ is in the domain of $\xi$ and so
\[
    ( \xi \ractB b )(\ell)
    \coloneqq
    \xi(\ell k\inv)\cdot b
    \text{ is defined and an element of }
    B(\ell k\inv) \cdot B(k)
    \subseteq B(\ell).
\] 
This extends to a map $M(h)\times B(k)\to M(hk)$. In fact, it induces an isomorphism $M(h)\otimes_{u} B(k)\cong M(hk)$ of right-Hilbert $\cst$-$B(w)$-modules.

We let $\cM = (q_{\cM}\colon M\to \cH)$ be the bundle with fibres $M(h)$, which we will now topologize.
If $\tau\in \Gamma_c (\cH;\cB)$ is a section of~$\cB$ and $h\in \cH u$, is arbitrary then $\tilde{\tau}(h) \coloneqq \tau|_{\cH u}$ is a continuous, compactly supported section $\cH u \to \cB$; i.e., $\tilde{\tau}(h)$ is an element of $M(h)$, so $\tilde{\tau}$ is a section of the bundle $\cM$. The family
\(
    \{ \tilde{\tau} \colon \tau\in \Gamma_c (\cH;\cB)\}
\)
uniquely induces a topology on~$\cM$ making it \usc. The attentive reader will have expected what comes next: we will show that~$\cM$ is a \demiequiv[right $\cB$-] over the \psp{$\cH$} $X=\cH$ (Example~\ref{ex:X=H}).

\begin{remark}
Let us compare what we have done so far with what was done in \cite{Muhly:Bdles,Kum:1998:Fell}; this will also give us an indicator as to what will happen next. \citeauthor{Muhly:Bdles} constructs first an \uscBb\ $\cV=(V\to \cH\z)$, which is exactly $\cV=\cM|_{\cH\z}$ and which \citeauthor{Kum:1998:Fell} calls the ``Hilbert $\cB\z$-module bundle over $\cH\z$'' in \cite[Subsection 4.2]{Kum:1998:Fell}; he points out that $\cV$ is full.

From $\cV$, they then construct a bundle $\cE=(E\to \cH)$; in the notation of the paper at hand, $\cE$ can be constructed as the pullback bundle of
\begin{align}\label{eq:def of cE}
    \cV\otimes_{\cB\z}\cB=\left(\bigsqcup_{(v,h)}V(v)\otimes_{v}B(h) \to \cH\z \bfp{\operatorname{id}}{r}\cH\right)
\end{align}
along the isomorphism $\cH\cong \cH\z \bfp{\operatorname{id}}{r}\cH, h\mapsto (r(h),h)$.\footnote{The bundle $\cV\otimes_{\cB\z}\cB$ is constructed in the same fashion as the bundle $\cK=\cM\otimes_{\cB\z}\cM\op$ in Lemma~\ref{lem:DL:MJM2023:Lemma 5.2}; a more general construction of  bundles of the form $\cM\otimes_{\cB\z}\cN$ can be found in \cite{DL:MJM2023}.}

It is then shown that $\cE$ is an equivalence between~$\cB$ and  the semi-direct product Fell bundle of a certain $\cH$-action on another bundle  which \citeauthor{Kum:1998:Fell} and \citeauthor{Muhly:Bdles} denote  by $\mathcal{K}(\cV)$. Following \cite[Subsection 1.7]{Kum:1998:Fell}, $\mathcal{K}(\cV)$ is the $\cst$-algebraic bundle with fibre $\compacts (V(u)_{B(u)})$ over $u\in \cH\z$. It is apparent that $\mathcal{K}(\cV)$ coincides with our $\compacts(\cV_{\cB\z})=\cV\otimes_{\cB\z}\cV\op$, the \iFbdl\ of the \demiequiv[right $\cB\z$-] $\cV$.\footnote{To nitpick, \citeauthor{Kum:1998:Fell}'s $\mathcal{K}(\cV)$ is the pullback of $\compacts(\cV_{\cB\z})$ along the homeomorphism $\cH\z\to\cH\z\times_{\cH\z}(\cH\z)\op, u \mapsto [u,u\op]$.} By construction $\cE\cong\cM$ via the map $\xi\otimes b \mapsto \xi\ractB b$,
and so in light of our main theorem, we should expect that the \iFbdl\ $\compacts(\cM_\cB)$ of~$\cM$ is (isomorphic to) a semi-direct product of $\compacts(\cV_{\cB\z})$ by $\cH$.
\end{remark}

By Example~\ref{ex:X=H}, the anchor maps of $X=\cH$ are exactly the range and the source map once we identify its \igpd\ $\cG=X\times_{\cH}X\op$ with $\cH$. This means that the  $\cB$-valued inner product
\[
 \rip{\cB}<\mvisiblespace,\mvisiblespace>\colon M \bfp{r}{r}M \to B
\]
covers the map $\reoqempty[X]{\cH}$ (Condition~\ref{item:rwordBdl:ip:fibre}). Moreover, the map
\[
\mvisiblespace\ractB \mvisiblespace\colon 
M \bfp{s}{r} B \to M 
\]
covers the multiplication map of~$\cH$ (or, in other words, the map that describes the right $\cH$-action on $X=\cH$; Condition~\ref{item:rwordBdl:ractB:fibre}). Let us indicate how to check all other conditions for~$\cM$ to be a \demiequiv.

\ref{item:rwordBdl:ip} 
Follows from the properties of Haar systems (Property~(HS2) in \cite[Definition 1.19]{Wil2019}) and the definition of the topology on~$\cM$ as given above.

\ref{item:rwordBdl:ip:C*linear} Straight-forward computation. 

\ref{item:rwordBdl:ip:adjoint}
Follows from the properties of Haar systems (Property~(HS3) in \cite[Definition 1.19]{Wil2019}).

\ref{item:rwordBdl:positive} 
Since the integrand of $\rip\cB<\xi,\xi>$ takes positive values in $B(s(h))$, so does the inner product itself, and  $\rip\cB<\xi,\xi>=0$ implies $\xi(\ell)=0$ for all $\ell\in \cH s(h)$, which is the entire domain of $\xi$.

\ref{item:rwordBdl:norm compatible} Holds by definition of the Banach space structure on $M(h)$.

\ref{item:rwordBdl:fibrewise full} Note that $M(h)=M(u)$ as right-Hilbert $\cst$-$B(u)$-modules, so it suffices to point out that $M(u)= V(u)$ is known to be full since~$\cB$ is a saturated Fell bundle.

\bigskip

It follows from Theorem~\ref{thm:demiequiv to equiv} that~$\cB$ is equivalent to the \iFbdl\ $\cM\otimes_{\cB}\cM\op$ of~$\cM$. Let us describe $\cM\otimes_{\cB}\cM\op$ as a bundle $\cA=(A\to \cH)$ over $\cH$. There are multiple variants here, depending on our choice of section $\cH \to X\bfp{s}{s}X\op$ of the quotient map $X\bfp{s}{s}X\op \to X\times_\cH X\op\cong \cH$.

In view of $\cE$ as described in Equation~\eqref{eq:def of cE} and in view of the fact that $B(h)\cong B(h\inv)\op$ by the Fell bundle properties, let us declare\footnote{Notationally, it would have been nicer to choose $A(h)=M(h)\otimes_{u} M(u)\op$, but then our description would drift farther away from those in \cite{Kum:1998:Fell,Muhly:Bdles}.}
\[
    A(h) \coloneqq M(v)\otimes_{v} M(h\inv)\op \text{ for } h\in \cH^v_u.
\]
To describe the multiplication map, we must be a bit careful: if $k\in\cH^u_w$, then our choice of section turns the composable pair $(h,k)$ into $([v,(h\inv)\op], [u,(k\inv)\op])$ in the \igpd. But $h\inv$ and $u$ do not coincide, and so neither will the `inner' parts of $A(h)$ and $A(k)$:
\begin{align*}
    A(h)\times A(k)
    &=
    \left(
    M(v)\otimes_{v}
    M(h\inv)\op
    \right)
    \times
    \left(
    M(u)\otimes_{u} 
    M(k\inv)\op
    \right).
\end{align*}
On the level of groupoids, we must therefore replace (for example) the representative $(u,(k\inv)\op)$ by $(h\inv,(k\inv h\inv)\op)$.\footnote{Again, this choice is not very pleasant, but it is indeed the best in this scenario.}  On the level of $\cA$, this is being done by the $\Psi$-map from Lemma~\ref{lem:DL:MJM2023:Thm5.20}:

\[
\begin{tikzcd}
    M(u)\otimes_{u} 
    M(k\inv)\op
    \ar[d, "\cong"]
    \ar[ddd, rounded corners, to path={-- ([xshift=-4ex]\tikztostart.west)
-| ([xshift=-10ex]\tikztotarget.west)
-- (\tikztotarget)}, "\Psi_{h}"]
    &
    {[\xi \ractB  b]}\otimes \eta\op
    \\
    {[M(h\inv)\otimes_{v}B(h)]}\otimes_{u} 
    M(k\inv)\op
    \ar[d, "\cong"]
    &
    {[\xi \otimes b]}\otimes \eta\op
    \ar[u]
    \ar[d]    
    \\
    M(h\inv)\otimes_{v}{[M(k\inv)\otimes_{u} B(h\inv)]}\op
    \ar[d, "\cong"]
    &
    \xi \otimes {[\eta\otimes b^*]}\op
    \ar[d]    
    \\
    M(h\inv)\otimes_{v} 
    M(k\inv h\inv)\op
    &
    \xi \otimes {[\eta\ractB  b^*]}\op
\end{tikzcd}
\]
Having found compatible representatives, we can then use the $U$-map from Lemma~\ref{lem:U}:
\begin{align*}
    A(h)\times A(k)
    &\overset{\operatorname{id}\times \Psi_{h}}{\longrightarrow} 
    \left(
        M(v)\otimes_{v}
        M(h\inv)\op
    \right)
    \times
    \left(
        M(h\inv)\otimes_{v}
        M(k\inv h\inv)\op
    \right)
    &&\ni \phantom{\mapsto}
    (\xi\otimes\eta_{1}\op,\eta_{2}\otimes \zeta\op) 
    \\
    &\overset{U_{h\inv}}{\longrightarrow} 
     M(v)\otimes_{v}
     M(k\inv h\inv )\op
     =
     A(hk)
     &&
     \mapsto
    (\xi\ractB \rip\cB<\eta_{1},\eta_{2}>)\otimes \zeta\op 
    .
\end{align*}
We therefore see that, with the choice of $A(h)$ that we have made, the multiplication $A(h)\times A(k)\to A(hk)$ is given by $U_{h\inv}\circ (\operatorname{id}\times \Psi_{h})$. 
The involution on $A$ is easier to decipher: it is given by 
\[
    A(h)
    =
    M(v)\otimes_{v}
    M(h\inv)\op
    \overset{\Flip}{\longrightarrow}
    M(h\inv)\otimes_{v}
    M(v)\op
    \overset{\Psi_{h\inv}}{\longrightarrow}
    M(u)\otimes_{u}
    M(h)\op
    =
    A(h\inv)
    .
\]
Now, $M(h)$ and $M(s(h))$ are {\em equal} as right-Hilbert $\cst$-modules; the letter $h$ is merely needed to keep track of the `mixed' inner product in Equation~\eqref{eq:rip for my version of cE}.
This means that we have an isomorphism
\[
    \Psi_{h}\colon \qquad V(u)\otimes_u V(u)\op \longrightarrow V(v)\otimes_u V(v)\op.
\]
If we make the canonical identification of $V(u) \otimes_{u}V(u)\op$ with $\compacts \bigl(V(u)_{B(u)}\bigr)=\compacts(\cV_{\cB\z})(s(h))$, then the above can be written as 
\[
    h\lact \mvisiblespace\colon \qquad \compacts(\cV_{\cB\z})(s(h)) \longrightarrow \compacts(\cV_{\cB\z})(r(h)).
\]
In other words, $\Psi$ encodes an action of~$\cH$ on the bundle $\compacts(\cV_{\cB\z})$ which covers the map $\mvisiblespace\lact\mvisiblespace\colon h\lact s(h)=r(h)$. 
In this picture, the map $\Flip$ is just the adjoint:
\[
\begin{tikzcd}  
    \xi \otimes \eta\op 
    \ar[d, mapsto, "\Flip"']
    \ar[r, mapsfrom]&
    \innercpct{}{\xi}{\eta}
    \ar[r, phantom, "\in"]
    &
    \compacts(\cV_{\cB\z})(v)
\ar[d, "*"]
    \\
     \eta \otimes \xi\op
    \ar[r, mapsto]&
    \innercpct{}{\eta}{\xi}
    =
    \innercpct{}{\xi}{\eta}^*
    \ar[r, phantom, "\in"]
    &
    \compacts(\cV_{\cB\z})(v)
\end{tikzcd}
\]
and $U_{h\inv}$ is just juxtaposition of compact operators: 
\[
\begin{tikzcd}  
    (\xi \otimes \eta_{1}\op,\eta_{2}\otimes \zeta\op) 
    \ar[d, mapsto, "U_{h\inv}"']
    \ar[r, mapsfrom]&
    \bigl(\innercpct{}{\xi}{\eta_{1}}, \innercpct{}{\eta_{2}}{\zeta}\bigr)
    \ar[r, phantom, "\in"]
    &
    \compacts(\cV_{\cB\z})(v)
\times \compacts(\cV_{\cB\z})(v)
\ar[d, "\mvisiblespace\circ\mvisiblespace"]
    \\
    (\xi\ractB \rip\cB<\eta_{1},\eta_{2}>)\otimes \zeta\op     
    \ar[r, mapsto]&
    \innercpct[]{}{\xi\ractB \rip\cB<\eta_{1},\eta_{2}>}{\zeta}
    =
    \innercpct{}{\xi}{\eta_{1}}
    \circ \innercpct{}{\eta_{2}}{\zeta}\ar[r, phantom, "\in"]
    &
    \compacts(\cV_{\cB\z})(v)
\end{tikzcd}
\]
If we write 
\[
A(h)
= [V(v)\times\{v\}]\otimes_{v}[V(v)\times\{h\inv\}]\op.
=
[V(v) \otimes_{v}V(v)\op]\times\{h\}
\cong 
\compacts \bigl(V(v)_{B(v)}\bigr)\times\{h\},
\]
then 
as sets, $\cA=\compacts \bigl(\cV_{\cB\z}\bigr)\rtimes \cH$. Moreover, we found that the multiplication $A(h)\times A(k)\to A(hk)$ is given by $U_{h\inv}\circ (\operatorname{id}\times \Psi_{h})$ and the involution $A(h)\to A(h\inv)$ by $\Psi_{h\inv}\circ\Flip$, which we have shown translates to
\[
    (T_{1},h)\cdot (T_{2},k)
    =
    \bigl(T_{1}\circ (h\lact T_{2}),hk\bigr)
    \quad\text{and}\quad
    (T,h)^*
    =
    (T^*,h\inv),
\]
which are exactly the formulas in the semi-direct product.
We deduce that the \iFbdl\ $\cA$ of the $\cB$-action on~$\cM$ is exactly the Fell bundle $\compacts \bigl(\cV_{\cB\z}\bigr)\rtimes \cH$, showing that Theorem~\ref{thm:demiequiv to equiv} recovers\footnote{Of course, there is still the difficulty of having to come up with a suitable $\cM$ in the first place; we have taken it for granted here.} \citeauthor{Kum:1998:Fell}'s Stabilization trick.

\subsection{Higher order operators {\em \`a la} Abadie--Ferraro}

As alluded to earlier, a theorem similar to our main theorem has appeared in \cite{AF:EquivFb} in the setting that $\cH=X=\cG$ is a group. Before we can cite it, let us first establish the following bridge between their terminology and ours:
\begin{lemma}\label{lem:wordBdl agrees with AF's}
	Suppose~$\cB$ is a saturated Fell bundle over a \LCH\ {\em group} $\cH$. Then \demiequiv[$\cB$-]s {\em over the same group} are exactly  right Hilbert~$\cB$-bundles in the sense of \cite[Definition 2.1]{AF:EquivFb} that are {\em fibrewise full}.
\end{lemma}

\begin{proof}
    {\em A priori}, the norm of a \demiequiv\ $\cM=(M\to\cH)$ is only \usc. However, since~$\cB$ is a Fell bundle over a group, it follows from  \cite[Lemma 3.30]{BMZ:HigherCat} that $b\mapsto \norm{b}$ is continuous (not only \usc) on $\cB$, and since the norm on~$\cM$ is given by $\norm{\rip\cB<m,m>}^{1/2}$ by Assumption~\ref{item:rwordBdl:norm compatible}, it is the concatenation of continuous maps and hence itself continuous. Therefore, the \uscBb~$\cM$ is actually a {\em continuous}  Banach bundle in the sense of \cite[Definition II.13.4]{FellDoran:VolI}, as needed for \cite[Definition 2.1]{AF:EquivFb}.
    
    As in \cite{AF:EquivFb}, the inner product and the right-action are continuous by \ref{item:rwordBdl:ip} and \ref{item:rwordBdl:ractB:cts}, respectively. The assumption that the Hilbert bundle be fibrewise full corresponds to \ref{item:rwordBdl:fibrewise full} (and is stronger than the assumption (7R) in \cite{AF:EquivFb}). The remaining items of \cite[Definition 2.1]{AF:EquivFb} correspond to our assumptions as follows.
    \begin{enumerate}[label=\textup{(\arabic*R)}]
        \item corresponds  to \ref{item:rwordBdl:ractB:fibre} and \ref{item:rwordBdl:ip:fibre}, using that $\reoq[\cH]{h_{1}}{h_{2}}{\cH}=h_{1}\inv h_{2}$ as explained in Example~\ref{ex:X=H};
        \item corresponds  to~\ref{item:rwordBdl:ractB:bilinear};
        \item corresponds to \ref{item:rwordBdl:ip}; 
        \item corresponds  to~\ref{item:rwordBdl:ip:C*linear} and~\ref{item:rwordBdl:ip:adjoint};
        \item corresponds  to~\ref{item:rwordBdl:positive}; and
        \item corresponds to~\ref{item:rwordBdl:norm compatible}. \hfill\qedhere
    \end{enumerate}
\end{proof}

We can now restate the result of \citeauthor{AF:EquivFb} to which we want to compare ours, in our terminology. We will artificially add the assumption that the fibres are all full, to align with our situation. We will further use different letters for  elements of the group $\cH$: we use $x$ if we think of~$\cH$ as a \psp{$\cH$}; we use $h$ if we think of~$\cH$ as the group acting on the right, and $g$ if it is acting on the left; in particular, $x\ract h$ and $g\lact x$ will just be the products $xh$ respectively $gx$ in $\cH$.
\begin{theorem}[{\cite[Definition 3.6, Theorem 3.9, Corollary 3.10]{AF:EquivFb}}]\label{thm:AF}
    Let~$\cH$ be a \LCH\ group, $\cB=(B\to \cH)$ a Fell bundle, and $\cM=(M\to \cH)$ a right \demiequiv[$\cB$-], also over $\cH$. For $g\in \cH$, let $\mathbb{B}_g (\cM )$ be the collection of continuous maps $S \colon M  \to M $ such that
    \begin{itemize}
        \item there exists $c \in \mathbb{R}$ such that $\norm{Sm}\leq c\norm{m}$ for all $m \in M $;
        \item $S(M(x)) \subseteq M(g\lact x)$ for all $x \in \cH$; and
        \item  there exists $S^* \colon M \to M$ such that $\rip\cB<Sm_{1}, m_{2}> = \rip\cB<m_{1}, S^*m_{2}>$ for all $m_{i} \in M $.
    \end{itemize}    
    Then there exists a unique Fell bundle $\mathbb{K}(\cM )$ over~$\cH$ such
that:
\begin{enumerate}[label=\textup{(\roman*)}]
    \item for all $g\in \cH$, the fiber $\mathbb{K}(\cM )_{g}$ is, as a Banach space, the closure in $\mathbb{B}_{g}(\cM )$ of
\[
\mathrm{span}\{\innercpct[]{}{m}{n} : m \in M(g\lact x), n \in M(x), x \in  \cH\}
\]
where $\innercpct[]{}{m}{n} \colon M \to M, k \mapsto m\ractB \rip\cB<n,k>$;
\item given $\psi  , \varphi \in C_c(M)$ and $x \in \cH$, the function $[ \psi  , \varphi, x] \colon \cH \to \mathbb{K}(\cM )$ given by
$[ \psi  , \varphi, x](g) = [ \psi  (g\lact x), \varphi(x)]$, is a continuous section of $\mathbb{K}(\cM)$.
\end{enumerate}
Moreover, $\cM $ is a $\mathbb{K}(\cM )-\cB$-equivalence
bundle with the action $\mathbb{K}(\cM ) \times M \to M$ given by $(S, m) \to S(m)$ and the left inner
product $M \times M \to \mathbb{K}(\cM )$ given by $(m,n) \to \innercpct[]{}{m}{n}$.
\end{theorem}

\citeauthor{AF:EquivFb} call the elements of $\mathbb{B}_g (\cM )$ {\em adjointable operators of order $g$}. We conclude:

\begin{corollary}\label{cor:cf AF's K(M)}
	Suppose~$\cH$, $\cB$, and~$\cM$ are as in Theorem~\ref{thm:AF}. The map 
	\begin{align*}
		\cM\otimes_{\cB}\cM\op \to \compacts(\cM)
	\end{align*}
	which maps $[m\otimes n\op]$ in the fibre over $[x,y\op]$ of $\cM\otimes_{\cB}\cM\op$ to the operator
	\[
		\innercpct[]{}{m}{n}\colon \quad 
M \ni k \mapsto m\ractB \rinner[\cM]{\cB}{n}{k} \in M
	\]
	in the fibre over $xy\inv$ of $\compacts(\cM)$, is an isomorphism of Fell bundles, covering the map $f$ from Example~\ref{ex:X=H}.
\end{corollary}

\begin{proof}
	By Lemma~\ref{lem:wordBdl agrees with AF's},~$\cM$ is a right Hilbert~$\cB$-bundle in the sense of \cite[Definition 2.1]{AF:EquivFb}.
	 It is proven in \cite[Corollary 3.10]{AF:EquivFb} that~$\cM$ is an equivalence between~$\cB$ and $\compacts(\cM)$. Since~$\cM$ is also an equivalence between~$\cB$ and $\cM\otimes_{\cB}\cM\op$, it thus follows  from \cite[Corollary 3.11]{AF:EquivFb} (or equivalently, from Proposition~\ref{prop:iFbdl is unique}), that the displayed map is the claimed isomorphism.
\end{proof}

\begin{remark}
    As alluded to in Remark~\ref{rmk:first comparison to AF}, the description {\em \`a la} \citeauthor{AF:EquivFb} of the \iFbdl\ $\cM\otimes_{\cB}\cM\op$  as ``adjointable operators with a shift'' is also built into our construction, albeit less visibly so.
\end{remark}

Even when we consider Fell bundles over groups, Theorem~\ref{thm:demiequiv to equiv} can cover some examples that are not covered by \cite{AF:EquivFb}, since there, $X=\cH$ always. One example is   \cite[Proposition A.2.3]{MS:2023:stablerank}; another one is Example~\ref{ex:Paul's Ex 14}.

\subsection{Imprimitivity Theorems}\label{ssec:imprim}

Earlier, we alluded to a relationship between Theorem~\ref{thm:demiequiv to equiv} and imprimitivity theorems; let is give one example.

\begin{example}[{\cite[Example 14]{Muhly:Bdles}}]\label{ex:Paul's Ex 14}
    If~$X$ is a \LCH\ group with closed subgroup~$H$, then~$H$ acts on~$X$ by right translation ($x\ract h \coloneqq xh$)  and~$X$ is a \psp{$H$}. Its \igpd\
    \(
        X\times_{H} X\op
    \) is isomorphic to the transformation group groupoid $X\ltimes X/H$ of~$X$ acting on the quotient space $X/H$ via
    \begin{equation}\label{eq:Muhly's Ex14:G}
        X\times_{H} X\op \to X\ltimes X/H
        ,\quad
        [x,y\op] \mapsto (xy\inv, yH).
    \end{equation}
    In terms of $G\coloneqq X\ltimes X/H$, the 
    anchor map described in \eqref{eq:anchor map of left G action on X} becomes the quotient map,
    \[
    \rho\colon X\to G\z\cong X/H,
    \quad
    x\mapsto (1_{X}, xH)\triangleeq xH,
    \]
    and the left action on~$X$ as described in \eqref{eq:left G action on X}
    becomes the action
    $(x, yH)\lact y \coloneqq xy$.
    We conclude that $G$
    is equivalent  to the group $H$ via $X$. 
    To sum up,
    \begin{align*}
        X\op\bfp{\rho}{\rho} X \to H,\quad
        \reoq[X]{x}{y}{H} = x\inv y,
        \quad\text{ and }\quad
        X\times X\op \to G,\quad
        \leoq[X]{G}{x}{y} = 
        (xy\inv, yH).
    \end{align*}
    Note that $\reoqempty[X]{H}$ indeed lands in $H$ since $\rho(x)=\rho(y)$.

    Now suppose that $A$ is a $\cst$-algebra with action $\alpha\colon X\to \operatorname{Aut}(A)$.
    Let $\cB= A\rtimes H$ be the Fell bundle over $H$ that encodes the restriction $\alpha|_H$, meaning that, for each $h\in H$, the fibre $B(h)$ is $A$ and the structure maps of~$\cB$ are given by 
   \[
    (a_1,h_1)\cdot (a_2,h_2)\coloneqq (a_1\alpha_{h_{1}}(a_{2}), h_{1}h_{2})
        \quad\text{and}\quad
        (a,h)^*\coloneqq
        (\alpha_{h\inv}(a)^*, h\inv)
        .
    \]
    Let $\cM=A\times X$ be the ``trivial $A$-bundle'' over $X$.  Then~$\cB$ acts on~$\cM$ via
    \[
        M\times B\to M,
        \quad
        (m,x)\ractB (b, h)
        \coloneqq
        (m\alpha_{x}(b), xh),
    \]
    and we have a~$\cB$-valued inner product given by
    \[
        M\bfp{\rho}{\rho} M\to B,
        \quad
        \rinner{\cB}{(m,x)}{(n,y)}
        \coloneqq
        \bigl(
        \alpha_{x\inv}(m^*n), 
        x\inv y
        \bigr).
    \]
    One quickly checks that~$\cM$ is a~\demiequiv[$\cB$-], 
    so~$\cB$ is equivalent via~$\cM$ to $\cM\otimes_{\cB}\cM\op$. Let us identify the fibre of the latter Fell bundle over an element $(xy\inv,yH)=\leoq[X]{G}{x}{y}$ of $G$: As explained in Theorem~\ref{thm:A is FB}, we use \eqref{eq:cA's fibres} to identify 
    \begin{align*}
        \left(\cM\otimes_{\cB}\cM\op\right) (xy\inv,yH)
        \cong
        M(x)\otimes_{e} M(y)\op
    \end{align*}
    as $\compacts_{B(e)}(M(x))-\compacts_{B(e)}(M(y))$-\ib s.
    Over a unit $(1_{G},yH)$ of $G$, the fibre is (canonically isomorphic to) the $\cst$-algebra $\compacts_{B(e)}(M(y))$ itself, where $y$ is any representative of $yH$. These $B(e)$-compact operators on $M(y)$ are generated by $\innercpct[]{}{n_{1}}{n_{2}}$ for $n_{i}\in M(y)$, and one quickly computes that such a rank-one operator just multiplies $n\in M(y)$ on the left by $n_{1}n_{2}^*$.

The balancing in $M(x)\otimes_{e} M(y)\op$ identifies
    \[
    (m\alpha_x(b),x)\otimes (n,y)\op
    =
    (m,x) \otimes (n\alpha_{y}(b^*),y)\op
    \]
    for all $b\in B(e)=A$. If we let $A(xy\inv,yH)\coloneqq A$ as a Banach space, then we can therefore show that the map  
    \begin{align*}
        M(x)\otimes_{e} M(y)\op    
        &\to
        A(xy\inv,yH)
        \\
        (m,x)\otimes (n,y)\op
        &\mapsto
        m\alpha_{xy\inv}(n)^*
    \end{align*}
    is a Banach space isomorphism. 
   These isomorphisms create a new Fell bundle $\cA=(p_{\cA}\colon A\to G)$ out of $\cM\otimes_\cB\cM\op$:~$\cA$ has the fibre $A(xy\inv,yH)$ over $(xy\inv,yH)$, and the structure maps of $\cM\otimes_\cB\cM\op$ described in \eqref{eq:cA:cdot} and Corollary~\ref{cor:cA:involution} translate to
    \begin{align*}
        \mvisiblespace\cdot\mvisiblespace\colon\qquad&
        A(x ,yzH) \times A(y ,zH)
        \to 
        A(xy , zH),
        \quad
        (a, b)\mapsto a \alpha_{x }(b),
    \intertext{and}
        \mvisiblespace^*\colon\qquad&
        A(x ,yH)
        \to 
        A(  x\inv , xyH),
        \quad
        a\mapsto 
        \alpha_{  x\inv }(a^*).
    \end{align*}
    The Fell bundle~$\cA$ is therefore exactly the semi-direct product Fell bundle $X\ltimes (A\times X/H)$ which encodes the left-$X$ action
    $x\cdot (a,yH)=(\alpha_x(a), xyH)$ on the constant $\cst$-algebraic bundle $A\times X/H$ over $X/H$. We have therefore recovered the same Fell bundle that was mentioned as being equivalent to~$\cB$ in \cite[Example 14]{Muhly:Bdles}.
\end{example}

Example~\ref{ex:Paul's Ex 14} is the extreme case where $X$ is a group; the other extreme case is when $X$ is just a  \psp{$H$}: We can likewise define a free and proper action of the transformation group groupoid $X\rtimes H$ on $X$ by letting $(x,h)$ transform $x$ into $x\ract h$. The associated \igpd\ $X\times_{X\rtimes H}X\op$ is exactly the quotient $X/H$ of the original action on $X$, and so an application of \cite[Theorem 2.8]{MRW:1987:Equivalence} recovers Green's theorem that $C_0(X/H)$  and  $C_0(X)\rtimes_{r} H$ are \SMEadj.

There is a plethora of generalizations of Green's result. For example, we can allow the object $X$ that is being acted on to be not just a space (as in Green's theorem) or a group (as in Example~\ref{ex:Paul's Ex 14}), but a groupoid $\cX$, and we allow the underlying bundles to be more than line-bundles: if $\cM=(M\to \cX)$ is a {\em Fell bundle} that  carries a free and proper action of a group $K$, then~$\cM$ is an equivalence between the semi-direct product Fell bundle $\cB=\cM\rtimes K$ and the quotient Fell bundle $K\backslash\cM$.
\footnote{This is a special case of \cite[Theorem 3.1]{KMQW:ImprimThms1}, where they consider {\em two} group actions on a groupoid $\cX$. \cite[Theorem 6.1]{DuLi:ImprimThms-pp} goes two steps further, by firstly allowing groupoid actions, and secondly by allowing the actions to be only {\em self-similar} rather than by homomorphisms; this yields a Fell bundle equivalence between the Zappa--Sz\'ep product Fell bundles $(\cM/\mathcal{K}_{1})\bowtie \mathcal{K}_{2}$ and $\mathcal{K}_{1}\bowtie(\mathcal{K}_{2}\backslash\cM)$ as constructed in \cite{DuLi:ZS}.}  Theorem~\ref{thm:demiequiv to equiv} can recover such results, in that~$\cM$ is a $\cB$-\demiequiv[], and $\cM\otimes_{\cB}\cM\op$ can be computed to be isomorphic to the quotient bundle.

\section*{Acknowledgements}
I was supported by Methusalem grant METH/21/03 –- long term structural funding of the Flemish Government, and an FWO Senior Postdoctoral Fellowship (project number 1206124N). 

I would like to thank Bram Mesland for asking me a question that the main theorem answers, and I would further like to thank him and Arthur Pander Maat for multiple helpful discussions during the production of this paper. 

\appendix

\section{Hilbert $\cst$-Modules}\label{app:Hilbert modules}

\begin{lemma}\label{lem:norm of compacts}
	If $\mbf{X}$ is a full right Hilbert $C^*$-module over a $\cst$-algebra $B$ and $\mbf{x}_{i}\in\mbf{X}$, then the operator-norm of the $B$-compact positive operator $\sum_{i=1}^{k}\innercpct[]{}{\mbf{x}_{i}}{\mbf{x}_{i}}$ is exactly the norm of
	$
		\sum_{i=1}^{k}\rinner[\mathbf{X}]{B}{\mbf{x}_{i}}{\mbf{x}_{i}}
	$ in $B$.
\end{lemma}
\begin{proof}
	Denote the compact operator in question by $T$. For any $\mbf{z}\in\mbf{X}$, we have
	\begin{align*}
		\rinner[\mbf{X}]{B}{T\mbf{z}}{\mbf{z}}
		&=
		\sum\nolimits_{i} 
		\rinner[\mathbf{X}]{B}{
			\mbf{x}_{i} \cdot
			\rinner[\mbf{X}]  {B}
			{\mbf{x}_{i}}
			{\mbf{z}}
		}
		{\mbf{z}}
		=
		\sum\nolimits_{i} 
		\rinner[\mbf{X}]  {B}
		{\mbf{z}}{\mbf{x}_{i}}
		\rinner[\mathbf{X}]{B}{
			\mbf{x}_{i}
		}
		{\mbf{z}}.
	\end{align*}
	Note that the right-hand side is a sum of positive elements and hence positive, so $T$ is indeed a positive operator by \cite[Lemma 2.28]{RaWi:Morita}.
	By the Cauchy--Schwarz inequality for Hilbert modules \cite[Lemma 2.5]{RaWi:Morita}, we have for each $i$
	\[ 
	\rinner[\mbf{X}]  {B}
	{\mbf{z}}{\mbf{x}_{i}}
	\rinner[\mathbf{X}]{B}{
		\mbf{x}_{i}
	}
	{\mbf{z}}
	\leq
	\norm{\mbf{z}}^2
	\rinner[\mathbf{X}]{B}{
		\mbf{x}_{i}
	}
	{\mbf{x}_{i}}
	\]
	in the $\cst$-algebra $B$, and hence we conclude that
	\[
	\rinner[\mbf{X}]{B}{T\mbf{z}}{\mbf{z}}
	\leq
	\norm{\mbf{z}}^2
	\sum\nolimits_{i} 
	\rinner[\mathbf{X}]{B}{
		\mbf{x}_{i}
	}
	{\mbf{x}_{i}}
	.
	\]
	Using \cite[Remark~2.29]{RaWi:Morita} and positivity of $T$ in the first equality of the following, we deduce
	\begin{align*}
		\norm{T}
		&=\sup
		\{\norm{\rinner[\mbf{X}]{B}{T\mbf{z}}{\mbf{z}}}:\norm{\mbf{z}}\leq 1
		\}
		\leq
		\norm{\sum\nolimits_{i} 
			\rinner[\mathbf{X}]{B}{
				\mbf{x}_{i}
			}
			{\mbf{x}_{i}}}.
	\end{align*}
	For the reverse inequality, note that $\mbf{X}$ is an \ib\ between $\compacts_{B}(\mbf{X})$ and $B$, meaning that $B$ is exactly the algebra of $\compacts_{B}(\mbf{X})$-compact operators on the {\em left}-Hilbert $\cst$-module $\mbf{X}$. In particular, the same proof for $T$ replaced by 
	$\sum_{i}\rinner[\mathbf{X}]{B}{
		\mbf{x}_{i}
	}
	{\mbf{x}_{i}}$ yields the other estimate.
\end{proof}

The following result is well known.
\begin{lemma}\label{lem:tensor and compacts}
    Suppose~$\mathbf{X}$ is an $A-B$-\ib\ and $\mathbf{Y}$ is a $C-B$-\ib. Then the balanced tensor product $\mathbf{X}\otimes_{B} \mathbf{Y}\op$ is isomorphic to the $B$-compact operators $\compacts_{B}(\mathbf{Y},\mathbf{X})$ as $A-C$-\ib s.
\end{lemma}

\begin{proof}
    First recall that $\compacts_{B}(\mathbf{Y},\mathbf{X})$ is densely spanned by the maps
    \[
     \innercpct[]{}{\mathbf{x}}{\mathbf{y}}\colon\quad\mathbf{y}'\mapsto \mathbf{x} \cdot  \rip{B}<\mathbf{y},\mathbf{y}'>
    \]
    for $\mathbf{x}\in \mathbf{X},\mathbf{y}\in \mathbf{Y}$. Its bimodule structure is given by
    \[
        a\cdot  \innercpct[]{}{\mathbf{x}}{\mathbf{y}} \coloneqq \innercpct[]{}{a\cdot  \mathbf{x}}{\mathbf{y}}
        \quad\text{and}\quad
        \innercpct[]{}{\mathbf{x}}{\mathbf{y}} \cdot  c \coloneqq \innercpct[]{}{\mathbf{x}}{c^*\cdot  \mathbf{y}}.
    \]
    The map
    \[
    \mathbf{X}\times \mathbf{Y}\op
    \to \compacts_{B}(\mathbf{Y},\mathbf{X}),
    \quad
    (\mathbf{x},\mathbf{y}\op)
    \mapsto 
    \innercpct[]{}{\mathbf{x}}{\mathbf{y}},
    \]
    is bilinear, and so it descends to a linear map with domain $\mathbf{X}\odot \mathbf{Y}\op$. For any $b\in B$, we have
    \[
     (\mathbf{x}\cdot  b) \cdot  \rip{B}<\mathbf{y},\mathbf{y}'>
     =
     \mathbf{x} \cdot  \rip{B}<\mathbf{y}\cdot  b^*,\mathbf{y}'>,
     \text{ so that }
      \innercpct[]{}{\mathbf{x}\cdot  b}{\mathbf{y}}
      =
      \innercpct[]{}{\mathbf{x}}{\mathbf{y}\cdot  b^*}.
    \]
    Since $(\mathbf{y}\cdot  b^*)\op = b\cdot  \mathbf{y}\op$, we conclude that we have a linear map
    \[
    \mathbf{X}\odot_{B} \mathbf{Y}\op
    \to \compacts_{B}(\mathbf{Y},\mathbf{X})
    \text{ determined by }
    \mathbf{x}\odot \mathbf{y}\op
    \mapsto 
    \innercpct[]{}{\mathbf{x}}{\mathbf{y}}.
    \]
    Clearly, this map is an $A-C$-bimodule map and, 
    by definition of $\compacts_{B}(\mathbf{Y},\mathbf{X})$, it has dense range. Thus, if we can show that the map is isometric, then it extends to an isomorphism of Hilbert bimodules. For $\mathbf{x}_{i}\in \mathbf{X}$ and $\mathbf{y}_{i}, \mathbf{y}\in \mathbf{Y}$, we have
    \begin{align*}
        \norm{\sum_{i} \innercpct[]{}{\mathbf{x}_{i}}{\mathbf{y}_{i}}(\mathbf{y})}_{\mathbf{X}}^2
        &=
        \norm{\sum_{i,j} 
        \rip{B}<\mathbf{x}_{i}\cdot  {\rip{B}<\mathbf{y}_{i},\mathbf{y}>}, \mathbf{x}_{j}\cdot  {\rip{B}<\mathbf{y}_{j},\mathbf{y}>}>
        }_{B}
        =
        \norm{\sum_{i,j} 
        \rip{B}<\mathbf{y},\mathbf{y}_{i}>
        \rip{B}<\mathbf{x}_{i}, \mathbf{x}_{j}> \rip{B}<\mathbf{y}_{j},\mathbf{y}>
        }_{B}
    \intertext{As in \cite[Lemma 2.65]{RaWi:Morita}, we can write $\rinner[\mathbf{X}]{B}{\mathbf{x}_{j}}{\mathbf{x}_{i}}=\sum_{l} b_{j,l}b_{i,l}^*$ for some $b_{i,j}\in B$, so that}
        &=
        \norm{\sum_{i,j,l} 
        \rip{B}<\mathbf{y},\mathbf{y}_{i}>
        b_{i,l}b_{j,l}^* \rip{B}<\mathbf{y}_{j},\mathbf{y}>
        }_{B}
        =
        \norm{\sum_{i,j,l} 
        \rip{B}<\mathbf{y},\mathbf{y}_{i} \cdot  b_{i,l}>
        \rip{B}<\mathbf{y}_{j}\cdot  b_{j,l},\mathbf{y}>
        }_{B}.
    \end{align*}
    Write $\mathbf{z}_{l}\coloneqq \sum_i \mathbf{y}_{i} \cdot  b_{i,l}$. Since
    \[
        \rinner[\mathbf{Y}]{B}
        {\mathbf{y}}{\mathbf{z_{l}}}\rinner[\mathbf{Y}]{B}
        {\mathbf{z}_{l}}{\mathbf{y}}
        =
        \rinner[\mathbf{Y}]{B}
        {\mathbf{y}}{\mathbf{z}_{l}\cdot  \rinner[\mathbf{Y}]{B}
        {\mathbf{z}_{l}}{\mathbf{y}}}
        =
        \rinner[\mathbf{Y}]{B}
        {\mathbf{y}}{ \linner[\mathbf{Y}]{C}
        {\mathbf{z}_{l}}{\mathbf{z}_{l}}\cdot  \mathbf{y}},
    \]
    it follows that
    \begin{align*}
    \norm{\sum_{i} \innercpct[]{}{\mathbf{x}_{i}}{\mathbf{y}_{i}}}_{\compacts}^2
    =
    \sup_{\norm{\mathbf{y}}\leq 1}
        \norm{\sum_{i} \innercpct[]{}{\mathbf{x}_{i}}{\mathbf{y}_{i}}(\mathbf{y})}_{\mathbf{X}}^2
        &=
        \sup_{\norm{\mathbf{y}}\leq 1}
        \norm{ 
        \rinner[\mathbf{Y}]{B}
        {\mathbf{y}}{ \sum_{l}\linner[\mathbf{Y}]{C}
        {\mathbf{z}_{l}}{\mathbf{z}_{l}}\cdot  \mathbf{y}}
        }_{B}
        \overset{(\dagger)}{=}
        \norm{\sum_{l}\linner[\mathbf{Y}]{C}
        {\mathbf{z}_{l}}{\mathbf{z}_{l}}},
    \end{align*}
    where $(\dagger)$ follows from \cite[Remark 2.29]{RaWi:Morita} applied to the positive operator $\sum_{l}\linner[\mathbf{Y}]{C}
        {\mathbf{z}_{l}}{\mathbf{z}_{l}}$. 
    Since $B$ acts by $C$-adjointable operators on $\mathbf{Y}$, we have
    \begin{align*}               
        \norm{\sum_{l} \linner[\mathbf{Y}]{C}  {\mathbf{z}_{l}}{\mathbf{z}_{l}
            }
        }_{C}
        &=
        \norm{\sum_{i,j,l} \linner[\mathbf{Y}]{C}  {\mathbf{y}_{i}}{
            \mathbf{y}_{j}\cdot  (b_{j,l}b_{i,l}^*)
            }
        }_{C}
        =\norm{\sum_{i,j} \linner[\mathbf{Y}]{C}  {\mathbf{y}_{i}}{
            \mathbf{y}_{j}\cdot  \rinner[\mathbf{X}]{B}{\mathbf{x}_{j}}{\mathbf{x}_{i}}
            }
        }_{C}
        \\
        &=
        \norm{\sum_{i,j} \rinner[\mathbf{Y}\op]{C}  {\mathbf{y}_{i}\op}{
            \rinner[\mathbf{X}]{B}{\mathbf{x}_{i}}{\mathbf{x}_{j}}\cdot 
            \mathbf{y}_{j}\op
            }
        }_{C}
        \overset{\eqref{eq:norm on compacts}}{=}
        \norm{\sum_{i} \mathbf{x}_{i} \otimes \mathbf{y}_{i}\op}_{\mathbf{X}\otimes_{B}\mathbf{Y}\op}^2,
    \end{align*}
    as claimed.
\end{proof}

\section{\uscBb s}\label{app:uscBb}

The following lemma explains that, in an \uscBb, being arbitrarily close to convergent nets  implies convergence. It is, in essence, just a restatement of \cite[Proposition C.20]{Wil2007}, where we drop the assumption that the bundle be a $\cst$-bundle.
\begin{lemma}\label{lem:convergence uscBb:close nets}
    Let $\cM=(q_{\cM}\colon M\to X)$ be an \uscBb\ over a \LCH\ space $X$. Suppose we are given a net $(m_{\lambda})_{\lambda}$ in~$\cM$ and a point $m\in \cM$ such that $q_{\cM}(m_{\lambda})\to q_{\cM}(m)$. Then $m_{\lambda}\to m$ if and only if, for every $\epsilon>0$, there exists a convergent net $(n_{\lambda}^\epsilon)_{\lambda}$ in~$\cM$ with $q_{\cM}(n_{\lambda}^\epsilon)=q_{\cM}(m_{\lambda})$ such that for all large $\lambda$, we have
    \[
        \norm{m_{\lambda} - n_{\lambda}^\epsilon}<\epsilon
        \text{ and }
        \norm{m - \lim_{\lambda} n_{\lambda}^\epsilon}<\epsilon.
    \]
\end{lemma}

\begin{proof}
    The forwards implication is trivial. The backwards implication follows from an application of \cite[Lemma A.3, (i)$\implies$(ii)]{DL:MJM2023} (applied to the convergent net $(n_{\lambda}^\epsilon)_{\lambda}$), the triangle inequality, and an application of \cite[Lemma A.3, (i)$\impliedby$(ii)]{DL:MJM2023}.
\end{proof}

\begin{lemma}[cf.\ {\cite[Proposition A.7]{DL:MJM2023}}]\label{lem:uscBb:cts}
    Let $\cM=(q_\cM\colon M\to X)$ and $\cN=(q_{\cN}\colon N\to Y)$ be two {\uscBb s}. Suppose we have a commutative diagram
    \begin{equation}\label{diag:Omega and omega}
    \begin{tikzcd}[row sep=small]
        \cM \ar[d, "q_{\cM}"']\ar[r, "\Omega"]& \cN \ar[d, "q_{\cN}"]
        \\
        X \ar[r, "\omega"]& Y
      \end{tikzcd}  
    \end{equation}
    and that the maps
    satisfy the following conditions.
\begin{enumerate}[label=\textup{(\alph*)}]
    \item \label{item:ptws linear} for each $x\in X$, $\Omega|_{M_x}$ is linear;
    \item\label{item:constant} there exists a constant $K>0$ such that $\|\Omega(m)\|\leq K\|m\|$ for all $m\in M$;
    \item $\omega$ is a continuous map;  and
    \item\label{item:Omega circ f} there exists a collection $\Gamma$ of continuous sections of~$\cM$ such that the $\mathbb{C}$-linear span of
 $\{\gamma(x): \gamma\in\Gamma \}$ is dense in each~$M_x$ and, for each $\gamma\in\Gamma $, the section
    $\Omega\circ\gamma\colon x\mapsto \bigl(x,\Omega(\gamma(x))\bigr)$ of the pull-back bundle $\omega^*(\cN)$ is continuous.
\end{enumerate}
Then $\Omega$ is continuous. 
\end{lemma}

\begin{proof}
    First we point out that the proof of \cite[Proposition A.7]{DL:MJM2023} does not require $\Gamma$ to be a vector space but only that $\mathrm{span}_{\mbb{C}}\{\gamma(x):\gamma\in \Gamma\}$ is dense in $M(x)$; in other words, we already know that  Lemma~\ref{lem:uscBb:cts} holds in the case that $\omega$ is a homeomorphism.

    Consider the commutative diagram
    \begin{equation}\label{diag:Omega and Psi}
    \begin{tikzcd}[row sep=small, contains/.style = {draw=none,"\in" description,sloped}, contained/.style = {draw=none,"\ni" description,sloped}]
        m\ar[r,mapsto]\ar[d,contains] &\bigl(q_{\cM}(m),\Omega(m)\bigr)\ar[r,mapsto] \ar[d,contains]&\Omega(m)\ar[d,contains]
        \\
        \cM \ar[d, "q_{\cM}"']\ar[r, "\Psi"]& \omega^* (\cN) \ar[d, "\mathrm{pr}_{1}"]\ar[r]& \cN\ar[d, "q_{\cN}"]
        \\
        X \ar[r,equal]& X\ar[r, "\omega"] & Y
      \end{tikzcd}  
    \end{equation}
    We claim that the map $\Psi$ is continuous, which will imply that $\Omega$ is continuous since Diagram~\ref{diag:Omega and Psi} commutes. Since the identity map on~$X$ is a homeomorphism, we are in good shape to invoke \cite[Proposition A.7]{DL:MJM2023} for the square on the left-hand side; because of  Assumption \ref{item:Omega circ f}, it only remains to verify that $\Psi$ is fibrewise linear and bounded.
    But clearly, 
    Assumption \ref{item:ptws linear}
    implies linearity of $\Psi|_{M_{x}}\colon M_{x}\to \omega^*(\cN)_{x} = \{x\}\times \cN_{\omega(x)}$ by the definition of the linear structure on $\omega^*(\cN)$. And likewise, by definition of the norm on $\omega^*(\cN)_{x}$, we have $\norm{\Psi(m)}=\norm{\Omega(m)}$, and hence $\norm{\Psi(m)}\leq K\norm{m}$ for all $m\in M$ by \ref{item:constant}. Using \cite[Proposition A.7]{DL:MJM2023}, we conclude that $\Psi$ and hence $\Omega$ is continuous. 
\end{proof}

\begin{lemma}[cf.\ {\cite[Proposition A.8]{DL:MJM2023}}]\label{lem:uscBb:opem}
Let $\cM=(q_\cM\colon M\to X)$ and $\cN=(q_{\cN}\colon N\to Y)$ be two {\uscBb s}, and suppose we have the commutative Diagram~\eqref{diag:Omega and omega}
    and that the maps $\Omega,\omega $ 
    satisfy the following conditions.
\begin{enumerate}[label=\textup{(\alph*)}]
    \item\label{item:} for each $x\in X$, $\Omega|_{M_x}$ is linear;
    \item there exists a constant $k>0$ such that $\|\Omega(m)\| \geq  k\|m\|$ for all $m\in M$;
    \item\label{item:omega embedding} $\omega$ is an embedding (i.e., a homeomorphism onto its image); and
    \item $\Omega$ is continuous.
\end{enumerate}
Then $\Omega\inv \colon \Omega(M)\to M$ is continuous.
\end{lemma}

\begin{proof}
    As in the proof of Lemma~\ref{lem:uscBb:cts}, consider  the map $\Psi$ in Diagram~\eqref{diag:Omega and Psi}. Again, since $\Omega|_{M_x}$ is linear, so is $\Psi|_{M_{x}}$, and since $\Omega$ is bounded below, so is $\Psi$. This time, continuity of $\Omega$ implies continuity of $\Psi$. We can therefore apply \cite[Proposition A.8]{DL:MJM2023} to the square on the left-hand side of Diagram~\eqref{diag:Omega and Psi} to conclude that the map $\Psi\inv\colon \Psi(M)\to M$ is continuous.

    Since \(
    q_{\cN}(\Omega(m)) = \omega(q_{\cM}(m))
    \) by commutativity of Diagram~\eqref{diag:Omega and omega}, it follows from  Assumption~\ref{item:omega embedding} that the map \(f=q_{\cM}\circ\omega\inv\circ q_{\cN}\colon \Omega(M)\to X, \Omega(m)\mapsto q_{\cM}(m)\), is well defined and continuous.
    By definition of $\Psi$, we have 
    \[
        \bigl(f(\Omega(m)),\Omega(m)\bigr) = \bigl(q_{\cM}(m), \Omega(m)\bigr)
        =
        \Psi(m)
        ,
    \]
    so that for $n\in \Omega(M)$,
    \(
        \Psi\inv (f(n),n) =  \Omega\inv (n)
    \). It follows that $\Omega\inv$ is continuous as concatenation of continuous maps.
\end{proof}

\printbibliography

\end{document}